\documentclass[11pt,a4paper,reqno]{amsart}
\usepackage{fancyhdr}
\usepackage{ifpdf}
\usepackage{amsmath,amsfonts,amsthm,amsopn,amssymb}
\usepackage{verbatim,wasysym,mathrsfs}
\usepackage[left=2.6cm,right=2.6cm,top=3.1cm,bottom=3.1cm]{geometry}
\usepackage{cite,marginnote}
\usepackage{color,enumitem,graphicx}
\usepackage[colorlinks=true,urlcolor=blue, citecolor=red,linkcolor=blue,
linktocpage,pdfpagelabels, bookmarksnumbered,bookmarksopen]{hyperref}
\usepackage[english]{babel}
\usepackage[T1]{fontenc}
\usepackage{lmodern}
\usepackage[hyperpageref]{backref}

\newtheorem{Thm}{Theorem}[section]
\newtheorem{Lem}[Thm]{Lemma}

\newtheorem{Cor}[Thm]{Corollary}
\newtheorem{Prop}[Thm]{Proposition}

\theoremstyle{definition}

\newtheorem{Rem}{Remark}[section]

\newcommand{\N}{{\mathbb N}}

\newcommand{\R}{{\mathbb R}}
\newcommand{\weakto}{{\rightharpoonup}}

\numberwithin{equation}{section}

\makeindex

\makeatletter
\def\@makefnmark{}
\makeatother

\allowdisplaybreaks

\begin{document}

\title[Asymptotic behavior of the Choquard equation in dimension two]{Asymptotic behavior of the least energy solutions to the Choquard equation in dimension two}

\author[Jinkai Gao]{Jinkai Gao}
\address[Jinkai Gao]{\newline\indent School of Mathematical Sciences,
\newline\indent Nankai University,
\newline\indent Tianjin, 300071, PRC.}
    \email{\href{mailto:jinkaigao@mail.nankai.edu.cn}{jinkaigao@mail.nankai.edu.cn}}
    
\author[Xinfu Li]{Xinfu Li}
\address[Xinfu Li]{\newline\indent School of Science,
\newline\indent  Tianjin University of Commerce,
\newline\indent Tianjin, 300134, PRC.}
    \email{\href{mailto:lxylxf@tjcu.edu.cn}{lxylxf@tjcu.edu.cn}}
   
 \author[Shiwang Ma]{Shiwang Ma$^*$}
\address[Shiwang Ma]
	{\newline\indent
		School of Mathematical Sciences and LPMC,
			\newline\indent
			Nankai University,
			\newline\indent
			Tianjin 300071, PRC.}
	\email{\href{mailto:shiwangm@nankai.edu.cn}{shiwangm@nankai.edu.cn}}

 \thanks{Corresponding author:
 {\tt Shiwang Ma}.}


\subjclass[2020]{Primary 35J15; Secondary 35B09, 35B40, 35B44.}
\date{\today}
\keywords{Choquard equation; Least energy solutions; Asymptotic behavior; Green's function.}

\begin{abstract}
In this paper, we are interested in the following planar Choquard equation
\begin{equation*}
    \begin{cases}
        -\Delta u=\displaystyle\left(\int\limits_{\Omega}\frac{u^{p+1}(y)}{|x-y|^\alpha}dy\right)u^{p},\quad u>0,\ \  &\mbox{in}\ \Omega,\\
 \quad \ \ u=0, \ \  &\mbox{on}\ \partial \Omega,
    \end{cases}
\end{equation*}
where $\Omega$ is a smooth bounded domain in $\mathbb{R}^2$, $\alpha\in (0,2)$ and $p>1$ is a positive parameter. Unlike the higher-dimensional case, we prove that the least energy solutions $u_{p}$ neither blow up nor vanish, and develop only one peak as $p\to+\infty$ under suitable assumptions on $\Omega$. In contrast, the modified solutions $pu_p$ exhibit blow-up behavior analogous to that observed in higher dimensions. Furthermore, as $\alpha \to 0$, the main results of this paper become consistent with the known conclusions for the corresponding Lane-Emden equation.
\end{abstract}

\maketitle


\section{Introduction}
In this paper, we consider the following Choquard equation
\begin{equation}\label{slightly subcritical choquard equation}
    \begin{cases}
        -\Delta u=\displaystyle\left(\int\limits_{\Omega}\frac{u^{p+1}(y)}{|x-y|^\alpha}dy\right)u^{p},\quad u>0,\ \  &\mbox{in}\ \Omega,\\
 \quad \ \ u=0, \ \  &\mbox{on}\ \partial \Omega,
    \end{cases}
\end{equation}
where $N\geq 2$, $\Omega$ is a smooth bounded domain in $\mathbb{R}^{N}$, $\alpha\in (0,N)$ and $p\in(1,2^*_{\alpha}-1)$. Here, $2^*_{\alpha}$ denotes the upper critical exponent determined by the Sobolev embedding and the Hardy-Littlewood-Sobolev inequality, which is explicitly given by $2^*_{\alpha} := \frac{2N-\alpha}{N-2} \ \text{if}\ N \geq 3,\ \text{and}\ +\infty\ \text{if}\ N=2.$ The Choquard equation arises from multiple physical contexts, including polaron theory \cite{FrohlichHerbert,Pekar}, one-component plasma models \cite{Lieb-1976}, and bosonic many-body systems with attractive long-range interactions, particularly in Bose-Einstein condensates \cite{Frohlich-2003,Lewin-2014}. Apart from the physical motivations, the Choquard equation has been extensively studied from a mathematical perspective because of its nonlocal term. We refer to  Moroz-Van Schaftingen \cite{Moroz-2017-JFPTA} and references therein for a broad survey.

Since $p\in(1,2^*_{\alpha} - 1)$, the existence of the least energy solutions $u_p$ to \eqref{slightly subcritical choquard equation} follows from standard variational arguments. Therefore, it is natural to consider the asymptotic behavior of $u_{p}$ as $p$ approaches the critical exponent. Indeed, the study of such problems can be traced back to at least \cite{Atkinson-Peletier,Brezispletier1989,Han1991,Rey1990,Rey1989ProofOT,Ren1994TAMS,Ren1996PAMS} and have grown rapidly in recent years. As a full survey is beyond our scope, we focus only on results relevant to our work. To be precise, we introduce the Green's function and the Robin's function on $\Omega$, both of which play crucial roles in the subsequent analysis. The Green's function $G$ for the Laplacian $-\Delta$ on $\Omega$ is defined, for each fixed $y\in\Omega$, as the solution to the following Dirichlet boundary value problem
\begin{equation}\label{defin-green-function}
\begin{cases}
-\Delta_{x} G(x,y)= \delta_y, &{\text{in}~\Omega}, \\
\qquad G(x,y)=0, &{\text{on}~\partial \Omega},
\end{cases}
\end{equation}
where $\delta_y$ denotes the Dirac delta function centered at $y$. For the Green's function $G(x,y)$, we have the following decomposition
\begin{equation}\label{eq-green-function}
G(x,y)=S(x,y)-H(x,y), ~~(x,y)\in \Omega\times \Omega,
\end{equation}
where 
\begin{equation}\label{eq-singular-part-of-the-green-function}
S(x,y):=\begin{cases}
    -\frac{1}{2\pi}\log|x-y|,&\text{~~for~~}N=2\text{~and~}x,y\in\Omega,\\
    \frac{1}{(N-2)\omega_{N}|x-y|^{N-2}},&\text{~~for~~}N\geq 3\text{~and~}x,y\in\Omega,\\
\end{cases}
\end{equation}
is the singular part, which is also the fundamental solution to the Laplace equation in $\R^{N}$ and $\omega_{N}=\frac{2\pi^{N/2}}{\Gamma(N/2)}$ is the measure of the unit sphere in $\R^{N}$. $H(x,y)$
is the regular part of $G(x,y)$ satisfying 
\begin{equation*}
 \begin{cases}
-\Delta_{x} H(x,y)=0, &{\text{in}~\Omega}, \\
\qquad H(x,y)=S(x,y), &{\text{on}~\partial \Omega}.
\end{cases}   
\end{equation*}
Recall that $G$ and $H$ are symmetric in $x$ and $y$. Specifically, $H$ is a smooth function in $\Omega\times\Omega$ and we define its leading term as
\begin{equation}\label{definition of Robin function}
    R(x):=H(x,x), ~~x\in \Omega,
\end{equation}
called the Robin function of $\Omega$ at $x$. Moreover, by the comparison principle, we have
\begin{equation*}
    G(x,y)>0,~~~\forall~ x,y\in\Omega,
\end{equation*}
and the Robin function satisfies
\begin{equation*}
    R(x)\to+\infty \text{~as~}x\to\partial\Omega.
\end{equation*}
For further properties of the Green's function, we refer to the appendices in \cite{DallAcqua-2004-jde, DAprile2013CPDE, ACKERMANN20134168}.

Now, we recall some known results related to equation \eqref{slightly subcritical choquard equation}. For $N \geq 3$ and $p$ approaching $2^*_{\alpha}-1$ from below, Chen and Wang \cite{chen2024blowingupsolutionschoquardtype} proved that the least energy solutions must necessarily blow up and concentrate at critical points of the Robin function $R(x)$. Moreover, when $x_0$ is a nondegenerate critical point of $R(x)$, a family of solutions that blow up and concentrate at $x_0$ can be constructed via a reduction argument. In contrast, for the two-dimensional case $N=2$, the asymptotic behavior of the least energy solutions to \eqref{slightly subcritical choquard equation} as $p\to+\infty$ remains an open problem. In the special case $\alpha=0$, the equation \eqref{slightly subcritical choquard equation} reduces, after a suitable scaling, to the well-known Lane-Emden equation
\begin{equation}\label{Lane Emden problem}
    \begin{cases}
        -\Delta u = u^{p}, \quad u>0, & \text{in } \Omega, \\
        u = 0, & \text{on } \partial\Omega,
    \end{cases}
\end{equation}
which models the mechanical structure of self-gravitating spheres, particularly stellar structures in astrophysics. For further physical background, we refer to \cite{Chandrasekhar-1957,Horedt-1987}. Despite its deceptively simple form, the Lane-Emden equation exhibits remarkably rich and complex solution structures, attracting sustained research interest over recent decades.

In the case $N=2$,  Ren and Wei \cite{Ren1994TAMS,Ren1996PAMS} proved that the least energy solution $u_p$ to \eqref{Lane Emden problem} neither blows up nor vanishes, and satisfies the energy condition
\begin{equation}\label{energy-convergence}
    p\int_{\Omega}|\nabla u_p|^2 dx \to 8\pi e \quad \text{as } p\to+\infty.
\end{equation}
Moreover, the maximum point $x_{p}$ of $u_p$ converges to $x_0 \in \Omega$ as $ p\to+\infty$ with $x_0$ being a critical point of the Robin function $R(x)$, and the normalized solutions $v_p := u_p^p/\int_\Omega u_p^p$ satisfy
\begin{equation*}
    v_{p} \to G(x,x_0) \text{ in } C_{loc}^{2}(\overline{\Omega}\setminus\{x_0\}).
\end{equation*}
Subsequently, Adimurthi and Grossi \cite{Adimurthi2004PAMS} showed that, under an appropriate rescaling, $u_p$ converges to the standard bubble function 
\begin{equation*}
    U(x) := \log \frac{1}{(1+\frac{1}{8}|x|^2)^2},
\end{equation*}
which is the unique solution to the Liouville equation \cite{Chen-Li-1991}:
\begin{equation}
    -\Delta u = e^u \text{ in } \mathbb{R}^2,\text{~with~~}u(0)=0 \text{~~and~~} \int_{\mathbb{R}^2} e^u dx = 8\pi.
\end{equation}
They also proved that $u_p(x_p) \to \sqrt{e}$ as $p\to+\infty$, confirming a conjecture previously made by Chen, Zhou, and Ni \cite{Chen-Ni-Zhou-2000}.

For general solutions to \eqref{Lane Emden problem} (not necessarily the least energy solutions), Kamburov and Sirakov \cite{Kamburov-Sirakov-2018} established the following uniform estimates: Given any $p_0 > 1$, there exists a constant $C > 0$ depending only on $p_0$ and $\Omega$ such that for all $p \geq p_0$, any solution $u_p$ to \eqref{Lane Emden problem} satisfies 
\begin{equation*}
    \|u_p\|_{L^\infty(\Omega)} \leq C.
\end{equation*}
 Moreover, if $\Omega$ is strictly star-shaped, then there exists another constant $C > 0$ (depending on $p_0$ and $\Omega$) such that
\begin{equation}\label{energy bounded assumption}
    p\int_\Omega |\nabla u_p|^2 dx \leq C.
\end{equation}
On the other hand, De Marchis et al.\cite{De-Grossi-Ianni-Pacella-2017,De-Grossi-Ianni-Pacella-2018} and Thizy \cite{Thizy-2019} established a complete description of the asymptotic behavior for general solutions to \eqref{Lane Emden problem}. To be more precise, they proved that for any family of solutions $u_{p}$ to \eqref{Lane Emden problem} satisfying the energy condition \eqref{energy bounded assumption}, there exists an integer $k\geq 1$ such that 
\begin{equation*}
  p\int_\Omega |\nabla u_p|^2 dx \to k8\pi e\text{~~as~~}p\to+\infty,
\end{equation*}
and $u_p$ concentrates at exactly $k$ distinct points, behaving like the standard bubble $U(x)$ near each point after rescaling. Conversely, in non-simply-connected domains, Esposito, Musso, and Pistoia \cite{Esposito-Musso-Pistoia-2006} 
established the existence of solutions to \eqref{Lane Emden problem} that concentrate at $k$ distinct points via a Lyapunov-Schmidt reduction scheme. For further results, we refer to \cite{Grossi-Ianni-Luo-Yan-2022,Battaglia-Ianni-Pistoia-2025,Chen-Cheng-Zhao-2024,Pistoia-Ricciardi-2024} and the references therein.

Motivated by the previous work, we naturally consider the asymptotic behavior of the least energy solutions to \eqref{slightly subcritical choquard equation} as $p\to+\infty$ in the case $N=2$ and $\alpha\in(0,2)$. First, to find a least energy solution to \eqref{slightly subcritical choquard equation}, it's natural to consider the following constrained minimizing problem
\begin{equation*}
    S_{p}^{2}:=\inf\left\{\int_{\Omega}|\nabla u|^{2}dx:~u\in H^{1}_{0}(\Omega),~\left(\int_{\Omega}\int_{\Omega}\frac{|u|^{p+1}(y)|u|^{p+1}(x)}{|x-y|^{\alpha}}dydx\right)^{\frac{1}{p+1}}=1\right\}.
\end{equation*}
Since the Sobolev embedding $H^{1}_{0}(\Omega)\hookrightarrow L^{q}(\Omega)$ is compact for any $q\in[1,+\infty)$, a standard variational argument gives that $S_{p}^{2}$ can be achieved by a positive function $u_{p}^{\prime}$.  Moreover, by the Lagrange multiplier theorem, we conclude that $u_{p}:=S_{p}^{1/p}u_{p}^{\prime}$ is a solution to \eqref{slightly subcritical choquard equation}. Next, we denote by $E_{p}$ the energy functional associated with \eqref{slightly subcritical choquard equation}
\begin{equation*}
    E_{p}(u):=\frac{1}{2}\int_{\Omega}|\nabla u|^{2}dx-\frac{1}{2(p+1)}\int_{\Omega}\int_{\Omega}\frac{u^{p+1}(y)u^{p+1}(x)}{|x-y|^{\alpha}}dydx,\quad u\in H^{1}_{0}(\Omega).
\end{equation*}
Notice that for any solution $u$ of \eqref{slightly subcritical choquard equation}, we have
\begin{equation*}
    E_{p}(u)=\left(\frac{1}{2}-\frac{1}{2(p+1)}\right)\int_{\Omega}|\nabla u|^{2}dx=\left(\frac{1}{2}-\frac{1}{2(p+1)}\right)\int_{\Omega}\int_{\Omega}\frac{u^{p+1}(y)u^{p+1}(x)}{|x-y|^{\alpha}}dydx.
\end{equation*}
Therefore, it is easy to verify that $u_{p}$ is a least energy solution to \eqref{slightly subcritical choquard equation} and
\begin{equation}\label{S-p}
    S_{p}=\frac{\left(\int_{\Omega}|\nabla u_{p}|^{2}dx\right)^{\frac{1}{2}}}{\left(\int_{\Omega}\int_{\Omega}\frac{u_{p}^{p+1}(y)u_{p}^{p+1}(x)}{|x-y|^{\alpha}}dydx\right)^{\frac{1}{2(p+1)}}}=\left(\int_{\Omega}|\nabla u_{p}|^{2}dx\right)^{\frac{p}{2(p+1)}}=\left(\frac{2(p+1)}{p}E_{p}(u_{p} )\right)^{\frac{p}{2(p+1)}}.
\end{equation}

In the following, we shall study the asymptotic behavior of the least energy solution $u_{p}$ as $p\to+\infty$. Before stating the main results, we first introduce the required assumptions on the domain $\Omega$.
\smallskip

\textbf{($H_{1}$)}: Let $R_{\Omega}:=\sup\{R:B_{R}(x)\subset\Omega\text{~for some~}x\in\Omega\}$, then 
\begin{equation}
    R_{\Omega}\geq\left(\frac{2(4-\alpha)\pi }{\tilde{C}_{\alpha}}\right)^{\frac{1}{4-\alpha}}\text{~with~}\tilde{C}_{\alpha}:=\int_{B_{1}(0)}\int_{B_{1}(0)}\frac{1}{|x-y|^{\alpha}}dydx.
\end{equation}
\smallskip

\textbf{($H_{2}$)}: There exists a point $y$ such that $\langle x-y,\nu(x)\rangle>0$ for any $x\in\partial\Omega$ and
\begin{equation}
    \int_{\partial\Omega}\frac{1}{\langle x-y,\nu(x)\rangle}d\sigma_{x}<2\pi e,
\end{equation}
\qquad\qquad where $\nu(x)$ denotes the unit outer normal of $\partial\Omega$ at $x$.

\begin{Rem}
   For sufficiently large  $R>0$, the ball $\Omega=B_{R}(0)$ centered at the origin satisfies assumptions $(H_{{1}})$ and $(H_{2})$ with $y=0$.
\end{Rem}
Our first result is as follows.
\begin{Thm}\label{thm-1}
Assume $\alpha\in(0,2)$ and let $\Omega\subset\R^{2}$ be a smooth bounded domain satisfying $(H_{1})$. Let $u_{p}$ be a family of the least energy solutions to \eqref{slightly subcritical choquard equation}. Then 
\begin{enumerate}[label=\upshape(\arabic*)]
    \item The least energy solutions $u_p$ neither vanish nor blow up
    \begin{equation}\label{eq-l-infty-bounded}
        1\leq \liminf_{p\to+\infty}\|u_{p}\|_{L^{\infty}(\Omega)}\leq \limsup_{p\to+\infty}\|u_{p}\|_{L^{\infty}(\Omega)}\leq \sqrt{e}.
    \end{equation}
    \item We have the following energy estimate
    \begin{equation}\label{eq-thm-energy}
    \begin{aligned}
        \lim_{p\to+\infty}p S_{p}^{2}&=\lim_{p\to+\infty}2p E_{p}(u_{p})= \lim_{p\to+\infty}p\int_{\Omega}|\nabla u_{p}|^{2}dx\\
        &=\lim_{p\to+\infty} p\int_{\Omega}\int_{\Omega}\frac{u_{p}^{p+1}(y)u_{p}^{p+1}(x)}{|x-y|^{\alpha}}dydx= 2(4-\alpha)\pi e.
    \end{aligned}        
    \end{equation}
    \item Under suitable rescaling, $u_p$ converges to the standard bubble. Let $x_{p}\in\Omega$ be the maximum point satisfying $u_{p}(x_{p})=\|u_{p}\|_{L^{\infty}(\Omega)}$. Define
    \begin{equation}\label{defin-vare-p}
        \varepsilon_{p}:=\left(pu_{p}^{2p}(x_{p})\right)^{-\frac{1}{4-\alpha}}
    \end{equation}
    and the rescaled function 
\begin{equation}
    v_{p}(x):=\frac{p}{u_{p}(x_{p})}\left(u_{p}(\varepsilon_{p}x+x_{p})-u_{p}(x_{p})\right),\text{~for any~}x\in\Omega_{p}:=\frac{\Omega-x_{p}}{\varepsilon_{p}}.
\end{equation}
Then $\varepsilon_{p}\to 0$ and $v_{p}\to v$ in $C^{2}_{loc}(\R^{2})$ as $p\to+\infty$, where 
\begin{equation}\label{definition of v}
    v(x):=\frac{4-\alpha}{2}\log \left(\frac{1}{1+C_{\alpha}^{-2}|x|^{2}}\right)\text{~with~}C_{\alpha}:=\left(\frac{(2-\alpha)(4-\alpha)}{\pi}\right)^{\frac{1}{4-\alpha}}
\end{equation}
is a solution of
\begin{equation}
    -\Delta v=\left(\int_{\R^{2}}\frac{e^{v(y)}}{|x-y|^{\alpha}}dy\right)e^{v(x)}\text{~~in~}\R^{2} \text{~with~}v(0)=0.
\end{equation}
In addition,
\begin{equation}
    \int_{\R^{2}}e^{\frac{4}{4-\alpha}v(x)}dx=\pi C^{2}_{\alpha}\text{~~and~~}  \int_{\R^{2}}\int_{\R^{2}}\frac{e^{v(y)}e^{v(x)}}{|x-y|^{\alpha}}dydx=2(4-\alpha)\pi.
\end{equation} 
\item There exist constants $C_{1},C_{2}>0$ such that for any $p$ large enough 
\begin{equation}\label{eq-thm-1.1-4}
    \frac{C_{1}}{p}\leq \int_{\Omega}\int_{\Omega}\frac{u_{p}^{p+1}(y)u_{p}^{p}(x)}{|x-y|^{\alpha}}dydx\leq \frac{C_{2}}{p}.
\end{equation}
\item It holds that $\sqrt{p}u_{p}\weakto 0$ in $H^{1}_{0}(\Omega)$ as $p\to+\infty$.
\end{enumerate}
\end{Thm}

Our second result provides a more refined description of the asymptotic profile of the least energy solution $u_{p}$. Before going on, we define the blow-up set of $p{u}_{p}$ as
\begin{equation}
    \mathcal{S}:=\left\{y\in\bar{\Omega}:\text{~there exist~}\{y_{p}\}\subset\Omega\text{~such that~}p{u}_{p}(y_{p})\to+\infty\text{~and~}y_{p}\to y\text{~as~}p\to+\infty\right\}.
\end{equation}
Let $x_{p}$ be the maximum point of $u_{p}$ converging to $x_{0}\in \bar{\Omega}$ as $p\to+\infty$, so that $x_{0}\in\mathcal{S}$. 
\begin{Thm}\label{thm-2}
Assume $\alpha\in(0,1)$ and let $\Omega\subset\R^{2}$ be a smooth bounded domain satisfying $(H_{1})$ and $(H_{2})$. For the least energy solutions $u_{p}$ to \eqref{slightly subcritical choquard equation}, we have $\mathcal{S}\cap\partial\Omega=\emptyset$ and $\mathcal{S}=\{x_{0}\}$. Moreover, the following properties hold 
   \begin{enumerate}[label=\upshape(\arabic*)]
    \item  The maximum value of $u_{p}$ tends to $\sqrt{e}$, that is
    \begin{equation}
        \lim_{p\to+\infty}u_{p}(x_{p})=\sqrt{e}.
    \end{equation}  
    \item The shape of $pu_{p}$ away from the blow-up point $x_{0}$ is given by 
    \begin{equation}
        \lim_{p\to+\infty}pu_{p}(x)=2(4-\alpha)\pi\sqrt{e}~G(x,x_{0})\text{~~in~~}C^{2}_{loc}({\Omega}\setminus\{x_{0}\}).
    \end{equation}
    \item The right-hand side of the equation satisfied by $pu_{p}$ tends to the Dirac delta function, that is
    \begin{equation}
       p\left(\int_{\Omega}\frac{u_{p}^{p+1}(y)}{|x-y|^{\alpha}}dy\right)u_{p}^{p}(x)\to (2(4-\alpha)\pi\sqrt{e})\delta_{x_{0}},
    \end{equation}
    in the sense of distribution, where $\delta_{x_{0}}$ is the Dirac delta function at point $x_{0}$.
    \item The blow-up point $x_0$ of $p u_p$ is a critical point of the Robin function, that is,
\begin{equation}
\nabla R(x_0) = 0.
\end{equation}
In particular, if $\Omega$ is a convex domain, then $x_0$ is the global minimum point of the Robin function.
\end{enumerate}
\end{Thm}
\begin{Rem}
    \begin{enumerate}
        \item Our results demonstrate that the solutions to equation \eqref{slightly subcritical choquard equation} exhibit qualitatively different behavior in two dimensions compared to higher dimensions as $p$ approaches the critical exponent. More precisely, the least energy solution $u_{p}$ neither blows up nor vanishes, and its profile resembles a single peak.
        \item As $\alpha\to0$, the  equation \eqref{slightly subcritical choquard equation} is formally reduced to the local Lane-Emden equation \eqref{Lane Emden problem}, and our results are consistent with those for this limiting Lane-Emden case. For $\alpha\neq0$, the appearance of the convolution term creates nontrivial difficulties, and the symmetry properties of double integrals, combined with the application of the Hardy-Littlewood-Sobolev inequality, play a crucial role in the analysis.
    \end{enumerate}
\end{Rem}

The rest of the paper is organized as follows. In Section~\ref{Preliminaries}, we recall some useful lemmas and establish several local Pohozaev identities. Section~\ref{section-prooftheorem} is devoted to the proof of Theorem~\ref{thm-1}. Finally, in Section~\ref{section-Proof of Theorem 2}, we prove that the blow-up set of $pu_{p}$ is disjoint from the boundary and consists of only one point. Based on this, we establish a refined decay estimate, from which Theorem~\ref{thm-2} follows.

\smallskip

\noindent\textbf{Notation.}
Throughout this paper, we use the following notations.
\begin{enumerate}
    \item We use $\|u\|_{H^{1}_{0}(\Omega)}=\left(\int_{\Omega}|\nabla u|^{2}dx\right)^{1/2}$ to denote the norm in $H^{1}_{0}(\Omega)$ and $\langle\cdot,\cdot\rangle$ means the corresponding inner product.  
    \item We use $C$ to denote various positive constant and use $C_{1}=o(\varepsilon)$ and $C_{2}=O(\varepsilon)$ to denote $C_{1}/\varepsilon\to0$ and $|C_{2}/\varepsilon|\leq C $ as $\varepsilon\to0$ respectively.
     \item  Let $f,g: X\to \R^{+}\cup\{0\}$ be two nonnegative function defined on some set $X$. we write $f\lesssim g$ or $g\gtrsim f$, if there exists a constant $C>0$ independent of $x$ such that $f(x)\leq C g(x)$ for any $x\in X$ and $f\sim g$  means that $f\lesssim g$ and $g\lesssim f$.
\end{enumerate}

\section{Preliminaries}\label{Preliminaries}

In this section, we recall some known results required for our analysis. We start with the Hardy-Littlewood-Sobolev (HLS) inequality, a key tool for estimating nonlocal terms.
\begin{Lem}\label{lema 2.1} \rm{\cite{Lieb2001}}
   Suppose $N\geq1$, $\alpha\in(0,N)$ and $\theta,\,r>1$ with $\frac{1}{\theta}+\frac{1}{r}+\frac{\alpha}{N}=2$. Let $f\in L^{\theta}(\R^N)$ and $g\in L^{r}(\R^N)$. Then, there exists a sharp constant $C(\theta,r,\alpha,N)$, independent of $f$ and $g$, such that
\begin{align}\label{HLS}
\displaystyle{\int_{\R^N}}\displaystyle{\int_{\R^N}}\frac{f(x)g(y)}{|x-y|^{\alpha}}dxdy\leq C(\theta,r,\alpha,N)\|f\|_{L^{\theta}(\R^N)}\|g\|_{L^{r}(\R^N)}.
\end{align}
If $\theta=r=\frac{2N}{2N-\alpha}$, then
\begin{equation}\label{definition of C N alpha}
    C(\theta,r,\alpha,N)=C_{N,\alpha}:=\pi^{\frac{\alpha}{2}}\frac{\Gamma\left(\frac{N-\alpha}{2}\right)}{\Gamma\left(N-\frac{\alpha}{2}\right)}\left(\frac{\Gamma(N)}{\Gamma\left(\frac{N}{2}\right)}\right)^{\frac{N-\alpha}{N}}.
\end{equation}
In this case, the equality in \eqref{HLS} holds if and only if $f\equiv (const.)\, g$, where
$$g(x)=A\left(\frac{1}{\gamma^{2}+|x-a|^{2}}\right)^{\frac{2N-\alpha}{2}},\quad \text{for some $A\in \mathbb{C}$, $0\neq\gamma\in\R$ and $a\in\R^N$.}$$ 
\end{Lem}

In the two-dimensional case, the Moser-Trudinger inequality  becomes particularly significant.
\begin{Lem}\label{trudinger inequality}\rm{\cite{Moser,Trudinger-1967}}
There exists an absolute constant $C>0$, independent of any parameters, such that for any $u\in H^{1}_{0}(\Omega)$
    \begin{equation}
        \int_{\Omega}e^{4\pi\left(\frac{u}{\|\nabla u\|_{L^{2}}}\right)^{2}}dx\leq C|\Omega|,
    \end{equation}
where $\Omega$ is a bounded domain in $\R^{2}$ and $|\Omega|$ is the Lebesgue measure of $\Omega$.
\end{Lem}

We further introduce two auxiliary lemmas related to the Moser-Trudinger inequality.
\begin{Lem} \label{lem ren2.1}\rm{\cite{Ren1994TAMS}}
For every $t\geq 2$, there is $D_{t}>0$ such that for any $u\in H^{1}_{0}(\Omega)$
\begin{equation}
   \|u\|_{L^{t}(\Omega)}\leq D_{t}t^{1/2}\|\nabla u\|_{L^{2}(\Omega)}, 
\end{equation}
where $\Omega$ is a bounded domain in $\R^{2}$. Furthermore,
\begin{equation}
 \lim_{t\to+\infty}D_{t}=(8\pi e)^{-1/2}.
\end{equation}
\end{Lem}
\begin{Lem}\rm{\cite{Brezis-Merle-1991}} \label{lem ren4.3}
Let $u$ be a solution of 
   \begin{equation}
    \begin{cases}
        -\Delta u=f,&\quad\text{in~}\Omega,\\
        \quad\ \ u=0,&\quad\text{on~}\partial\Omega,
    \end{cases}        
    \end{equation}
where $\Omega$ is a smooth bounded domain in $\R^{2}$. Then for any $0<\varepsilon<4\pi$ we have 
\begin{equation}
    \int_{\Omega}e^{\frac{(4\pi-\varepsilon)|u|(x)}{\|f\|_{L^{1}(\Omega)}}}dx\leq\frac{4\pi^{2}}{\varepsilon}(\text{diam~}\Omega)^{2}.
\end{equation}
\end{Lem}

When $N=2$, the limit equation of \eqref{slightly subcritical choquard equation} as $p\to+\infty$ becomes the following planar Choquard equation with an exponential nonlinearity
\begin{equation}\label{limit-equation}
    -\Delta u=\left(\int_{\R^{2}}\frac{e^{u(y)}}{|x-y|^{\alpha}}dy\right)e^{u(x)},\ \ \text{~in~}\R^{2}.
\end{equation}
The classification of the solutions to \eqref{limit-equation} is as follows. 
\smallskip

\noindent\textbf{Theorem B.}\cite{Guo2024JGA,Yang,Niu2025,Gluck2025dcds} \emph{Suppose $\alpha\in(0,2)$ and $u\in L^{1}_{loc}(\R^{2})$ is a distributional solution to the equation \eqref{limit-equation} satisfying
\begin{equation}
    \int_{\R^{2}}e^{\frac{4}{4-\alpha}u(x)}dx<+\infty.
\end{equation}
Then $u\in C^{\infty}(\R^{2})$ and must have the following form
\begin{equation}\label{defin-U-xi-lambda}
    u(x)=U_{\mu,\xi}(x):=\frac{4-\alpha}{2}\log\left( \frac{C_{\alpha}\mu}{1+\mu^{2}|x-\xi|^{2}}\right),
\end{equation}
where $C_{\alpha}:=\left(\frac{(2-\alpha)(4-\alpha)}{\pi}\right)^{\frac{1}{4-\alpha}}$ is a positive constant, $\mu>0$ and $\xi\in\R^{2}$ are two parameters. Moreover, we have
\begin{equation}
    \int_{\R^{2}}e^{\frac{4}{4-\alpha}U_{\mu,\xi}(x)}dx=\pi C^{2}_{\alpha}\text{~and~}\int_{\R^{2}}\int_{\R^{2}}\frac{e^{U_{\mu,\xi}(y)}e^{U_{\mu,\xi}(x)}}{|x-y|^{\alpha}}dydx=2(4-\alpha)\pi.
\end{equation}
}
In particular
\begin{equation}
    \int_{\R^{2}}\frac{e^{U_{\mu,\xi}(y)}}{|x-y|^{\alpha}}dy=\frac{2(4-\alpha)}{C_{\alpha}^{2}}e^{\frac{\alpha}{4-\alpha}U_{\mu,\xi}(x)}.
\end{equation}

Finally, we establish various local Poho\v zaev identities, which are useful to study the local properties of solutions to \eqref{slightly subcritical choquard equation}. 
\begin{Lem}
Suppose that $u_{p}$ is a classic solution of equation~\eqref{slightly subcritical choquard equation}.
Then, for any bounded domain $\Omega^{\prime}\subset\Omega$, the following identities hold
\begin{equation}\label{Pohozaev 1}
    \begin{split}
        &\frac{1}{2}\int_{\partial \Omega^{\prime}}\big\langle x-z,\nu\big\rangle|\nabla u_{p}|^2d\sigma_{x}-\int_{\partial\Omega^{\prime}}\frac{\partial u_{p}}{\partial\nu}\big\langle x-z,\nabla u_{p}\big\rangle d\sigma_{x}\\
&=-\frac{4-\alpha}{2(p+1)}\int_{\Omega^{\prime}}\int_{\Omega^{\prime}}\frac{u_{p}^{p+1}(y)u_{p}^{p+1}(x)}
{|x-y|^{\alpha}}dydx\\
&\quad-\frac{2}{p+1}\int_{\Omega^{\prime}}\int_{\Omega\setminus\Omega^{\prime}}\frac{u_{p}^{p+1}(y)u_{p}^{p+1}(x)}
{|x-y|^{\alpha}}dydx\\
&\quad+\frac{\alpha}{p+1}\int_{\Omega^{\prime}}\int_{\Omega\setminus\Omega^{\prime}}\langle x-z,x-y\rangle\frac{u_{p}^{p+1}(y)u_{p}^{p+1}(x)}
{|x-y|^{\alpha+2}}dy dx\\
&\quad+\frac{1}{p+1}\int_{\partial \Omega^{\prime}}\int_{ \Omega}\frac{u_{p}^{p+1}(y)u_{p}^{p+1}(x)}
{|x-y|^{\alpha}}\big\langle x-z,\nu\big\rangle dy d\sigma_{x}\\
    \end{split}
\end{equation}
and
\begin{equation}\label{Pohozaev 2}
    \begin{split}
        -&\int_{\partial \Omega^{\prime}}\frac{\partial u_{p}}{\partial x_j}\frac{\partial u_{p}}{\partial \nu}d\sigma_{x}+\frac{1}{2}\int_{\partial \Omega^{\prime}}|\nabla u_{p}|^{2}v_{j}d\sigma_{x}\\
        &=\frac{\alpha}{p+1}\int_{\Omega^{\prime}}\int_{\Omega\setminus \Omega^{\prime}}(x_j-y_j)\frac{u_{p}^{p+1}(y)u_{p}^{p+1}(x)}
{|x-y|^{\alpha+2}}dy dx\\
&\quad+\frac{1}{p+1}\int_{\partial \Omega^{\prime}}\int_{ \Omega}\frac{u_{p}^{p+1}(y)u_{p}^{p+1}(x)}{|x-y|^{\alpha}}\nu_jdy d\sigma_{x},
    \end{split}
\end{equation}
where $j=1,2$, $z\in\R^{2}$ and $\nu=\nu(x)$ denotes the unit outward normal to the boundary $\partial\Omega^{\prime}$. Moreover, when $\Omega^{\prime}=\Omega$, we have
\begin{equation}\label{Pohozaev 3}
    \begin{split}
\int_{\partial \Omega}\big\langle x-z,\nu\big\rangle\left(\frac{\partial u_{p}}{\partial\nu}\right)^2d\sigma_{x}
=\frac{(4-\alpha)}{p+1}\int_{\Omega}\int_{\Omega}\frac{u_{p}^{p+1}(y)u_{p}^{p+1}(x)}
{|x-y|^{\alpha}}dydx.
    \end{split}
\end{equation}
\end{Lem}
\begin{proof}
Without loss of generality, we may assume that $z=0$. First, we multiply both sides of equation \eqref{slightly subcritical choquard equation} by $\langle x,\nabla u_{p}\rangle$ and integrate on $\Omega^{\prime}$, 
\begin{equation}\label{Pohozaev 1 proof 1}
-\int_{\Omega^{\prime}}\Delta u_{p}\langle x,\nabla u_{p}\rangle dx
=\int_{\Omega^{\prime}}\langle x,\nabla u_{p}(x)\rangle\Big(\int_{\Omega}\frac{u_{p}^{p+1}(y)}
{|x-y|^{\alpha}}dy\Big)u_{p}^{p}(x)dx.
\end{equation}
Notice that
\begin{equation}\label{Pohozaev 1 proof 2}
\begin{split}
\int_{\Omega^{\prime}}&\big\langle x,\nabla u_{p}(x)\big\rangle\Big(\int_{\Omega}\frac{u_{p}^{p+1}(y)}
{|x-y|^{\alpha}}dy\Big)u_{p}^{p}(x)dx\\&
=\int_{\Omega^{\prime}}\big\langle x,\nabla u_{p}(x)\big\rangle\Big(\int_{\Omega^{\prime}}\frac{u_{p}^{p+1}(y)}
{|x-y|^{\alpha}}dy\Big)u_{p}^{p}(x)dx\\
&\quad+\int_{\Omega^{\prime}}\big\langle x,\nabla u_{p}(x)\big\rangle\Big(\int_{\Omega\setminus\Omega^{\prime}}\frac{u_{p}^{p+1}(y)}
{|x-y|^{\alpha}}dy\Big)u_{p}^{p}(x)dx.
\end{split}
\end{equation}
We calculate the first term on the right-hand side of \eqref{Pohozaev 1 proof 2}
\begin{equation}\label{Pohozaev 1 proof 3}
\begin{split}
\int_{\Omega^{\prime}}&\big\langle x,\nabla u_{p}(x)\big\rangle\Big(\int_{\Omega^{\prime}}\frac{u_{p}^{p+1}(y)}
{|x-y|^{\alpha}}dy\Big)u_{p}^{p}(x)dx\\&
=-\frac{2}{p+1}\int_{\Omega^{\prime}}\int_{\Omega^{\prime}}\frac{u_{p}^{p+1}(y)u_{p}^{p+1}(x)}
{|x-y|^{\alpha}}dydx
+\frac{1}{p+1}\int_{\partial \Omega^{\prime}}\int_{ \Omega^{\prime}}\frac{u_{p}^{p+1}(y)u_{p}^{p+1}(x)}
{|x-y|^{\alpha}}\big\langle x,\nu\big\rangle dy d\sigma_{x}\\
&\quad+\frac{\alpha}{p+1}\int_{\Omega^{\prime}}\int_{\Omega^{\prime}}\langle x,x-y\rangle\frac{u_{p}^{p+1}(y)u_{p}^{p+1}(x)}
{|x-y|^{\alpha+2}}dy dx.\\
\end{split}
\end{equation}
Similarly, we can deduce
\begin{equation}\label{Pohozaev 1 proof 4}
\begin{split}
\int_{\Omega^{\prime}}&\big\langle y,\nabla u_{p}(y)\big\rangle\Big(\int_{\Omega^{\prime}}\frac{u_{p}^{p+1}(x)}
{|x-y|^{\alpha}}dy\Big)u_{p}^{p}(y)dy\\&
=-\frac{2}{p+1}\int_{\Omega^{\prime}}\int_{\Omega^{\prime}}\frac{u_{p}^{p+1}(x)u_{p}^{p+1}(y)}
{|x-y|^{\alpha}}dxdy
+\frac{1}{p+1}\int_{\partial \Omega^{\prime}}\int_{ \Omega^{\prime}}\frac{u_{p}^{p+1}(x)u_{p}^{p+1}(y)}
{|x-y|^{\alpha}}\big\langle y,\nu\big\rangle dx d\sigma_{y}\\
&\quad+\frac{\alpha}{p+1}\int_{\Omega^{\prime}}\int_{\Omega^{\prime}}\langle y,y-x\rangle\frac{u_{p}^{p+1}(x)u_{p}^{p+1}(y)}
{|x-y|^{\alpha+2}}dx dy.\\
\end{split}
\end{equation}
Combining \eqref{Pohozaev 1 proof 3} and \eqref{Pohozaev 1 proof 4} together, then we can deduce that
\begin{equation}\label{Pohozaev 1 proof 5}
\begin{split}
&\int_{\Omega^{\prime}}\big\langle x,\nabla u_{p}(x)\big\rangle\Big(\int_{\Omega^{\prime}}\frac{u_{p}^{p+1}(y)}
{|x-y|^{\alpha}}dy\Big)u_{p}^{p}(x)dx\\&
=\frac{\alpha-4}{2(p+1)}\int_{\Omega^{\prime}}\int_{\Omega^{\prime}}\frac{u_{p}^{p+1}(y)u_{p}^{p+1}(x)}
{|x-y|^{\alpha}}dydx+\frac{1}{p+1}\int_{\partial \Omega^{\prime}}\int_{ \Omega^{\prime}}\frac{u_{p}^{p+1}(y)u_{p}^{p+1}(x)}
{|x-y|^{\alpha}}\big\langle x,\nu\big\rangle dy d\sigma_{x}.
\end{split}
\end{equation}
For the second term on the right-hand side of \eqref{Pohozaev 1 proof 2}, we have
\begin{equation}
\begin{split}
\int_{\Omega^{\prime}}&\big\langle x,\nabla u_{p}(x)\big\rangle\Big(\int_{\Omega\setminus\Omega^{\prime}}\frac{u_{p}^{p+1}(y)}
{|x-y|^{\alpha}}dy\Big)u_{p}^{p}(x)dx\\&
=-\frac{2}{p+1}\int_{\Omega^{\prime}}\int_{\Omega\setminus\Omega^{\prime}}\frac{u_{p}^{p+1}(y)u_{p}^{p+1}(x)}
{|x-y|^{\alpha}}dydx\\
&\quad+\frac{1}{p+1}\int_{\partial \Omega^{\prime}}\int_{ \Omega\setminus\Omega^{\prime}}\frac{u_{p}^{p+1}(y)u_{p}^{p+1}(x)}
{|x-y|^{\alpha}}\big\langle x,\nu\big\rangle dy d\sigma_{x}\\
&\quad+\frac{\alpha}{p+1}\int_{\Omega^{\prime}}\int_{\Omega\setminus\Omega^{\prime}}\langle x,x-y\rangle\frac{u_{p}^{p+1}(y)u_{p}^{p+1}(x)}
{|x-y|^{\alpha+2}}dy dx.
\end{split}
\end{equation}
On the other hand, we have
\begin{equation}\label{Pohozaev 1 proof 9}
-\int_{\Omega^{\prime}}\Delta u_{p}\big\langle x,\nabla u_{p}\big\rangle dx=\frac{1}{2}\int_{\partial \Omega^{\prime}}\big\langle x,\nu\big\rangle|\nabla u_{p}|^2d\sigma_{x}-\int_{\partial\Omega^{\prime}}\frac{\partial u_{p}}{\partial\nu}\big\langle x,\nabla u_{p}\big\rangle d\sigma_{x},
\end{equation}
Now combining \eqref{Pohozaev 1 proof 1}-\eqref{Pohozaev 1 proof 9} together, we can obtain \eqref{Pohozaev 1}.  

To prove \eqref{Pohozaev 2}, we multiply both sides of \eqref{slightly subcritical choquard equation} by $\frac{\partial u_{p}}{\partial x_j}$ and integrate on $\Omega^{\prime}$, 
\begin{equation}\label{Pohozaev 2 proof 1}
-\int_{\Omega^{\prime}}\Delta u_{p}\frac{\partial u_{p}}{\partial x_j}dx
=\int_{\Omega^{\prime}}\frac{\partial u_{p}}{\partial x_j}\Big(\int_{\Omega}\frac{u_{p}^{p+1}(y)}
{|x-y|^{\alpha}}dy\Big)u_{p}^{p}(x)dx.
\end{equation}
The right-hand side of \eqref{Pohozaev 2 proof 1} can be estimated as follows
\begin{equation}\label{Pohozaev 2 proof 2}
\begin{split}
    \int_{\Omega^{\prime}}\frac{\partial u_{p}}{\partial x_j}\Big(\int_{\Omega}\frac{u_{p}^{p+1}(y)}
{|x-y|^{\alpha}}dy\Big)u_{p}^{p}(x)dx&=\int_{\Omega^{\prime}}\frac{\partial u_{p}}{\partial x_j}\Big(\int_{\Omega^{\prime}}\frac{u_{p}^{p+1}(y)}
{|x-y|^{\alpha}}dy\Big)u_{p}^{p}(x)dx\\
&\quad+\int_{\Omega^{\prime}}\frac{\partial u_{p}}{\partial x_j}\Big(\int_{\Omega\setminus \Omega^{\prime}}\frac{u_{p}^{p+1}(y)}
{|x-y|^{\alpha}}dy\Big)u_{p}^{p}(x)dx.
\end{split} 
\end{equation}
We calculate the first term on the right-hand side of \eqref{Pohozaev 2 proof 2}
\begin{equation}
\begin{split}
\int_{\Omega^{\prime}}&\frac{\partial u_{p}}{\partial x_j}\Big(\int_{\Omega^{\prime}}\frac{u_{p}^{p+1}(y)}
{|x-y|^{\alpha}}dy\Big)u_{p}^{p}(x)dx\\
&=\frac{1}{p+1}\int_{\partial \Omega^{\prime}}\int_{ \Omega^{\prime}}\frac{u_{p}^{p+1}(y)u_{p}^{p+1}(x)}
{|x-y|^{\alpha}}\nu_jdy d\sigma_{x}\\
&\quad+\frac{\alpha}{p+1}\int_{\Omega^{\prime}}\int_{\Omega^{\prime}}(x_j-y_j)\frac{u_{p}^{p+1}(y)u_{p}^{p+1}(x)}
{|x-y|^{\alpha+2}}dy dx\\
&=\frac{1}{p+1}\int_{\partial \Omega^{\prime}}\int_{ \Omega^{\prime}}\frac{u_{p}^{p+1}(y)u_{p}^{p+1}(x)}
{|x-y|^{\alpha}}\nu_jdy d\sigma_{x}.
\end{split}
\end{equation}
For the second term on the right-hand side of \eqref{Pohozaev 2 proof 2}
\begin{equation}
\begin{split}
\int_{\Omega^{\prime}}&\frac{\partial u_{p}}{\partial x_j}\Big(\int_{\Omega\setminus \Omega^{\prime}}\frac{u_{p}^{p+1}(y)}
{|x-y|^{\alpha}}dy\Big)u_{p}^{p}(x)dx\\
&=\frac{1}{p+1}\int_{\partial \Omega^{\prime}}\int_{ \Omega\setminus \Omega^{\prime}}\frac{u_{p}^{p+1}(y)u_{p}^{p+1}(x)}
{|x-y|^{\alpha}}\nu_jdy d\sigma_{x}\\
&\quad+\frac{\alpha}{p+1}\int_{\Omega^{\prime}}\int_{\Omega\setminus \Omega^{\prime}}(x_j-y_j)\frac{u_{p}^{p+1}(y)u_{p}^{p+1}(x)}
{|x-y|^{\alpha+2}}dy dx.
\end{split}
\end{equation}
On the other hand, the left-hand side of \eqref{Pohozaev 2 proof 1}
\begin{equation}\label{Pohozaev 2 proof 5}
    \begin{split}
        -\int_{\Omega^{\prime}}\Delta u_{p}\frac{\partial u_{p}}{\partial x_j}dx=-\int_{\partial \Omega^{\prime}}\frac{\partial u_{p}}{\partial x_j}\frac{\partial u}{\partial \nu}d\sigma_{x}+\frac{1}{2}\int_{\partial \Omega^{\prime}}|\nabla u_{p}|^{2}v_{j}d\sigma_{x}.
    \end{split}
\end{equation}
Now, combining \eqref{Pohozaev 2 proof 1}-\eqref{Pohozaev 2 proof 5} together, we can obtain \eqref{Pohozaev 2}. 

Finally, notice that $u_{p}=0$ on $\partial\Omega$, thus $\nabla u_{p}=\pm|\nabla u_{p}|\nu$ and we can deduce \eqref{Pohozaev 3} from \eqref{Pohozaev 1}. This completes the proof.
\end{proof}


\section{Proof of Theorem \ref{thm-1}}\label{section-prooftheorem}
In this section, we consider the asymptotic behavior of the least energy solution $u_{p}$ as $p\to+\infty$. First, we establish a refined estimate for the minimizing value $S_{p}$.
\begin{Prop}\label{estimate of S-p}
Let $\alpha\in(0,2)$ and $\Omega\subset\R^{2}$ be a smooth bounded domain. It holds that
\begin{equation}\label{prop-estimate-of-S-p}
    \lim_{p\to+\infty}p^{1/2}S_{p}=(2(4-\alpha)\pi e)^{1/2}.
\end{equation}
\end{Prop}
\begin{proof}
   By the HLS inequality and Lemma \ref{lem ren2.1}, we have
    \begin{equation}
    \begin{aligned}
         \left(\int_{\Omega}\int_{\Omega}\frac{u_{p}^{p+1}(y)u_{p}^{p+1}(x)}{|x-y|^{\alpha}}dydx\right)^{\frac{1}{2(p+1)}}&\leq C_{2,\alpha}^{\frac{1}{2(p+1)}}\|u_{p}\|_{L^{\frac{4(p+1)}{4-\alpha}}(\Omega)}\\
         &\leq C_{2,\alpha}^{\frac{1}{2(p+1)}} D_{\frac{4(p+1)}{4-\alpha}}\left(\frac{4(p+1)}{4-\alpha}\right)^{1/2}\|\nabla u_{p}\|_{L^{2}(\Omega)},
    \end{aligned}  
    \end{equation}
    where $C_{2,\alpha}$ is a constant defined in \eqref{definition of C N alpha} and $D_{t}$ is a constant for any $t\geq 2$ with $\lim_{t\to+\infty}D_{t}=(8\pi e)^{-1/2}$.
    Then by \eqref{S-p}
    \begin{equation}
        \begin{aligned}
            S_{p}=\frac{\left(\int_{\Omega}|\nabla u_{p}|^{2}dx\right)^{\frac{1}{2}}}{\left(\int_{\Omega}\int_{\Omega}\frac{u_{p}^{p+1}(y)u_{p}^{p+1}(x)}{|x-y|^{\alpha}}dydx\right)^{\frac{1}{2(p+1)}}}\geq C_{2,\alpha}^{-\frac{1}{2(p+1)}} D^{-1}_{\frac{4(p+1)}{4-\alpha}}\left(\frac{4(p+1)}{4-\alpha}\right)^{-1/2}.
        \end{aligned}
    \end{equation}
Hence
\begin{equation}\label{prop-estimate-of-S-p-proof-1}
    \liminf_{p\to+\infty}p^{1/2}S_{p}\geq (2(4-\alpha)\pi e)^{1/2}.
\end{equation}

Next, without loss of generality, we assume that $0\in\Omega$ and let $L>0$ be such that $B_{L}(0)\subset\Omega$. For any $0<l<L$, we consider the following Moser's function 
    \begin{equation}
        m_{l}(x)=\frac{1}{\sqrt{2\pi}}\begin{cases}
            (\log(L/l))^{1/2},&~0\leq |x|\leq l,\\
            \frac{\log(L/|x|)}{(\log(L/l))^{1/2}},&~l\leq |x|\leq L,\\
            0,&~L\leq |x|,\\
        \end{cases}
    \end{equation}
Then $m_{l}\in H^{1}_{0}(\Omega)$ and $\|\nabla m_{l}\|_{L^{2}(\Omega)}=1$. Moreover, we have
\begin{equation}
    \begin{aligned}
        \left(\int_{\Omega}\int_{\Omega}\frac{m_{l}^{p+1}(y)m_{l}^{p+1}(x)}{|x-y|^{\alpha}}dydx\right)^{\frac{1}{2(p+1)}}\geq \frac{1}{\sqrt{2\pi}}(\log(L/l))^{\frac{1}{2}}l^{\frac{4-\alpha}{2(p+1)}}\tilde{C}_{\alpha}^{\frac{1}{2(p+1)}},
    \end{aligned}
\end{equation}
where 
\begin{equation}\label{definition of C-alpha}
    \tilde{C}_{\alpha}:=\int_{B(0,1)}\int_{B(0,1)}\frac{1}{|x-y|^{\alpha}}dydx<+\infty.
\end{equation}
Choosing $l=Le^{-\frac{p+1}{4-\alpha}}$, thus from the definition of $S_{p}$ and the estimates of $m_{l}$ above, we have
\begin{equation}\label{eq upper estimate}
\begin{aligned}
    p^{1/2}S_{p}&\leq p^{1/2}\frac{\left(\int_{\Omega}|\nabla m_{l}|^{2}dx\right)^{\frac{1}{2}}}{\left(\int_{\Omega}\int_{\Omega}\frac{m_{l}^{p+1}(y)m_{l}^{p+1}(x)}{|x-y|^{\alpha}}dydx\right)^{\frac{1}{2(p+1)}}}\\
    &\leq (2(4-\alpha)\pi e)^{1/2}\left(\frac{p}{p+1}\right)^{1/2}\tilde{C}_{\alpha}^{-\frac{1}{2(p+1)}}L^{-\frac{4-\alpha}{2(p+1)}}.
\end{aligned} 
\end{equation}
Hence
\begin{equation}\label{prop-estimate-of-S-p-proof-2}
   \limsup_{p\to+\infty}p^{1/2}S_{p}\leq (2(4-\alpha)\pi e)^{1/2}. 
\end{equation}
and \eqref{prop-estimate-of-S-p} follows from \eqref{prop-estimate-of-S-p-proof-1} and \eqref{prop-estimate-of-S-p-proof-2}.
\end{proof}

\begin{Cor}\label{cor energy}Let $\alpha\in(0,2)$ and $\Omega\subset\R^{2}$ be a smooth bounded domain. Then it holds that
    \begin{equation}\label{cor-1}
       \lim_{p\to+\infty}2p E_{p}(u_{p})= \lim_{p\to+\infty}p\int_{\Omega}|\nabla u_{p}|^{2}dx=\lim_{p\to+\infty} p\int_{\Omega}\int_{\Omega}\frac{u_{p}^{p+1}(y)u_{p}^{p+1}(x)}{|x-y|^{\alpha}}dxdy= 2(4-\alpha)\pi e.
    \end{equation}
\end{Cor}
\begin{proof}
Equation \eqref{cor-1} follows directly from \eqref{slightly subcritical choquard equation}, \eqref{S-p}, and Proposition \ref{estimate of S-p}.
\end{proof}

By Corollary \ref{cor energy}, we observe that  
\begin{equation}\label{eq energy to 0}  
    E_{p}(u_{p}) \to 0 \text{~~and~~} \int_{\Omega}|\nabla u_{p}|^{2}dx \to 0\text{~~as~~} p \to +\infty,  
\end{equation}  
Thus $u_{p} \to 0$ a.e. as $p \to +\infty$. However, the following lemma shows that $u_p$ does not vanish.  

\begin{Prop}\label{lower bound of L-infty norm}
   Let $\alpha\in(0,2)$ and $\Omega\subset\R^{2}$ be a smooth bounded domain. Then it holds that
   \begin{equation}
       \liminf_{p\to+\infty}\|u_{p}\|_{L^{\infty}(\Omega)}\geq 1\text{~~and~~}\lim_{p\to+\infty}p\|u_{p}\|_{L^{\infty}(\Omega)}^{2p}=+\infty.
   \end{equation}
\end{Prop}
     
\begin{proof}
Let $\lambda_{1}(\Omega)$ be the first eigenvalue of $-\Delta$ with homogeneous Dirichlet boundary condition. Then from equation \eqref{slightly subcritical choquard equation}, the HLS inequality, the H\"older inequality and the Poincar\'{e} inequality, we deduce that
\begin{equation}\label{lower bound of L-infty norm-proof-1}
    \begin{aligned}
        1=\frac{\int_{\Omega}\int_{\Omega}\frac{u_{p}^{p+1}(y)u_{p}^{p+1}(x)}{|x-y|^{\alpha}}dydx}{\int_{\Omega}|\nabla u_{p}|^{2}dx}&\leq \frac{\|u_{p}\|_{L^{\infty}(\Omega)}^{2p}C_{2,\alpha}\|u_{p}\|_{L^{\frac{4}{4-\alpha}}(\Omega)}^{2}}{\int_{\Omega}|\nabla u_{p}|^{2}dx}\\
        &\leq \frac{\|u_{p}\|_{L^{\infty}(\Omega)}^{2p}C_{2,\alpha}|\Omega|^{\frac{2-\alpha}{2}}\|u_{p}\|_{L^{2}(\Omega)}^{2}}{\int_{\Omega}|\nabla u_{p}|^{2}dx}\\
        &\leq\|u_{p}\|_{L^{\infty}(\Omega)}^{2p}C_{2,\alpha}|\Omega|^{\frac{2-\alpha}{2}}\lambda_{1}^{-1}(\Omega),
    \end{aligned}
\end{equation}
where $C_{2,\alpha}$ is a constant defined in \eqref{definition of C N alpha}. Then $\|u_{p}\|_{L^{\infty}(\Omega)}\geq \left(\frac{\lambda_{1}(\Omega)}{C_{2,\alpha}|\Omega|^{\frac{2-\alpha}{2}}}\right)^{\frac{1}{2p}}\to 1$ as $p\to+\infty$. 
\end{proof}

In contrast to the higher-dimensional case \cite{chen2024blowingupsolutionschoquardtype}, we show that \( u_{p} \) does not blow up.
\begin{Lem}\label{upper bound of L-infty norm}
Let $\alpha\in(0,2)$ and $\Omega\subset\R^{2}$ be a smooth bounded domain. Then there exists a constant $C>0$ independent of $p$ such that $\|u_{p}\|_{L^{\infty}(\Omega)}\leq C$ for all sufficiently large $p$.
\end{Lem}
\begin{proof}
Let
    \begin{equation}
        \gamma_{p}:=\max_{x\in\bar{\Omega}}u_{p}(x),\quad\mathcal{A}:=\{x:\gamma_{p}/2<u_{p}(x)\},\quad \Omega_{t}:=\{x:t<u_{p}(x)\}.
    \end{equation}
From equation \eqref{slightly subcritical choquard equation} and the Coarea formula \cite{Federer1969}, we have
\begin{equation}
    \int_{\Omega_{t}}\int_{\Omega}\frac{u_{p}^{p+1}(y)u_{p}^{p}(x)}{|x-y|^{\alpha}}dydx=-\int_{\Omega_{t}}\Delta u_{p}dx=\int_{\partial\Omega_{t}}|\nabla u_{p}|d\sigma_{x},\quad -\frac{d}{dt}|\Omega_{t}|=\int_{\partial\Omega_{t}}\frac{d\sigma_{x}}{|\nabla u_{p}|}.
\end{equation}
Then by the Schwarz inequality and the isoperimetric inequality, we have
\begin{equation}
    -\frac{d}{dt}|\Omega_{t}| \int_{\Omega_{t}}\int_{\Omega}\frac{u_{p}^{p+1}(y)u_{p}^{p}(x)}{|x-y|^{\alpha}}dydx=\int_{\partial\Omega_{t}}\frac{d\sigma_{x}}{|\nabla u_{p}|}\int_{\partial\Omega_{t}}|\nabla u_{p}|d\sigma_{x}\geq |\partial\Omega_{t}|^{2}\geq 4\pi|\Omega_{t}|.
\end{equation}
Now, we define a function $r(t)$ such that $|\Omega_{t}|=\pi r^{2}(t)$. 
Notice that for any fixed $q_{1},q_{2}>1$ with $\frac{1}{q_{1}}+\frac{1}{q_{2}}=1$ and $\alpha q_{1}<2$, by the H\"older inequality, Lemma \ref{lem ren2.1} and Corollary \ref{cor energy}, we obtain that
\begin{equation}
\begin{aligned}
    \int_{\Omega}\frac{u_{p}^{p+1}(y)}{|x-y|^{\alpha}}dy&\leq \left(\int_{\Omega}\frac{1}{|x-y|^{\alpha q_{1}}}dy\right)^{\frac{1}{q_{1}}}\left(\int_{\Omega}u_{p}^{(p+1 )q_{2}}dy\right)^{\frac{1}{q_{2}}}\\
    &\leq C (D_{(p+1)q_{2}}((p+1)q_{2})^{1/2}\|\nabla u_{p}\|_{L^{2}(\Omega)})^{p+1}\\
    &\leq C q_{2}^{\frac{p+1}{2}},
\end{aligned}
\end{equation}
for any $p$ large enough. Thus there exist a constant $C>0$ independent of $p$ such that 
\begin{equation}
    -\frac{dt}{dr}\leq \frac{1}{2\pi r}\int_{\Omega_{t}}\int_{\Omega}\frac{u_{p}^{p+1}(y)u_{p}^{p}(x)}{|x-y|^{\alpha}}dydx\leq \frac{1}{2}Cq_{2}^{\frac{p+1}{2}}r\gamma_{p}^{p}.
\end{equation}
Integrating the above inequality from $0$ to $r_{0}$ and choosing $r_{0}$ such that $t(r_{0})=\gamma_{p}/2$, we get
\begin{equation}\label{upper bound of L-infty norm-proof-6}
    \gamma_{p}\leq \frac{1}{2}Cq_{2}^{\frac{p+1}{2}}r_{0}^{2}\gamma_{p}^{p}=\frac{1}{2\pi}|\Omega_{\gamma_{p}/2}|Cq_{2}^{\frac{p+1}{2}}\gamma_{p}^{p}=\frac{1}{2\pi}|\mathcal{A}|Cq_{2}^{\frac{p+1}{2}}\gamma_{p}^{p}.
\end{equation}
On the other hand, from Lemma \ref{lem ren2.1} and Corollary \ref{cor energy}, we have
\begin{equation}
    \left(\int_{\Omega}u_{p}^{2p}dx\right)^{1/(2p)}\leq D_{2p}  (2p)^{1/2}\left(\int_{\Omega}|\nabla u_{p}|^{2}dx\right)^{1/2}\to \left(\frac{4-\alpha}{2}\right)^{\frac{1}{2}},\text{~~as~~}p\to+\infty.
\end{equation}
Thus there exist a constant $M>0$ independent of $p$ such that 
\begin{equation}\label{upper bound of L-infty norm-proof-8}
    \left(\frac{\gamma_{p}}{2}\right)^{2p}|\mathcal{A}|\leq \int_{\Omega}u_{p}^{2p}dx\leq M^{2p}.
\end{equation}
Combining \eqref{upper bound of L-infty norm-proof-6} and \eqref{upper bound of L-infty norm-proof-8}, we get
\begin{equation}
    \gamma_{p}^{p+1}\leq \frac{1}{2\pi}Cq_{2}^{\frac{p+1}{2}}(2M)^{2p},
\end{equation}
Then we can conclude that $\gamma_{p}\leq C$, where $C$ independent on $p$.
\end{proof}

Let $x_{p}\in\Omega$ be the point such that $u_{p}(x_{p})=\|u_{p}\|_{L^{\infty}(\Omega)}$. We define the scaling parameter $\varepsilon_{p}^{4-\alpha}:=\frac{1}{pu_{p}^{2p}(x_{p})}$ and consider the rescaled function:
\begin{equation}\label{definition-of-z-p}
    v_{p}(x):=\frac{p}{u_{p}(x_{p})}\left(u_{p}(\varepsilon_{p}x+x_{p})-u_{p}(x_{p})\right)\text{~for any~}x\in\Omega_{p}:=\frac{\Omega-x_{p}}{\varepsilon_{p}}.
\end{equation}
Directly from the definitions of $v_{p}$ and $\varepsilon_{p}$, we have
\begin{equation}\label{eq propertity of z-p}
   \varepsilon_{p}\to0,\quad v_{p}(0)=\max_{x\in\Omega_{p}}v_{p}(x)=0,\quad0<1+\frac{v_{p}}{p}\leq 1
\end{equation}
and 
$v_{p}(x)$ satisfy
\begin{equation}\label{eq for z-p}
    \begin{cases}
        -\Delta v_{p}=\left(\int_{\Omega_{p}}\frac{\left(1+\frac{v_{p}(y)}{p}\right)^{p+1}}{|x-y|^{\alpha}}dy\right)\left(1+\frac{v_{p}(x)}{p}\right)^{p},&\text{~~in~}\Omega_{p},\\
        \quad\ \ v_{p}=-p,&\text{~~on~}\partial\Omega_{p}.\\
    \end{cases}
\end{equation}

\begin{Lem}\label{lemma uniform estimate} Let $\alpha\in(0,2)$ and $\Omega\subset\R^{2}$ be a smooth bounded domain satisfying $(H_{1})$. Then we have
 \begin{equation}
   \int_{\Omega_{p}}\frac{\left( 1+\frac{v_{p}(y)}{p}\right)^{p+1}}{|x-y|^{\alpha}}dy\lesssim 1\text{~~and~~}\int_{\Omega_{p}\setminus B(x,R)}\frac{\left( 1+\frac{v_{p}(y)}{p}\right)^{p+1}}{|x-y|^{\alpha}}dy=o(1),
\end{equation}   
for any $x\in\Omega_{p}$ and $p,R$ large enough.
\end{Lem}
\begin{proof}
First, by $0< 1+\frac{v_{p}}{p}\leq 1$ and the H\"older inequality, we have
 \begin{equation}\label{proof of lemma 3.6-1}
     \begin{aligned}
         \int_{\Omega_{p}}&\frac{\left(1+\frac{v_{p}(y)}{p}\right)^{p+1}}{|x-y|^{\alpha}}dy\\
         &=\int_{\Omega_{p}\cap B(x,1)}\frac{\left(1+\frac{v_{p}(y)}{p}\right)^{p+1}}{|x-y|^{\alpha}}dy+\int_{\Omega_{p}\setminus B(x,1)}\frac{\left(1+\frac{v_{p}(y)}{p}\right)^{p+1}}{|x-y|^{\alpha}}dy\\
          &\leq C +\left(\int_{\R^{2}\setminus B(x,1)}\frac{1}{|x-y|^{4}}dy\right)^{\frac{\alpha}{4}}\left(\int_{\Omega_{p}}\left(1+\frac{v_{p}(y)}{p}\right)^{\frac{4p}{4-\alpha}}dy\right)^{\frac{4-\alpha}{4}}\\
          &\leq C+C\left(\int_{\Omega_{p}}\left(1+\frac{v_{p}(y)}{p}\right)^{\frac{4p}{4-\alpha}}dy\right)^{\frac{4-\alpha}{4}},
     \end{aligned}
 \end{equation}   
for any $p$ large enough. 
Since $\log x\leq \frac{x}{e}$ for any $x>0$ 
\begin{equation}\label{proof of lemma 3.6-2}
 \begin{aligned}
      \int_{\Omega_{p}}\left(1+\frac{v_{p}(y)}{p}\right)^{\frac{4p}{4-\alpha}}dy
      =\int_{\Omega}(p^{\frac{1}{p}}u_{p}^{2})^{\frac{2p}{4-\alpha}}(y)dy=\int_{\Omega}e^{\frac{2p}{4-\alpha}\log(p^{\frac{1}{p}}u_{p}^{2}(y))}dy
      \leq \int_{\Omega}e^{c_{p}4\pi\left(\frac{u_{p}}{\|\nabla u_{p}\|_{L^{2}}}\right)^{2}}dy,
 \end{aligned}    
 \end{equation}
where $c_{p}:=\frac{p^{\frac{1+p}{p}}\|\nabla u_{p}\|_{L^{2}}^{2}}{2(4-\alpha)\pi e}=\frac{p^{\frac{1+p}{p}}S_{p}^{\frac{2(p+1)}{p}}}{2(4-\alpha)\pi e}$. Moreover, using \eqref{eq upper estimate} with $L\geq\left(\frac{2(4-\alpha)\pi }{\tilde{C}_{\alpha}}\right)^{\frac{1}{4-\alpha}}$, we have
 \begin{equation}
     c_{p}\leq (2(4-\alpha)\pi e)^{\frac{1}{p}}\left(\frac{p}{p+1}\right)^{\frac{p+1}{p}}\tilde{C}_{\alpha}^{-\frac{1}{p}}L^{-\frac{4-\alpha}{p}}\leq 1,
 \end{equation}
for any $p$ large enough. Thus by \eqref{proof of lemma 3.6-2} and Lemma \ref{trudinger inequality}
\begin{equation}\label{estimate integral}
    \int_{\Omega_{p}}\left(1+\frac{v_{p}(y)}{p}\right)^{\frac{4p}{4-\alpha}}dy\leq \int_{\Omega}e^{4\pi\left(\frac{u_{p}}{\|\nabla u_{p}\|_{L^{2}}}\right)^{2}}dy\leq C|\Omega|,
\end{equation}
for any $p$ large enough. Combining \eqref{proof of lemma 3.6-1} and \eqref{estimate integral} together, we get
\begin{equation}
\begin{aligned}
    \int_{\Omega_{p}}\frac{\left( 1+\frac{v_{p}(y)}{p}\right)^{p+1}}{|x-y|^{\alpha}}dy\leq C.
\end{aligned}  
\end{equation}
On the other hand, by the H\"older inequality
\begin{equation}
    \begin{aligned}
        &\int_{\Omega_{p}\setminus B(x,R)}\frac{\left( 1+\frac{v_{p}(y)}{p}\right)^{p+1}}{|x-y|^{\alpha}}dy\\
         &\lesssim\left(\int_{ \R^{2}\setminus B(x,R)}\frac{1}{|x-y|^{4}}dy\right)^{\frac{\alpha}{4}}\left(\int_{\Omega_{p}}\left( 1+\frac{v_{p}(y)}{p}\right)^{\frac{4p}{4-\alpha}}dy\right)^{\frac{4-\alpha}{4}}\\\
         &=o(1),
    \end{aligned}
\end{equation}
for any $p$ and $R$ large enough. This completes the proof.

\end{proof}

\begin{Prop}\label{prop convergence of z-p}
 Let $\alpha\in(0,2)$ and $\Omega\subset\R^{2}$ be a smooth bounded domain satisfying $(H_{1})$. After passing to a subsequence, we have $v_{p}\to v$ in $C^{2}_{loc}(\R^{2})$ and
    \begin{equation}
        v(x)=\frac{4-\alpha}{2}\log \left(\frac{1}{1+C_{\alpha}^{-2}|x|^{2}}\right),
    \end{equation}
where $C_{\alpha}:=\left(\frac{(2-\alpha)(4-\alpha)}{\pi}\right)^{\frac{1}{4-\alpha}}$ is a positive constant. Moreover, $v$ satisfies 
\begin{equation}
    -\Delta v=\left(\int_{\R^{2}}\frac{e^{v(y)}}{|x-y|^{\alpha}}dy\right)e^{v(x)}\textrm{~in~}\R^{2},
\end{equation}
with
\begin{equation}\label{eq-integral}
    \int_{\R^{2}}e^{\frac{4}{4-\alpha}v(x)}dx=\pi C^{2}_{\alpha}\text{~and~} \int_{\R^{2}}\int_{\R^{2}}\frac{e^{v(y)}e^{v(x)}}{|x-y|^{\alpha}}dydx=2(4-\alpha)\pi.
\end{equation} 
\end{Prop}

\begin{proof}
First, we prove that $\Omega_{p}\to\R^{2}$ as $p\to+\infty$. Indeed, since $\varepsilon_{p}\to0$ as $p\to+\infty$, either $\Omega_{p}\to\mathbb{R}^{2}$ or $\Omega_{p}\to\mathbb{R}\times(-\infty,R)$ (up to a rotation) as
$p\to+\infty$ for some $R\geq0$ . In the second case, we let
$v_{p}:=\varphi_{p}+\psi_{p}$ in $\Omega_{p}\cap B_{2R+1}(0)$ with $-\Delta\varphi_{p}=-\Delta v_{p}$ in $\Omega_{p}\cap B_{2R+1}(0)$ and $\psi_{p}=v_{p}$ in $\partial\left(\Omega_{p}\cap B_{2R+1}(0)\right)$.
Thanks to \eqref{eq propertity of z-p} and Lemma \ref{lemma uniform estimate}, we have, by standard elliptic theory, that $\varphi_p$ is uniformly bounded in ${\Omega}_{p}\cap B_{2R+1}(0).$ On the other hand, the function $\psi_p$ is harmonic in $\Omega_{p}\cap B_{2R+1}(0)$ and satisfies $\psi_p=-p\to-\infty$ on $\partial\Omega_p\cap B_{2R+1}(0).$ Since $\partial\Omega_p\cap B_{2R+1}(0)\to(\mathbb{R}\times\{R\})\cap B_{2R+1}(0)$ as $p\to+\infty$, then by Harnack inequality, we easily gets that $\psi_p(0)\to-\infty$ as $p\to+\infty$. This is a contradiction, since $\psi_p(0)=-\varphi_{p}(0)$ and $\varphi_p$ is bounded, hence $\Omega_{p}\to \R^{2}$ as $p\to+\infty$. 

Notice that for any $R>0$, $B_{R}(0)\subset\Omega_{p}$ for $p$ large enough, $v_{p}$ is a family of positive functions with uniformly bounded Laplacian in $B_{R}(0)$ and with $v_{p}(0)=0$. Now, arguing as before, we write $v_{p}=\varphi_{p}+\psi_{p}$, where $\varphi_{p}$ is uniformly bounded in $B_{R}(0)$ and $\psi_{p}$ is an harmonic function, which is uniformly bounded above. By the Harnack inequality, either $\psi_{p}$ is uniformly bounded in $B_{R}(0)$ or it tends to $-\infty$ on each compact set of $B_{R}(0)$. But the second alternative cannot happen because, by definition $\psi_{p}(0)=v_{p}(0)-\varphi_{p}(0)\geq -C.$ Hence $v_{p}$ is uniformly bounded in $B_{R}(0)$ for all $R>0$. Then by standard elliptic regularity theory and the Arzela–Ascoli theorem, we have that $v_{p}\to v$ in $C^{2}_{loc}(\R^{2})$ as $p\to+\infty$. 


Since $v_{p}\to v$ in $C_{loc}^{2}(\R^{2})$, then by Taylor expansion, we get
\begin{equation}\label{Prop-converge-1}
    v_{p}+p\left(\log\left( 1+\frac{v_{p}}{p}\right)-\frac{v_{p}}{p}\right)\to v,\text{~in~~}C_{loc}(\R^{2}).
\end{equation}
Then, by the Fatou's lemma, changing of variables, the definition of $\varepsilon_{p}$, the H\"older inequality and Corollary \ref{cor energy}, we deduce that
\begin{equation}\label{proof-converge-estimate-1}
\begin{aligned}
    \int_{\R^{2}}\int_{\R^{2}}\frac{e^{v(y)}e^{v(x)}}{|x-y|^{\alpha}}dydx&\leq \liminf_{p\to+\infty}\int_{\Omega_{p}}\int_{\Omega_{p}}\frac{\left(1+\frac{v_{p}(y)}{p}\right)^{p}\left(1+\frac{v_{p}(x)}{p}\right)^{p}}{|x-y|^{\alpha}}dydx\\
    &=\liminf_{p\to+\infty}p\int_{\Omega}\int_{\Omega}\frac{u_{p}^{p}(y)u_{p}^{p}(x)}{|x-y|^{\alpha}}dydx\\
    &\leq \liminf_{p\to+\infty}\left(p\int_{\Omega}\int_{\Omega}\frac{u_{p}^{p+1}(y)u_{p}^{p+1}(x)}{|x-y|^{\alpha}}dydx\right)^{\frac{p}{p+1}}\left(p\int_{\Omega}\int_{\Omega}\frac{1}{|x-y|^{\alpha}}dydx\right)^{\frac{1}{p+1}}\\
    &=2(4-\alpha)\pi e
\end{aligned}  
\end{equation}
and similar to the estimate \eqref{estimate integral}
\begin{equation}\label{estimate of u-p-p+1}
    \begin{aligned}
        \int_{\R^{2}}e^{\frac{4}{4-\alpha}v(x)}dx&\leq \liminf_{p\to+\infty}\int_{\Omega_{p}}\left(1+\frac{v_{p}(x)}{p}\right)^{\frac{4p}{4-\alpha}}dx=\liminf_{p\to+\infty}\frac{1}{\varepsilon_{p}^{2}u_{p}^{\frac{4p}{4-\alpha}}(x_{p})}\int_{\Omega}u_{p}^{\frac{4p}{4-\alpha}}(x)dx\\
        &=\liminf_{p\to+\infty}p^{\frac{2}{4-\alpha}}\int_{\Omega}u_{p}^{\frac{4p}{4-\alpha}}(x)dx\lesssim1.
    \end{aligned}
\end{equation}

For any $\phi\in C_{c}^{\infty}(\R^{2})$, there exists $R_{1}>0$ such that $\phi(x)=0$ for all $x\in \R^{2}\setminus B_{R_{1}}(0)$. Next, choose $R>0$ sufficiently large so that $B_{R/2}(x)\subset B_{R}(0)$ for all $x\in B_{R_{1}}(0)$. Notice that, from \eqref{proof-converge-estimate-1}, we have
\begin{equation}\label{Prop-converge-4}
   \int_{B_{R}(0)}\int_{B_{R}(0)}\frac{e^{v(y)}e^{v(x)}}{|x-y|^{\alpha}}dydx=o_{R}(1),\text{~for any~}R\text{~large enough~}. 
\end{equation}
Then combining Lemma \ref{lemma uniform estimate}, \eqref{Prop-converge-1} and \eqref{Prop-converge-4}, we obtain that
\begin{equation}
    \begin{aligned}
       &\int_{\Omega_{p}}\int_{\Omega_{p}}\frac{\left(1+\frac{v_{p}(y)}{p}\right)^{p+1}\left(1+\frac{v_{p}(x)}{p}\right)^{p}\phi(x)}{|x-y|^{\alpha}}dxdy\\
       &= \int_{B_{R}(0)} \int_{B_{R_{1}}(0)}\frac{\left(1+\frac{v_{p}(y)}{p}\right)^{p+1}\left(1+\frac{v_{p}(x)}{p}\right)^{p}\phi(x)}{|x-y|^{\alpha}}dxdy\\
       &\quad+\int_{\Omega_{p}\setminus B_{R}(0)} \int_{B_{R_{1}}(0)}\frac{\left(1+\frac{v_{p}(y)}{p}\right)^{p+1}\left(1+\frac{v_{p}(x)}{p}\right)^{p}\phi(x)}{|x-y|^{\alpha}}dxdy\\
       &=\int_{\R^{2}} \int_{\R^{2}}\frac{e^{v(y)}e^{v(x)}\phi(x)}{|x-y|^{\alpha}}dxdy+o_{p}(1)+o_{R}(1),\text{~for any~}p\text{~and~}R\text{~large enough}. 
    \end{aligned}
\end{equation}
From \eqref{eq for z-p} and the convergence $v_{p}\to v$ in $C_{loc}^{2}(\mathbb{R}^{2})$, we then deduce that $v$ is a distributional solution to
\begin{equation}
    -\Delta v=\left(\int_{\R^{2}}\frac{e^{v(y)}}{|x-y|^{\alpha}}dy\right)e^{v(x)}\text{~in~}\R^{2}.
\end{equation}
Moreover, using \eqref{eq propertity of z-p}, we have $v(0)=\max_{x\in\R^{2}}v(x)=0$. Thanks to the classification results in Theorem B, we conclude that $v$ must take the form
\begin{equation}
    v(x)=\frac{4-\alpha}{2}\log\left(\frac{1}{1+C_{\alpha}^{-2}|x|^{2}}\right)
\end{equation}
and therefore \eqref{eq-integral} holds.
\end{proof}
\begin{Prop}\label{prop L-infty estimate}  Let $\alpha\in(0,2)$ and $\Omega\subset\R^{2}$ be a smooth bounded domain satisfying $(H_{1})$. Then we have
\begin{equation}
    1\leq \liminf_{p\to+\infty}\|u_{p}\|_{L^{\infty}(\Omega)}\leq \limsup_{p\to+\infty}\|u_{p}\|_{L^{\infty}(\Omega)}\leq \sqrt{e}.
\end{equation}
\end{Prop}
\begin{proof}
    First, from the definition of $S_{p}$, $u_{p}$ and $v_{p}$, we get
    \begin{equation}
    \begin{aligned}
         S_{p}^{\frac{2(p+1)}{p}}&=\int_{\Omega}\int_{\Omega}\frac{u_{p}^{p+1}(y)u_{p}^{p+1}(x)}{|x-y|^{\alpha}}dydx\\
         &=\frac{\|u_{p}\|_{L^{\infty}(\Omega)}^{2}}{p}\int_{\Omega_{p}}\int_{\Omega_{p}}\frac{\left(1+\frac{v_{p}(y)}{p}\right)^{p+1}\left(1+\frac{v_{p}(x)}{p}\right)^{p+1}}{|x-y|^{\alpha}}dydx.
    \end{aligned}    
    \end{equation}
    Then by Proposition \ref{prop convergence of z-p} and the Fatou's lemma, we have
    \begin{equation}
    \begin{aligned}
        2(4-\alpha)\pi=\int_{\R^{2}}\int_{\R^{2}}\frac{e^{v(y)}e^{v(x)}}{|x-y|^{\alpha}}dydx&\leq \liminf_{p\to+\infty} \int_{\Omega_{p}}\int_{\Omega_{p}}\frac{\left(1+\frac{v_{p}(y)}{p}\right)^{p+1}(\left(1+\frac{v_{p}(x)}{p}\right)^{p+1}}{|x-y|^{\alpha}}dydx\\
        &=\liminf_{p\to+\infty} \frac{pS_{p}^{\frac{2(p+1)}{p}}}{\|u_{p}\|_{L^{\infty}(\Omega)}^{2}}.
    \end{aligned}
\end{equation}
From Proposition \ref{estimate of S-p}, we obtain $\limsup_{p\to+\infty}\|u_{p}\|_{L^{\infty}(\Omega)}\leq \sqrt{e}$. Combining this with the lower bound in Proposition \ref{lower bound of L-infty norm} completes the proof.
\end{proof}

\begin{Cor}\label{cor-3.8}
 Let $\alpha\in(0,2)$ and $\Omega\subset\R^{2}$ be a smooth bounded domain satisfying $(H_{1})$. Then, there exist positive constants $C_{1}$ and $C_{2}$ such that for all sufficiently large $p$,
    \begin{equation}
        \frac{C_{1}}{p}\leq \int_{\Omega}\int_{\Omega}\frac{u_{p}^{p+1}(y)u_{p}^{p}(x)}{|x-y|^{\alpha}}dydx\leq \frac{C_{2}}{p}.
    \end{equation}
\end{Cor}
\begin{proof}
    This follows immediately from Corollary \ref{cor energy} and Proposition \ref{prop L-infty estimate}.
\end{proof}

\begin{Prop}
 Let $\alpha\in(0,2)$ and $\Omega\subset\R^{2}$ be a smooth bounded domain satisfying $(H_{1})$. Then it holds that $\sqrt{p}u_{p}\weakto 0$ in $H^{1}_{0}(\Omega)$ as $p\to+\infty$.
\end{Prop}
\begin{proof}
    From Corollary \ref{cor energy}, we know that $\|\sqrt{p}u_{p}\|_{H^{1}_{0}(\Omega)}$ are bounded uniformly on $p$, thus there exists a function $w\in H^{1}_{0}(\Omega)$ such that $\sqrt{p}u_{p}\weakto w$ in $H^{1}_{0}(\Omega)$ as $p\to+\infty$. For any $\varphi\in C_{c}^{\infty}(\Omega)$, from \eqref{slightly subcritical choquard equation} and Corollary \ref{cor-3.8}, we have
    \begin{equation}
    \begin{aligned}
        \int_{\Omega}\nabla (\sqrt{p}u_{p})\cdot\nabla\varphi dx&=\sqrt{p} \int_{\Omega} \int_{\Omega}\frac{u_{p}^{p+1}(y)u_{p}^{p}(x)\varphi(x)}{|x-y|^{\alpha}}dydx\\
        &\leq \frac{\|\varphi\|_{L^{\infty}}}{\sqrt{p}} p\int_{\Omega}\int_{\Omega}\frac{u_{p}^{p+1}(y)u_{p}^{p}(x)}{|x-y|^{\alpha}}dydx\\
        &\leq \frac{C\|\varphi\|_{L^{\infty}}}{\sqrt{p}}\to0\text{~~as~~}p\to+\infty.
    \end{aligned}        
    \end{equation}
    On the other hand
    \begin{equation}
        \int_{\Omega}\nabla (\sqrt{p}u_{p})\cdot\nabla\varphi dx\to\int_{\Omega}\nabla w\cdot\nabla\varphi dx\text{~~as~~}p\to+\infty.
    \end{equation}
    Thus for any $\varphi\in C_{c}^{\infty}(\Omega)$, we have $\int_{\Omega}\nabla w\cdot\nabla\varphi dx=0$ and then $w=0$.
\end{proof}

\begin{proof}[Proof of Theorem \ref{thm-1}]
Theorem \ref{thm-1} is now a direct corollary of the preceding propositions.
\end{proof}

\section{Proof of Theorem \ref{thm-2}}\label{section-Proof of Theorem 2}
In this section, we study the asymptotic behavior of $\bar{u}_{p}(x)$, where
\begin{equation}
    \bar{u}_{p}(x):=pu_{p}(x)\text{~and~}f_{p}(x):=p\left(\int_{\Omega}\frac{u_{p}^{p+1}(y)}{|x-y|^{\alpha}}dy\right)u_{p}^{p}(x)\text{~for any~}x\in\Omega.
\end{equation}
The function \(\bar{u}_{p}\) satisfies equation
\begin{equation}\label{eq-bar-u-p}
    \begin{cases}
         -\Delta \bar{u}_{p}=f_{p},\quad \bar{u}_{p}>0,\ \  &\mbox{in}\ \Omega,\\
 \quad \ \ \bar{u}_{p}=0, \ \  &\mbox{on}\ \partial \Omega,
    \end{cases}
\end{equation}
and its blow-up set is defined as
\begin{equation}
    \mathcal{S}:=\left\{y\in\bar{\Omega}:\text{~there exist~}\{y_{p}\}\subset\Omega\text{~such that~}\bar{u}_{p}(y_{p})\to+\infty\text{~and~}y_{p}\to y\text{~as~}p\to+\infty\right\}.
\end{equation}
Since $\liminf_{p\to+\infty}u_{p}(x_{p})=\liminf_{p\to+\infty}\max_{x\in\bar{\Omega}}u_{p}\geq1$, it follows that $x_{p}\to x_{0}\in \mathcal{S}$ and $\mathcal{S}\not= \emptyset$.

\subsection{Blow-up analysis}
In this subsection, we prove that the blow-up set $\mathcal{S}$ is disjoint from the boundary and consists of a single point $x_{0}$. Before proceeding, we introduce some notations. Assume that there exist $n\in\N$ families of points $\{x_{i,p}\}\subset\Omega$, $i=1,\cdots,n$ such that 
\begin{equation}
    pu_{p}^{2p}(x_{i,p})\to+\infty\text{~as~}p\to+\infty,
\end{equation}
and we define the parameters $\varepsilon_{i,p}$ by
\begin{equation}
    \varepsilon^{-(4-\alpha)}_{i,p}:=pu_{p}^{2p}(x_{i,p})\text{~for any~}i=1,\cdots,n.
\end{equation}
Then $\varepsilon_{i,p}\to0$ and 
\begin{equation}\label{eq-liminf-u-p}
    \liminf_{p\to+\infty}u_{p}(x_{i,p})\geq 1.
\end{equation}
Next, we define the concentration set
\begin{equation}
    {\Lambda}:=\left\{\lim_{p\to+\infty}x_{i,p},~i=1,\cdots,n\right\}\subset\bar{\Omega}
    \end{equation}
and the distance function
\begin{equation}
    \mathcal{D}_{n,p}(x):=\min_{i=1,\cdots,n}|x-x_{i,p}|\text{~for any~}x\in\Omega.
\end{equation}
Finally, we introduce the following properties:
\begin{enumerate}
    \item[($\mathbf{\mathcal{P}_{1}^{n}}$)] For any $i,j \in \{1,\dots,n\}$ with $i \neq j$,
    \begin{equation}
        \lim_{p \to +\infty} \frac{\mathrm{dist}(x_{i,p}, \partial \Omega)}{\varepsilon_{i,p}} 
        = \lim_{p \to +\infty} \frac{|x_{i,p} - x_{j,p}|}{\varepsilon_{i,p}} = +\infty.
    \end{equation}

    \item[($\mathbf{\mathcal{P}_{2}^{n}}$)] For any $i \in \{1,\dots,n\}$ and $x \in \Omega_{i,p} := \frac{\Omega - x_{i,p}}{\varepsilon_{i,p}}$,
    \begin{equation}
        v_{i,p}(x) := \frac{p}{u_{p}(x_{i,p})} \big( u_{p}(x_{i,p} + \varepsilon_{i,p} x) - u_{p}(x_{i,p}) \big) \to v(x)
    \end{equation}
    in $C^{2}_{\mathrm{loc}}(\mathbb{R}^{2})$ as $p \to +\infty$, where
    \begin{equation}
        v(x) := \frac{4 - \alpha}{2} \log \left( \frac{1}{1 + C_{\alpha}^{-2} |x|^{2}} \right)\text{
        with~}C_{\alpha} := \left( \frac{(2 - \alpha)(4 - \alpha)}{\pi} \right)^{\frac{1}{4 - \alpha}}.
    \end{equation}

    \item[($\mathbf{\mathcal{P}_{3}^{n}}$)] There exists $C > 0$ such that
    \begin{equation}
        p \mathcal{D}_{n,p}^{4 - \alpha}(x) u_{p}^{2p}(x) \leq C
    \end{equation}
    for any $p > 1$ and all $x \in \Omega$.
\end{enumerate}

\begin{Lem}\label{lem-n-2}
    If there exist $n\in\N$ families of points $\{x_{i,p}\}$, $i=1,\cdots,n$ such that ($\mathbf{\mathcal{P}_{1}^{n}}$) and ($\mathbf{\mathcal{P}_{2}^{n}}$) hold. Then $n\leq 2$.
    \end{Lem}
\begin{proof}
For any $R > 0$, by ($\mathbf{\mathcal{P}_{1}^{n}}$) and a change of variables, we obtain that
 \begin{equation}
     \begin{aligned}
        p\int_{\Omega}|\nabla u_{p}|^{2}dx&=p\int_{\Omega}\int_{\Omega}\frac{u_{p}^{p+1}(y)u_{p}^{p+1}(x)}{|x-y|^{\alpha}}dydx\\
        &\geq\sum_{i=1}^{n}p\int_{B_{R\varepsilon_{i,p}}(x_{i,p})}\int_{B_{R\varepsilon_{i,p}}(x_{i,p})}\frac{u_{p}^{p+1}(y)u_{p}^{p+1}(x)}{|x-y|^{\alpha}}dydx\\
        &=\sum_{i=1}^{n}u_{p}^{2}(x_{i,p})\int_{B_{R}(0)}\int_{B_{R}(0)}\frac{(1+\frac{v_{i,p}(y)}{p})^{p+1}(1+\frac{v_{i,p}(x)}{p})^{p+1}}{|x-y|^{\alpha}}dydx,
     \end{aligned}
 \end{equation}
for any $p$ large enough. Moreover by \eqref{eq-liminf-u-p}, ($\mathbf{\mathcal{P}_{2}^{n}}$) and Fatou's lemma, we have
\begin{equation}
  \liminf_{p\to+\infty}p\int_{\Omega}|\nabla u_{p}|^{2}dx\geq \sum_{i=1}^{n}\int_{B_{R}(0)}\int_{B_{R}(0)}\frac{e^{v(y)}e^{v(x)}}{|x-y|^{\alpha}}dydx
\end{equation}
for any $R>0$. Recall that $\lim_{p\to+\infty}p\int_{\Omega}|\nabla u_{p}|^{2}dx=2(4-\alpha)\pi e$ and $\int_{\R^{2}}\int_{\R^{2}}\frac{e^{v(y)}e^{v(x)}}{|x-y|^{\alpha}}dydx=2(4-\alpha)\pi$, then $n\leq e$ and so $n\leq 2$.
\end{proof}
\begin{Prop}\label{prop-multi-peak}
 Let $\alpha\in(0,2)$ and $\Omega\subset\R^{2}$ be a smooth bounded domain satisfying $(H_{1})$. Then, there exist $k\in\{1,2\}$ and $k$ families of points $\{x_{i,p}\}\subset\Omega$, $i=1,\cdots,k$ such that ($\mathbf{\mathcal{P}_{1}^{k}}$), ($\mathbf{\mathcal{P}_{2}^{k}}$) and ($\mathbf{\mathcal{P}_{3}^{k}}$) hold. Moreover, $x_{1,p}=x_{p}$ and given any family of points $x_{k+1,p}$, it is impossible to extract a new sequence from the previous one such that ($\mathbf{\mathcal{P}_{1}^{k+1}}$), ($\mathbf{\mathcal{P}_{2}^{k+1}}$) and ($\mathbf{\mathcal{P}_{3}^{k+1}}$) hold with the sequence $\{x_{i,p}\}$, $i=1,\cdots,k+1$. Finally, we have
\begin{equation}\label{prop-multi-peak-eq-1}
    \sqrt{p}u_{p}\to 0\text{~~in~~}C^{1}_{loc}(\bar{\Omega}\setminus\Lambda)\text{~as~}p\to+\infty.
\end{equation}
\end{Prop}
\begin{proof}
 Let $x_{1,p}=x_{p}$, then ($\mathbf{\mathcal{P}_{1}^{1}}$) and ($\mathbf{\mathcal{P}_{2}^{1}}$) hold. If ($\mathbf{\mathcal{P}_{3}^{1}}$) holds, then the assertion is proved with $k=1$. Otherwise, take $x_{2,p} \in \bar{\Omega}$ such that
 \begin{equation}\label{prop-multi-blow-up-proof-1}
     p|x_{2,p}-x_{1,p}|^{4-\alpha}u_{p}^{2p}(x_{2,p})=\max_{x\in\bar{\Omega}}p|x-x_{1,p}|^{4-\alpha}u_{p}^{2p}(x)\to+\infty
 \end{equation}
 and then define $\varepsilon_{2,p}$ by
 \begin{equation}
     \varepsilon_{2,p}^{-(4-\alpha)}:=pu_{p}^{2p}(x_{2,p}).
 \end{equation}
Since $\Omega$ is bounded and $u_p = 0$ on $\partial\Omega$, it follows that $x_{2,p} \in \Omega$, $p u_p^{2p}(x_{2,p}) \to +\infty$, $\varepsilon_{2,p} \to 0$, and  
\begin{equation}\label{prop-multi-blow-up-proof-3}
    \lim_{p \to +\infty} \frac{|x_{2,p} - x_{1,p}|}{\varepsilon_{2,p}} = +\infty.
\end{equation}
Moreover, we have
\begin{equation} 
    \lim_{p \to +\infty} \frac{|x_{2,p} - x_{1,p}|}{\varepsilon_{1,p}} = +\infty.
\end{equation}
Otherwise, there exists $R>0$ such that $\frac{|x_{2,p} - x_{1,p}|}{\varepsilon_{1,p}}\to R$ as $p\to+\infty$. Thanks to ($\mathbf{\mathcal{P}_{2}^{1}}$), we have
\begin{equation}
    p|x_{1,p}-x_{2,p}|^{4-\alpha}u_{p}^{2p}(x_{2,p})\to \left
    (\frac{R}{1+C_{\alpha}^{-2}R^{2}}\right)^{4-\alpha}<+\infty,
\end{equation}
this contracts to \eqref{prop-multi-blow-up-proof-1}.

Next, we consider the rescaled function:
\begin{equation}
    v_{2,p}(x):=\frac{p}{u_{p}(x_{2,p})}\left(u_{p}(\varepsilon_{2,p}x+x_{2,p})-u_{p}(x_{2,p})\right)\text{~for any~}x\in\Omega_{2,p}:=\frac{\Omega-x_{2,p}}{\varepsilon_{2,p}}.
\end{equation}
It's easy to verify that $v_{2,p}(x)$ satisfies the following equation
\begin{equation}
    \begin{cases}
        -\Delta v_{2,p}=\left(\int_{\Omega_{2,p}}\frac{\left(1+\frac{v_{2,p}(y)}{p}\right)^{p+1}}{|x-y|^{\alpha}}dy\right)\left(1+\frac{v_{2,p}(x)}{p}\right)^{p},&\text{~~in~}\Omega_{2,p},\\
        \quad\ \ v_{2,p}=-p,&\text{~~on~}\partial\Omega_{2,p}.\\
    \end{cases}
\end{equation}
Fix $R>0$. Let $\tilde{x}_{p}$ be any point in $ \Omega_{2,p}\cap B_{R}(0)$ and $\tilde{y}_{p}\in \Omega_{2,p}\cap B_{1}(\tilde{x}_{p})\subset\Omega_{2,p}\cap B_{R+1}(0)$. The corresponding points in $\Omega$ are $\bar{x}_{p}=x_{2,p}+\varepsilon_{2,p}\tilde{x}_{p}$ and $\bar{y}_{p}=x_{2,p}+\varepsilon_{2,p}\tilde{y}_{p}$. By definition of $x_{2,p}$, we have
\begin{equation}\label{prop-multi-blow-up-proof-8}
    p|\bar{x}_{p}-x_{1,p}|^{4-\alpha}u_{p}^{2p}(\bar{x}_{p})\leq  p|x_{2,p}-x_{1,p}|^{4-\alpha}u_{p}^{2p}(x_{2,p})
\end{equation}
and
\begin{equation}\label{prop-multi-blow-up-proof-9}
    p|\bar{y}_{p}-x_{1,p}|^{4-\alpha}u_{p}^{2p}(\bar{y}_{p})\leq  p|x_{2,p}-x_{1,p}|^{4-\alpha}u_{p}^{2p}(x_{2,p}).
\end{equation}
Since $|\bar{x}_{p}-x_{2,p}|\leq R\varepsilon_{2,p}$, the triangle inequality yields
\begin{equation}
   |\bar{x}_{p}-x_{1,p}|\geq |x_{2,p}-x_{1,p}|-|\bar{x}_{p}-x_{2,p}|\geq |x_{2,p}-x_{1,p}| -R\varepsilon_{2,p}
\end{equation}
and 
\begin{equation}
    |\bar{x}_{p}-x_{1,p}|\leq |x_{2,p}-x_{1,p}|+|\bar{x}_{p}-x_{2,p}|\leq |x_{2,p}-x_{1,p}|+R\varepsilon_{2,p}.
\end{equation}
Therefore, by \eqref{prop-multi-blow-up-proof-3}, we obtain
\begin{equation}
    |\bar{x}_{p}-x_{1,p}|=(1+o(1))|x_{2,p}-x_{1,p}|.
\end{equation}
Similarly, since $|\bar{y}_{p}-x_{2,p}|\leq (R+1)\varepsilon_{2,p}$, we have
\begin{equation}
    |\bar{y}_{p}-x_{1,p}|=(1+o(1))|x_{2,p}-x_{1,p}|.
\end{equation}
Consequently, from \eqref{prop-multi-blow-up-proof-8} and \eqref{prop-multi-blow-up-proof-9}, we conclude
\begin{equation}
    u_{p}^{2p}(\bar{x}_{p})\leq (1+o(1))u_{p}^{2p}(x_{2,p})\text{~and~}u_{p}^{2p}(\bar{y}_{p})\leq(1+o(1))u_{p}^{2p}(x_{2,p}).
\end{equation}
Next, if $v_{2,p}(\tilde{x}_{p})>0$, then $u_{p}(\bar{x}_{p})=\frac{u_{p}(x_{2,p})}{p}v_{2,p}(\tilde{x}_{p})+u_{p}(x_{2,p})\geq u_{p}(x_{2,p})>0$ and 
\begin{equation}
    0<\left(1+\frac{v_{2,p}(\tilde{x}_{p})}{p}\right)^{p}=\left(\frac{u_{p}(\bar{x}_{p})}{u_{p}(x_{2,p})}\right)^{p}=1+o(1).
\end{equation}
If $v_{2,p}(\tilde{x}_{p})\leq 0$, then $0<u_{p}(\bar{x}_{p})=\frac{u_{p}(x_{2,p})}{p}v_{2,p}(\tilde{x}_{p})+u_{p}(x_{2,p})\leq u_{p}(x_{2,p})$ and 
\begin{equation}
   0<\left(1+\frac{v_{2,p}(\tilde{x}_{p})}{p}\right)^{p}=\left(\frac{u_{p}(\bar{x}_{p})}{u_{p}(x_{2,p})}\right)^{p}\leq 1.
\end{equation}
Similarly, we can obtain that
\begin{equation}
    0<\left(1+\frac{v_{2,p}(\tilde{y}_{p})}{p}\right)^{p}=\left(\frac{u_{p}(\bar{y}_{p})}{u_{p}(x_{2,p})}\right)^{p}\leq 1+o(1).
\end{equation}
Using the techniques similar to those in Lemma \ref{lemma uniform estimate} and Proposition \ref{prop convergence of z-p}, we can prove that
\begin{equation}
    \int_{\Omega_{2,p}}\frac{\left(1+\frac{v_{2,p}(y)}{p}\right)^{p+1}}{|x-y|^{\alpha}}dy\lesssim 1\text{~for any~}x\in \Omega_{2,p}\cap B_{2R}(0)
\end{equation}
and then
\begin{equation}
     \lim_{p \to +\infty} \frac{\mathrm{dist}(x_{2,p}, \partial \Omega)}{\varepsilon_{2,p}}=+\infty \text{~and~}v_{2,p}\to v\text{~in~}C^{2}_{loc}(\R^{2}).
\end{equation}
This implies that ($\mathbf{\mathcal{P}_{1}^{2}}$) and ($\mathbf{\mathcal{P}_{2}^{2}}$) hold. 

Now, if ($\mathbf{\mathcal{P}_{3}^{2}}$) holds, then the assertion is proved with $k=2$. Otherwise, we can prove ($\mathbf{\mathcal{P}_{1}^{3}}$) and ($\mathbf{\mathcal{P}_{2}^{3}}$) hold similarly. However, this contracts to Lemma \ref{lem-n-2}. Hence, there exist $k\in\{1,2\}$ and $k$ families of points $\{x_{i,p}\}\subset\Omega$, $i=1,\cdots,k$ such that ($\mathbf{\mathcal{P}_{1}^{k}}$), ($\mathbf{\mathcal{P}_{2}^{k}}$) and ($\mathbf{\mathcal{P}_{3}^{k}}$) hold. Moreover, given any other family of points \(x_{k+1,p}\), it is impossible to extract a new sequence from it such that \((\mathcal{P}_{1}^{k+1})\), \((\mathcal{P}_{2}^{k+1})\) and \((\mathcal{P}_{3}^{k+1})\) hold together with the points $\{x_{i,p}\},i=1,\ldots,k+1$. Indeed, if \((\mathcal{P}_{1}^{k+1})\) hold then
\begin{equation}
   |x_{k+1,p}-x_{i,p}|/\varepsilon_{k+1,p}\to+\infty\quad\text{ as }p\to+\infty,\text{ for any }i\in\{1,\ldots,k\}, 
\end{equation}
but this would contradict \((\mathcal{P}_{3}^{k})\).

Finally, the proof of \eqref{prop-multi-peak-eq-1} is a direct consequence of \((\mathcal{P}_{3}^{k})\). Indeed, for give any \(K\) is a compact subset of $(\bar{\Omega}\setminus{\Lambda})$,  there exists $\eta>0$ small enough such that $(\Omega\cap B_{\eta}(x))\cap{\Lambda}=\emptyset$ for all $x\in K$. Then by \eqref{eq-l-infty-bounded}, \eqref{estimate of u-p-p+1} and property ($\mathcal{P}_{3}^{k}$), we have
\begin{equation}
\begin{aligned}
   &\max_{x\in K}\sqrt{p}\left(\int_{\Omega}\frac{u_{p}^{p+1}(y)}{|x-y|^{\alpha}}dy\right)u_{p}^{p}(x)\\
   &=\max_{x\in K}\sqrt{p}\left(\int_{\Omega\cap B_{\eta}(x)}\frac{u_{p}^{p+1}(y)}{|x-y|^{\alpha}}dy\right)u_{p}^{p}(x)+\max_{x\in K}\sqrt{p}\left(\int_{\Omega\setminus B_{\eta}(x)}\frac{u_{p}^{p+1}(y)}{|x-y|^{\alpha}}dy\right)u_{p}^{p}(x)\\
   &\leq Cp^{-\frac{1}{2}}\to0.
\end{aligned}    
\end{equation}
Hence standard elliptic theory shows that \(\sqrt{p}\, u_{p}\to w\) in \(C^{1}(K)\) for some \(w\). Moreover by Theorem \ref{thm-1}, we know that \(\sqrt{p}\, u_{p}\to 0\), so \(w=0\). This ends the proof.
\end{proof}

\begin{Prop}\label{prop-estimate-gradient} Let $\alpha\in(0,2)$ and $\Omega\subset\R^{2}$ be a smooth bounded domain satisfying $(H_{1})$. Then there exists $C > 0$ such that
\begin{equation}
    \begin{aligned}
        |\nabla \bar{u}_{p}(x)|&\leq  \frac{C}{\min_{i=1,\cdots,k}|x-x_{i,p}|}
    \end{aligned}
\end{equation}
for any $x\in\Omega$ and $p$ large enough.
\end{Prop}
\begin{proof}
By the Green's representation formula and the gradient estimate $|\nabla_{x}G(x,y)|\leq \frac{C}{|x-y|}$ (see for instance \cite{DallAcqua-2004-jde}), we have
    \begin{equation}\label{prop-estimate for-p-u-p-proof-1}
        \begin{aligned}
            |\nabla \bar{u}_{p}(x)&|=p\left|\int_{\Omega
            }\nabla_{x}G(x,y)\left(\int_{\Omega}\frac{u_{p}^{p+1}(z)}{|y-z|^{\alpha}}dz\right)u_{p}^{p}(y)dy\right|\\
            &\leq p\int_{\Omega
            }|\nabla_{x}G(x,y)|\left(\int_{\Omega}\frac{u_{p}^{p+1}(z)}{|y-z|^{\alpha}}dz\right)u_{p}^{p}(y)dy\\
            &\leq Cp\int_{\Omega
            }\left(\int_{\Omega}\frac{u_{p}^{p+1}(z)}{|y-z|^{\alpha}}dz\right)\frac{u_{p}^{p}(y)}{|x-y|}dy.
        \end{aligned}
    \end{equation}
Let $\mathcal{D}_{p}(x)=\min_{i=1,\cdots,k}|x-x_{i,p}|$ and  $\Omega_{i,p}=\left\{x\in\Omega:|x-x_{i,p}|=\mathcal{D}_{p}(x)\right\}$ for any $i=1,\cdots,k.$ Then, we have
\begin{equation}\label{prop-estimate for-p-u-p-proof-2}
    \begin{aligned}
       p\int_{\Omega_{i,p}
            }\left(\int_{\Omega}\frac{u_{p}^{p+1}(z)}{|y-z|^{\alpha}}dz\right)\frac{u_{p}^{p}(y)}{|x-y|}dy&=p\int_{\Omega_{i,p}\cap B_{|x-x_{i,p}|/2}(x_{i,p})
            }\left(\int_{\Omega}\frac{u_{p}^{p+1}(z)}{|y-z|^{\alpha}}dz\right)\frac{u_{p}^{p}(y)}{|x-y|}dy\\
            &\quad+p\int_{\Omega_{i,p}\setminus B_{|x-x_{i,p}|/2}(x_{i,p})
            }\left(\int_{\Omega}\frac{u_{p}^{p+1}(z)}{|y-z|^{\alpha}}dz\right)\frac{u_{p}^{p}(y)}{|x-y|}dy. 
    \end{aligned}
\end{equation}
For $y\in \Omega_{i,p}\cap B_{|x-x_{i,p}|/2}(x_{i,p})$, we have
$ |x-y|\geq |x-x_{i,p}|-|y-x_{i,p}|\geq |x-x_{i,p}|/2$. Then by \eqref{eq-thm-1.1-4}
\begin{equation}\label{prop-estimate for-p-u-p-proof-3}
\begin{aligned}
    p&\int_{\Omega_{i,p}\cap B_{|x-x_{i,p}|/2}(x_{i,p})
            }\left(\int_{\Omega}\frac{u_{p}^{p+1}(z)}{|y-z|^{\alpha}}dz\right)\frac{u_{p}^{p}(y)}{|x-y|}dy\\
            &\leq\frac{2p}{|x-x_{i,p}|}\int_{\Omega_{i,p}\cap B_{|x-x_{i,p}|/2}(x_{i,p})
            }\left(\int_{\Omega}\frac{u_{p}^{p+1}(z)}{|y-z|^{\alpha}}dz\right)u_{p}^{p}(y)dy\\
            &\leq \frac{C}{|x-x_{i,p}|}.
\end{aligned}   
\end{equation}
For $y\in \Omega_{i,p}\setminus B_{|x-x_{i,p}|/2}(x_{i,p})$, we have $pu_{p}^{2p}(y)\leq \frac{C}{|y-x_{i,p}|^{4-\alpha}}\leq\frac{C}{|x-x_{i,p}|^{4-\alpha}}$. Let $\tilde{\Omega}_{i,p}:=\Omega_{i,p}\setminus B_{|x-x_{i,p}|/2}(x_{i,p})\cap\{|x-y|\leq |x-x_{i,p}|\}$, then by the HLS inequality and \eqref{estimate of u-p-p+1}, we have
\begin{equation}
    \begin{aligned}
        p&\int_{\tilde{\Omega}_{i,p}
            }\left(\int_{\Omega}\frac{u_{p}^{p+1}(z)}{|y-z|^{\alpha}}dz\right)\frac{u_{p}^{p}(y)}{|x-y|}dy\\
            &\leq Cp\left(\int_{\Omega
            }u_{p}^{\frac{4p}{4-\alpha}}(z)dz\right)^{\frac{4-\alpha}{4}}\left(\int_{\tilde{\Omega}_{i,p}}\frac{u_{p}^{\frac{4p}{4-\alpha}}(y)}{|x-y|^{\frac{4}{4-\alpha}}}dy\right)^{\frac{4-\alpha}{4}}\\
            &\leq C\left(\int_{\tilde{\Omega}_{i,p}}\frac{p^{\frac{2}{4-\alpha}}u_{p}^{\frac{4p}{4-\alpha}}(y)}{|x-y|^{\frac{4}{4-\alpha}}}dy\right)^{\frac{4-\alpha}{4}}\\
             &\leq \frac{C}{|x-x_{i,p}|^{(4-\alpha)/2}}\left(\int_{|x-y|\leq |x-x_{i,p}|}\frac{1}{|x-y|^{\frac{4}{4-\alpha}}}dy\right)^{\frac{4-\alpha}{4}}\\
    \end{aligned}
\end{equation}
and
\begin{equation}\label{prop-estimate for-p-u-p-proof-5}
    \begin{aligned}
        p&\int_{\Omega_{i,p}\setminus B_{|x-x_{i,p}|/2}(x_{i,p})
            }\left(\int_{\Omega}\frac{u_{p}^{p+1}(z)}{|y-z|^{\alpha}}dz\right)\frac{u_{p}^{p}(y)}{|x-y|}dy\\
            &\leq \frac{C}{|x-x_{i,p}|^{(4-\alpha)/2}}\left(\int_{|x-y|\leq |x-x_{i,p}|}\frac{1}{|x-y|^{\frac{4}{4-\alpha}}}dy\right)^{\frac{4-\alpha}{4}}\\
            &\quad+\frac{p}{|x-x_{i,p}|}\int_{\Omega}\left(\int_{\Omega}\frac{u_{p}^{p+1}(z)}{|y-z|^{\alpha}}dz\right)u_{p}^{p}(y)dy\\
            &\leq \frac{C}{|x-x_{i,p}|}.
    \end{aligned}
\end{equation}
Combining \eqref{prop-estimate for-p-u-p-proof-1}-\eqref{prop-estimate for-p-u-p-proof-5} together, we obtain
\begin{equation}
    \begin{aligned}
        |\nabla \bar{u}_{p}(x)|&\leq  \frac{C}{\min_{i=1,\cdots,k}|x-x_{i,p}|}
    \end{aligned}
\end{equation}
for any $x\in\Omega$ and $p$ large enough. This completes the proof.
\end{proof}

\begin{Prop}\label{prop-equiv-blow-up-set}
 Let $\alpha\in(0,2)$ and $\Omega\subset\R^{2}$ be a smooth bounded domain satisfying $(H_{1})$. Then it holds that 
    \begin{equation}
        \Lambda=\mathcal{S}=\left\{x\in\bar{\Omega}:\forall~ r_{0}>0,\forall~p_{0}>1, \exists~p>p_{0}~s.t.~ p\int_{\Omega\cap B_{r_{0}}(x)}\int_{\Omega}\frac{u_{p}^{p+1}(y)u_{p}^{p+1}(x)}{|x-y|^{\alpha}}dydx\geq 1\right\}.
    \end{equation}
\end{Prop}
\begin{proof}
For any $x \in \Lambda$, there exists a sequence $\{x_{i,p}\} \subset \Omega$ such that $ x_{i,p} \to x \text{~and~} p u_{p}^{2p}(x_{i,p}) \to +\infty \text{~as~} p \to +\infty.$
Consequently, for sufficiently large $p$, we have $u_{p}(x_{i,p}) \geq \frac{1}{2}$, which implies $p u_{p}(x_{i,p}) \to +\infty \text{~as~} p \to +\infty.$ Thus, $x \in \mathcal{S}$ and $\Lambda \subseteq \mathcal{S}$.
Conversely, for any $y \in \mathcal{S}$, there exists a sequence $\{y_p\} \subset \Omega$ such that $y_p \to y \text{~and~}\bar{u}_p(y_p) \to +\infty \text{~as~} p \to +\infty.$
If $y \notin \Lambda$, Proposition \ref{prop-estimate-gradient} yields the uniform bound $\bar{u}_p(y_p) \leq C$, leading to a contradiction. Therefore, $y \in \Lambda$, and we conclude $\mathcal{S} \subseteq \Lambda$. Combining both inclusions, we obtain $\Lambda = \mathcal{S}.$

For any $x\in\Lambda$, there exists a sequence $\{x_{i,p}\} \subset \Omega$ such that $ x_{i,p} \to x $ and $ p u_{p}^{2p}(x_{i,p})\to+\infty $ as $ p \to +\infty.$ Recall that for the function $v$ defined in \eqref{definition of v}, we have $\int_{\R^{2}}\int_{\R^{2}}\frac{e^{v(y)}e^{v(x)}}{|x-y|^{\alpha}}dydx=2(4-\alpha)\pi$. Consequently, there exists $R>0$ such that $\int_{B_{R}(0)}\int_{B_{R}(0)}\frac{e^{v(y)}e^{v(x)}}{|x-y|^{\alpha}}dydx>1$. Using \eqref{eq-liminf-u-p}, properties ($\mathcal{P}_{1}^{k}$) and ($\mathcal{P}_{2}^{k}$), along with the Fatou's lemma, we obtain
\begin{equation}
    \begin{aligned}
        &p\int_{\Omega\cap B_{r_{0}}(x)}\int_{\Omega}\frac{u_{p}^{p+1}(y)u_{p}^{p+1}(x)}{|x-y|^{\alpha}}dydx\\
        &\geq p\int_{B_{R\varepsilon_{i,p}}(x_{i,p})}\int_{B_{R\varepsilon_{i,p}}(x_{i,p})}\frac{u_{p}^{p+1}(y)u_{p}^{p+1}(x)}{|x-y|^{\alpha}}dydx\\
        &=u_{p}^{2}(x_{i,p})\int_{B_{R}(0)}\int_{B_{R}(0)}\frac{(1+\frac{v_{i,p}(y)}{p})^{p+1}(1+\frac{v_{i,p}(x)}{p})^{p+1}}{|x-y|^{\alpha}}dydx
    \end{aligned}
\end{equation}
and
\begin{equation}
  \liminf_{p\to+\infty}p\int_{\Omega\cap B_{r_{0}}(x)}\int_{\Omega}\frac{u_{p}^{p+1}(y)u_{p}^{p+1}(x)}{|x-y|^{\alpha}}dydx\geq\int_{B_{R}(0)}\int_{B_{R}(0)}\frac{e^{v(y)}e^{v(x)}}{|x-y|^{\alpha}}dydx>1.
\end{equation}
Conversely, if $x_{0}\not\in\Lambda$, for any sufficiently small $r > 0$, applying \eqref{eq-l-infty-bounded}, the HLS inequality, \eqref{estimate of u-p-p+1} and property ($\mathcal{P}_{3}^{k}$) yields
\begin{equation}
\begin{aligned}
   &{p}\int_{\Omega\cap B_{r}(x_{0})}\int_{\Omega}\frac{u_{p}^{p+1}(y)u_{p}^{p+1}(x)}{|x-y|^{\alpha}}dydx\\
   &={p}\int_{\Omega\cap B_{r}(x_{0})}\int_{\Omega\cap B_{2r}(x_{0})}\frac{u_{p}^{p+1}(y)u_{p}^{p+1}(x)}{|x-y|^{\alpha}}dydx+{p}\int_{\Omega\cap B_{r}(x_{0})}\int_{\Omega\setminus B_{2r}(x_{0})}\frac{u_{p}^{p+1}(y)u_{p}^{p+1}(x)}{|x-y|^{\alpha}}dydx\\
   &\leq C\int_{B_{2r}(x_{0})}\int_{B_{2r}(x_{0})}\frac{1}{|x-y|^{\alpha}}dydx+Cp\left(\int_{\Omega\cap B_{r}(x_{0})}u_{p}^{\frac{4p}{4-\alpha}}(x)dx\right)^{\frac{4-\alpha}{4}}\left(\int_{\Omega}u_{p}^{\frac{4p}{4-\alpha}}(y)dy\right)^{\frac{4-\alpha}{4}}\\
   &\leq Cr^{4-\alpha}+Cr^{\frac{4-\alpha}{2}}\to0\text{~as~}r\to 0^{+},
\end{aligned}    
\end{equation}
uniformly for $p$ large enough. This completes the proof.
\end{proof}

\begin{Cor}\label{cor-4.5}
Let $\alpha\in(0,2)$ and $\Omega\subset\R^{2}$ be a smooth bounded domain satisfying $(H_{1})$.   For any compact subset $K\subset\bar{\Omega}\setminus\mathcal{S}$, we have
    \begin{equation}\label{eq-estimate-gradite-away-from-the-blow-up-point-1}
        \lim_{p\to+\infty}\|pu_{p}^{2p+2}(x)\|_{L^{\infty}(K)}=\lim_{p\to+\infty}p\int_{K}\int_{\Omega}\frac{u_{p}^{p+1}(y)u_{p}^{p+1}(x)}{|x-y|^{\alpha}}dydx=0,
    \end{equation}
    and 
    \begin{equation}\label{eq-estimate-gradite-away-from-the-blow-up-point-2}
        \lim_{p\to+\infty}\|p|\nabla u_{p}(x)|^{2}\|_{L^{\infty}(K)}=\lim_{p\to+\infty}p\int_{K}|\nabla u_{p}(x)|^{2}dx=0.
    \end{equation}
\end{Cor}
\begin{proof}
  Notice that by Proposition \ref{prop-equiv-blow-up-set} and property ($\mathcal{P}_{3}^{k}$), there exists $C_{K}>0$ such that $\|pu_{p}^{2p}(x)\|_{L^{\infty}(K)}\leq C_{K}$. Then, by \eqref{eq-l-infty-bounded} and \eqref{prop-multi-peak-eq-1}, we obtain that for any $x \in K$
  \begin{equation}
      pu_{p}^{2p+2}(x)\leq C_{K}\|u_{p}\|_{L^{\infty}(\Omega)}u_{p}(x)\leq Cu_{p}(x)\to0\text{~uniformly as~}p\to+\infty.
  \end{equation}
Moreover, by \eqref{eq-l-infty-bounded}, the HLS inequality, \eqref{estimate of u-p-p+1} and \eqref{prop-multi-peak-eq-1}
  \begin{equation}
      \begin{aligned}
          &p\int_{K}\int_{\Omega}\frac{u_{p}^{p+1}(y)u_{p}^{p+1}(x)}{|x-y|^{\alpha}}dydx\\
          &\leq Cp\|u_{p}\|_{L^{\infty}(K)}\left(\int_{K}u_{p}^{\frac{4p}{4-\alpha}}(x)dx\right)^{\frac{4-\alpha}{4}}\left(\int_{\Omega}u_{p}^{\frac{4p}{4-\alpha}}(y)dy\right)^{\frac{4-\alpha}{4}}\\
          &\leq C\|u_{p}\|_{L^{\infty}(K)}\to 0\text{~as~}p\to+\infty.
      \end{aligned}
  \end{equation}
This establishes \eqref{eq-estimate-gradite-away-from-the-blow-up-point-1}.

Finally, by Proposition \ref{prop-estimate-gradient} and Proposition \ref{prop-equiv-blow-up-set}, there exists $C_{K}>0$ such that $p|\nabla u_{p}(x)|\leq C_{K}$ for any $x\in K$, which proves \eqref{eq-estimate-gradite-away-from-the-blow-up-point-2}.
\end{proof}

Let $N \in \mathbb{N}$ denote the number of points in $\mathcal{S}$. Then $N \leq k$. Without loss of generality, we may relabel the points $\{x_{i,p}\}$, $i=1,\cdots,k$ and assume that
\begin{equation}
\mathcal{S} = \{x_1, \dots, x_N\} \text{ and } x_{i,p} \to x_i \text{ as } p \to +\infty \text{ for each } i = 1, \dots, N.
\end{equation}

\begin{Lem}\label{lemma-out-the-blow-up-point}
Let $\alpha\in(0,2)$ and $\Omega\subset\R^{2}$ be a smooth bounded domain satisfying $(H_{1})$. Then there exists $\gamma_{i}>0$, $i=1,\cdots,N$ such that
    \begin{equation}\label{lemma-away-the-blow-up-point-1}
        \lim_{p\to+\infty}\bar{u}_{p}=\sum_{i=1}^{N}\gamma_{i} G(\cdot,x_{i})\text{~~in~~}C^{2}_{loc}(\bar{\Omega}\setminus\mathcal{S})
    \end{equation}
    and
    \begin{equation}\label{lemma-away-the-blow-up-point-2}
        \gamma_{i}=\lim_{\delta\to0}\lim_{p\to+\infty}\int_{\Omega\cap B_{\delta}(x_{i})}f_{p}(x)dx.
    \end{equation}
\end{Lem}
\begin{proof}
For any compact set $K\subset\bar{\Omega}\setminus\mathcal{S}$, there exists $\eta>0$ small enough such that $
O_{\eta}(K):=\{x\in\Omega:dist(x,K)\leq\eta\}\subset \bar{\Omega}\setminus\mathcal{S}$. Then by \eqref{eq-l-infty-bounded}, \eqref{estimate of u-p-p+1}, property ($\mathcal{P}_{3}^{N}$) and \eqref{prop-multi-peak-eq-1}, we have
\begin{equation}\label{eq-lemma-away-the-blow-up-point-proof-1}
    \begin{aligned}
        f_{p}(x)&=p\left(\int_{\Omega}\frac{u_{p}^{p+1}(y)}{|x-y|^{\alpha}}dy\right)u_{p}^{p}(x)\\
        &=p\left(\int_{\Omega\cap B_{\eta}(x)}\frac{u_{p}^{p+1}(y)}{|x-y|^{\alpha}}dy\right)u_{p}^{p}(x)+p\left(\int_{\Omega\setminus B_{\eta}(x)}\frac{u_{p}^{p+1}(y)}{|x-y|^{\alpha}}dy\right)u_{p}^{p}(x)\\
        &\lesssim \|pu_{p}^{2p+1}\|_{L^{\infty}(O_{\eta}(K))}+p\left(\int_{\Omega}u_{p}^{\frac{4p}{4-\alpha}}(y)dy\right)^{\frac{4-\alpha}{4}}u_{p}^{p}(x)\\
        &\lesssim \|pu_{p}^{2p+1}\|_{L^{\infty}(O_{\eta}(K))}+\|p^{1/2}u_{p}^{p}\|_{L^{\infty}(K)}\to 0 \text{~as~}p\to+\infty
    \end{aligned}
\end{equation}
uniformly for any $x\in K$. On the other hand, by Proposition \ref{prop-estimate-gradient}, we have $|\bar{u}_{p}|+|\nabla \bar{u}_{p}|\leq C$ uniformly for any $x\in K$. Hence by the standard elliptic regularity theory, there exists $\bar{u}\in C^{2}_{loc}(\bar{\Omega}\setminus\mathcal{S})$ such that $\bar{u}_{p}\to \bar{u}$ in $C^{2}_{loc}(\bar{\Omega}\setminus\mathcal{S})$ as $p\to+\infty$.

Let $r>0$ small enough such that $B_{r}(x_{i})\cap B_{r}(x_{j})=\emptyset$ for any $i,j\in\{1,\cdots,N\}$. Then for any $\delta\in(0,r)$, by the Green's representation formula 
    \begin{equation}
    \begin{aligned}
        \bar{u}_{p}(x)&=\int_{\Omega}G(x,y)f_{p}(y)dy\\
        &=\sum_{i=1}^{N}\int_{\Omega\cap B_{\delta}(x_{i})}G(x,y)f_{p}(y)dy+\int_{\Omega\setminus \cup_{i=1}^{N} B_{\delta}(x_{i})}G(x,y)f_{p}(y)dy\\
        &=\sum_{i=1}^{N}G(x,x_{i})\int_{\Omega\cap B_{\delta}(x_{i})}f_{p}(y)dy+\sum_{i=1}^{N}\int_{\Omega\cap B_{\delta}(x_{i})}(G(x,y)-G(x,x_{i}))f_{p}(y)dy\\
        &\quad+\int_{\Omega\setminus \cup_{i=1}^{N} B_{\delta}(x_{i})}G(x,y)f_{p}(y)dy.
    \end{aligned}        
    \end{equation}
Furthermore, by the continuity of $G(x,\cdot)$ in $\bar{\Omega}\setminus\{x\}$, Theorem \ref{thm-1} and \eqref{eq-lemma-away-the-blow-up-point-proof-1}, we have
\begin{equation}
   \sum_{i=1}^{N}\int_{B_{\delta}(x_{i})}\left(G(x,y)-G(x,x_{i})\right)f_{p}(y)dy=o_{\delta}(1)
\end{equation}
and
\begin{equation}
    \int_{\Omega\setminus \cup_{i=1}^{N} B_{\delta}(x_{i})}G(x,y)f_{p}(y)dy=o_{p}(1).
\end{equation}
Hence \eqref{lemma-away-the-blow-up-point-1} and \eqref{lemma-away-the-blow-up-point-2} hold. Finally, we prove $\gamma_{i}>0$. Indeed, by \eqref{eq-liminf-u-p}, Proposition \ref{prop-multi-peak} and the Fatou's lemma, we get
 \begin{equation}
 \begin{aligned}
     \lim_{p\to+\infty}\int_{\Omega\cap B_{\delta}(x_{i})}f_{p}(x)dx&=\lim_{p\to+\infty}p\int_{\Omega\cap B_{\delta}(x_{i})}\int_{\Omega}\frac{u_{p}^{p+1}(y)u_{p}^{p}(x)}{|x-y|^{\alpha}}dydx\\
     &\geq \lim_{p\to+\infty}u_{p}(x_{i,p})\int_{\Omega_{i,p}\cap B_{\delta/2\varepsilon_{i,p}}(0)}\int_{\Omega_{i,p}}\frac{(1+\frac{v_{i,p}(y)}{p})^{p+1}(1+\frac{v_{i,p}(y)}{p})^{p}}{|x-y|^{\alpha}}dydx\\
     &\geq \int_{\R^{2}}\int_{\R^{2}}\frac{e^{v(y)}e^{v(x)}}{|x-y|^{\alpha}}dydx=2(4-\alpha)\pi>0.
 \end{aligned}        
    \end{equation}
Hence $\gamma_{i}>0$, and the proof is complete.
\end{proof}

\begin{Prop}\label{prop-no-boundary-blow-up}
Let $\alpha\in(0,1)$ and $\Omega\subset\R^{2}$ be a smooth bounded domain satisfying $(H_{1})$. Then $\mathcal{S}\cap\partial\Omega=\emptyset.$
\end{Prop}
\begin{proof}
By contradiction, suppose there exists $x_i \in \mathcal{S} \cap \partial \Omega$ for some $i \in {1,\cdots,N}$, and let $r > 0$ be such that $\mathcal{S} \cap B_r(x_i) = \{x_i\}$. Let $z_p := x_i + \rho_{p,\delta}\nu(x_i)$, where 
\begin{equation}
\rho_{p,\delta} := \frac{\int_{\partial\Omega \cap B_\delta(x_i)} \left(\frac{\partial u_p(x)}{\partial\nu}\right)^2 \langle x-x_{i},\nu(x)\rangle d\sigma_x}{\int_{\partial\Omega \cap B_\delta(x_i)} \left(\frac{\partial u_p(x)}{\partial\nu}\right)^2 \langle \nu(x_i),\nu(x)\rangle d\sigma_x}
\end{equation}
with $\delta \ll r$ chosen such that $a_1 \leq \langle \nu(x_i),\nu(x)\rangle \leq 1$ for some given $0 < a_1 < 1$ (to be chosen later) and for all $x \in \partial\Omega \cap B_\delta(x_i)$. Then we have $|\rho_{p,\delta}|\leq \frac{\delta}{a_{1}}$ and
    \begin{equation}\label{prop-no-boundary-blow-up-proof-2}
        \int_{\partial\Omega\cap B_{\delta}(x_{i})}\left(\frac{\partial u_{p}(x)}{\partial\nu}\right)^{2}\langle x-z_{p},\nu(x)\rangle d\sigma_{x}=0.
    \end{equation}
Applying the local Pohozaev identity \eqref{Pohozaev 1} to $\Omega' = \Omega \cap B_\delta(x_i)$ with $z = z_p$ and multiplying by $p^2$ yields
\begin{equation}
    \begin{split}
    &\frac{(4-\alpha)p^{2}}{2(p+1)}\int_{\Omega\cap B_{\delta}(x_{i})}\int_{\Omega}\frac{u_{p}^{p+1}(y)u_{p}^{p+1}(x)}
{|x-y|^{\alpha}}dydx\\
&\quad+\frac{\alpha p^{2}}{2(p+1)}\int_{\Omega\cap B_{\delta}(x_{i})}\int_{\Omega\setminus B_{\delta}(x_{i})}\frac{u_{p}^{p+1}(y)u_{p}^{p+1}(x)}
{|x-y|^{\alpha}}dydx\\
&\quad-\frac{\alpha p^{2}}{p+1}\int_{\Omega\cap B_{\delta}(x_{i})}\int_{\Omega\setminus B_{\delta}(x_{i})}\langle x-z_{p},x-y\rangle\frac{u_{p}^{p+1}(y)u_{p}^{p+1}(x)}
{|x-y|^{\alpha+2}}dy dx\\
        &=-\frac{1}{2}\int_{\partial (\Omega\cap B_{\delta}(x_{i}))}|p\nabla u_{p}|^2\big\langle x-z_{p},\nu(x)\big\rangle d\sigma_{x}\\
        &\quad+p^{2}\int_{\partial(\Omega\cap B_{\delta}(x_{i}))}\frac{\partial u_{p}}{\partial\nu}\big\langle x-z_{p},\nabla u_{p}(x)\big\rangle d\sigma_{x}\\
        &\quad+\frac{p^{2}}{p+1}\int_{\partial  (\Omega\cap B_{\delta}(x_{i}))}\int_{ \Omega}\frac{u_{p}^{p+1}(y)u_{p}^{p+1}(x)}
{|x-y|^{\alpha}}\big\langle x-z_{p},\nu(x)\big\rangle dy d\sigma_{x}.
    \end{split}
\end{equation}
Notice that $u_{p}=0$ on $\partial \Omega$, then $|\frac{\partial u_{p}}{\partial \nu}|=|\nabla u_{p}|$ on $\partial \Omega$. Moreover using  \eqref{prop-no-boundary-blow-up-proof-2}, we have
    \begin{equation}\label{prop-no-boundary-blow-up-proof-4}
    \begin{split}    
&\frac{(4-\alpha)p^{2}}{2(p+1)}\int_{\Omega\cap B_{\delta}(x_{i})}\int_{\Omega}\frac{u_{p}^{p+1}(y)u_{p}^{p+1}(x)}
{|x-y|^{\alpha}}dydx\\
&\quad+\frac{\alpha p^{2}}{2(p+1)}\int_{\Omega\cap B_{\delta}(x_{i})}\int_{\Omega\setminus B_{\delta}(x_{i})}\frac{u_{p}^{p+1}(y)u_{p}^{p+1}(x)}
{|x-y|^{\alpha}}dydx\\
&\quad-\frac{\alpha p^{2}}{p+1}\int_{\Omega\cap B_{\delta}(x_{i})}\int_{\Omega\setminus B_{\delta}(x_{i})}\langle x-z_{p},x-y\rangle\frac{u_{p}^{p+1}(y)u_{p}^{p+1}(x)}
{|x-y|^{\alpha+2}}dy dx\\
&=-\frac{1}{2}\int_{\Omega\cap \partial B_{\delta}(x_{i})}|p\nabla u_{p}(x)|^{2}\langle x-z_{p},\nu(x)\rangle d\sigma_{x}\\
&\quad+\int_{\Omega\cap \partial B_{\delta}(x_{i})}\langle x-z_{p},p\nabla u_{p}(x)\rangle\langle p\nabla u_{p}(x),\nu(x)d\sigma_{x}\\
&\quad+\frac{p^{2}}{p+1}\int_{\Omega\cap \partial B_{\delta}(x_{i})}\int_{ \Omega}\frac{u_{p}^{p+1}(y)u_{p}^{p+1}(x)}
{|x-y|^{\alpha}}\big\langle x-z_{p},\nu(x)\big\rangle dy d\sigma_{x}.
    \end{split}
\end{equation}

We now estimate the left-hand side (LHS) of \eqref{prop-no-boundary-blow-up-proof-4}. The third term can be bounded as follows
\begin{equation}\label{prop-no-boundary-blow-up-proof-5}
    \begin{aligned}
        &\frac{\alpha p^{2}}{p+1}\left|\int_{\Omega\cap B_{\delta}(x_{i})}\int_{\Omega\setminus B_{\delta}(x_{i})}\langle x-z_{p},x-y\rangle\frac{u_{p}^{p+1}(y)u_{p}^{p+1}(x)}
{|x-y|^{\alpha+2}}dy dx\right|\\
&\leq \frac{(1+a_{1})\delta}{a_{1}}\frac{\alpha p^{2}}{p+1}\int_{\Omega\cap B_{\delta}(x_{i})}\int_{\Omega\setminus B_{\delta}(x_{i})}\frac{u_{p}^{p+1}(y)u_{p}^{p+1}(x)}
{|x-y|^{\alpha+1}}dy dx\\
&=\frac{(1+a_{1})\delta}{a_{1}}\frac{\alpha p^{2}}{p+1}\int_{\Omega\cap B_{\delta/a_{2}}(x_{i})}\int_{\Omega\setminus B_{\delta}(x_{i})}\frac{u_{p}^{p+1}(y)u_{p}^{p+1}(x)}{|x-y|^{\alpha+1}}dy dx\\
&\quad+\frac{(1+a_{1})\delta}{a_{1}}\frac{\alpha p^{2}}{p+1}\int_{\Omega\cap (B_{\delta}(x_{i})\setminus B_{\delta/a_{2}}(x_{i}))}\int_{\Omega\setminus B_{\delta}(x_{i})}\frac{u_{p}^{p+1}(y)u_{p}^{p+1}(x)}{|x-y|^{\alpha+1}}dy dx
    \end{aligned}
\end{equation}
for some constant $a_{2}>1$ to be determined later. Moreover, using \eqref{eq-l-infty-bounded}, the HLS inequality, \eqref{estimate of u-p-p+1}, $(\mathbf{\mathcal{P}_{3}^{N}})$ and \eqref{prop-multi-peak-eq-1},  we have
\begin{equation}\label{prop-no-boundary-blow-up-proof-6}
    \begin{aligned}
     &\frac{\alpha p^{2}}{p+1}\left|\int_{\Omega\cap B_{\delta}(x_{i})}\int_{\Omega\setminus B_{\delta}(x_{i})}\langle x-z_{p},x-y\rangle\frac{u_{p}^{p+1}(y)u_{p}^{p+1}(x)}
{|x-y|^{\alpha+2}}dy dx\right|\\
        &\leq \frac{(1+a_{1})a_{2}}{a_{1}(a_{2}-1)}\frac{\alpha p^{2}}{p+1}\int_{\Omega\cap B_{\delta/a_{2}}(x_{i})}\int_{\Omega\setminus B_{\delta}(x_{i})}\frac{u_{p}^{p+1}(y)u_{p}^{p+1}(x)}{|x-y|^{\alpha}}dy dx\\
&\quad+C(\alpha) \frac{(1+a_{1})\delta}{a_{1}}\|u_{p}\|_{L^{\infty}(\Omega\cap (B_{\delta}(x_{i})\setminus B_{\delta/a_{2}}(x_{i}))}\left(\int_{B_{\delta}(x_{i})\setminus B_{\delta/a_{2}}(x_{i}) }(pu_{p}^{2p})^{\frac{2}{2-\alpha}}(y)dy\right)^{\frac{2-\alpha}{4}}\\
&\leq \frac{(1+a_{1})a_{2}}{a_{1}(a_{2}-1)}\frac{\alpha p^{2}}{p+1}\int_{\Omega\cap B_{\delta}(x_{i})}\int_{\Omega\setminus B_{\delta}(x_{i})}\frac{u_{p}^{p+1}(y)u_{p}^{p+1}(x)}{|x-y|^{\alpha}}dy dx\\
&\quad+ C(\alpha) \frac{(1+a_{1})a_{2}}{a_{1}}\|u_{p}\|_{L^{\infty}(\Omega\cap (B_{\delta}(x_{i})\setminus B_{\delta/a_{2}}(x_{i}))}\\
&=\frac{(1+a_{1})a_{2}}{a_{1}(a_{2}-1)}\frac{\alpha p^{2}}{p+1}\int_{\Omega\cap B_{\delta}(x_{i})}\int_{\Omega\setminus B_{\delta}(x_{i})}\frac{u_{p}^{p+1}(y)u_{p}^{p+1}(x)}{|x-y|^{\alpha}}dy dx+o_{p}(1).
    \end{aligned}
\end{equation}
Hence, combining \eqref{prop-no-boundary-blow-up-proof-5} and \eqref{prop-no-boundary-blow-up-proof-5} together
\begin{equation}
    \begin{aligned}      
(\text{LHS})&\geq\frac{(4-\alpha)p^{2}}{2(p+1)}\int_{\Omega\cap B_{\delta}(x_{i})}\int_{\Omega}\frac{u_{p}^{p+1}(y)u_{p}^{p+1}(x)}
{|x-y|^{\alpha}}dydx\\
&\quad-\left(\frac{2(1+a_{1})a_{2}}{a_{1}(a_{2}-1)}-1\right)\frac{\alpha p^{2}}{2(p+1)}\int_{\Omega\cap B_{\delta}(x_{i})}\int_{\Omega\setminus B_{\delta}(x_{i})}\frac{u_{p}^{p+1}(y)u_{p}^{p+1}(x)}
{|x-y|^{\alpha}}dydx+o_{p}(1)\\
&\geq\left(2-\frac{(1+a_{1})a_{2}}{a_{1}(a_{2}-1)}\alpha\right) \frac{p^{2}}{p+1}\int_{\Omega\cap B_{\delta}(x_{i})}\int_{\Omega}\frac{u_{p}^{p+1}(y)u_{p}^{p+1}(x)}
{|x-y|^{\alpha}}dydx+o_{p}(1)\\
    \end{aligned}
\end{equation}
Notice that $\alpha<1$ and $\lim_{(a_{1},a_{2})\to(1,+\infty)}\frac{2a_{1}(a_{2}-1)}{a_{2}(1+a_{1})}=1$, then $\alpha<\frac{2a_{1}(a_{2}-1)}{a_{2}(1+a_{1})}$ and so $2-\frac{(1+a_{1})a_{2}}{a_{1}(a_{2}-1)}\alpha>0$ for any $1-a_{1}>0$ small enough and any $a_{2}>1$ large enough. On the other hand, by Proposition \ref{prop-equiv-blow-up-set}, we have that for all $\delta<r$ and for all $p_{0}>1$, there exists $p>p_{0}$ such that
 \begin{equation}
        p\int_{\Omega\cap B_{\delta}(x_{i})}\int_{\Omega}\frac{u_{p}^{p+1}(y)u_{p}^{p+1}(x)}{|x-y|^{\alpha}}dydx\geq 1.
    \end{equation}
Hence the left-hand side (LHS) of \eqref{prop-no-boundary-blow-up-proof-4} can be estimated as follows
\begin{equation}\label{prop-no-boundary-blow-up-proof-9}
    \begin{aligned}
        \lim_{\delta\to 0}\lim_{p\to+\infty}\text{LHS}\geq 2-\frac{(1+a_{1})a_{2}}{a_{1}(a_{2}-1)}\alpha>0.
    \end{aligned}
\end{equation}   

Next, we estimate the right-hand side (RHS) of \eqref{prop-no-boundary-blow-up-proof-4}. Since $x_{i}\in\partial\Omega$, then by Lemma \ref{lemma-out-the-blow-up-point}, \eqref{defin-green-function} and \cite[Lemma A.2]{DAprile2013CPDE}
\begin{equation}
    \begin{aligned}
        &-\frac{1}{2}\int_{\Omega\cap \partial B_{\delta}(x_{i})}|p\nabla u_{p}(x)|^{2}\langle x-z_{p},\nu(x)\rangle d\sigma_{x}\\
&\quad+\int_{\Omega\cap \partial B_{\delta}(x_{i})}\langle x-z_{p},p\nabla u_{p}(x)\rangle\langle p\nabla u_{p}(x),\nu(x)d\sigma_{x}\\
 &\leq C\int_{\Omega\cap \partial B_{\delta}(x_{i})} |x-z_{p}| d\sigma_{x}=O(\delta^{2}).
    \end{aligned}
\end{equation}
Moreover, by Theorem \ref{thm-1} and \eqref{prop-multi-peak-eq-1}, we get
\begin{equation}
    \begin{aligned}
&\frac{p^{2}}{p+1}\int_{\Omega\cap \partial B_{\delta}(x_{i})}\int_{ \Omega}\frac{u_{p}^{p+1}(y)u_{p}^{p+1}(x)}
{|x-y|^{\alpha}}\big\langle x-z_{p},\nu(x)\big\rangle dy d\sigma_{x}\\
&\leq C\delta \|u_{p}\|_{L^{\infty}(\Omega\cap \partial B_{\delta}(x_{i}))}p\int_{\Omega\cap \partial B_{\delta}(x_{i})}\int_{ \Omega}\frac{u_{p}^{p+1}(y)u_{p}^{p}(x)}
{|x-y|^{\alpha}} dy d\sigma_{x}=o_{p}(1)O(\delta).
    \end{aligned}
\end{equation}
Hence the right-hand side (LHS) of \eqref{prop-no-boundary-blow-up-proof-4} can be estimated as follows
\begin{equation}
    \lim_{\delta\to 0}\lim_{p\to+\infty}\text{RHS}=0,
\end{equation}
a contraction to \eqref{prop-no-boundary-blow-up-proof-9}. 
\end{proof}

Lemma \ref{lem-n-2} and Proposition \ref{prop-no-boundary-blow-up} imply that $\mathcal{S}\cap\partial \Omega=\emptyset$ and $\#\mathcal{S}\leq 2$. Next, we can further prove that $\#\mathcal{S}=1$ and hence $\mathcal{S}=\{x_{0}\}$ with $x_{0}\in\Omega$.

\begin{Lem}\label{lemma for estimate L-0}
Let $\alpha\in(0,1)$ and $\Omega\subset\R^{2}$ be a smooth bounded domain satisfying $(H_{1})$.
Then it holds that $L_{0}:= \limsup_{p \to +\infty} \int_{\Omega} f_{p}(x) \, dx \leq 2(4-\alpha)\pi e$. 
Furthermore, under the additional assumption $(H_{2})$ on the domain $\Omega$, the strict inequality $L_{0} < 2(4-\alpha)\pi e$ holds.
\end{Lem}
\begin{proof}
 First, by the H\"{o}lder inequality, Corollary \ref{cor energy}, the HLS inequality and \eqref{estimate of u-p-p+1}, we get
 \begin{equation}
 \begin{aligned}
     L_{0}&=\limsup_{p\to+\infty}p\int_{\Omega}\int_{\Omega}\frac{u_{p}^{p+1}(y)u_{p}^{p}(x)}{|x-y|^{\alpha}}dydx\\
   &\leq   \limsup_{p\to+\infty}\left(p\int_{\Omega}\int_{\Omega}\frac{u_{p}^{p+1}(y)u_{p}^{p+1}(x)}{|x-y|^{\alpha}}dydx\right)^{\frac{p}{p+1}}\left(p\int_{\Omega}\int_{\Omega}\frac{u_{p}^{p+1}(y)}{|x-y|^{\alpha}}dydx\right)^{\frac{1}{p+1}}\\
   &\leq2(4-\alpha)\pi e\cdot\limsup_{p\to+\infty}\left(C p\left(\int_{\Omega}u_{p}^{\frac{4p}{4-\alpha}}dy\right)^{\frac{4-\alpha}{4}}\right)^{\frac{1}{p+1}}\\
   &=2(4-\alpha)\pi e.
 \end{aligned}   
 \end{equation}
 Next, we integrate the equation \eqref{slightly subcritical choquard equation} on $\Omega$ directly
 \begin{equation}
     \int_{\partial\Omega}\frac{\partial u_{p}}{\partial \nu}d\sigma_{x}=\int_{\Omega}\int_{\Omega}\frac{u_{p}^{p+1}(y)u_{p}^{p}(x)}{|x-y|^{\alpha}}dydx,
 \end{equation}
 which together with the H\"{o}lder inequality and the Pohozaev identity \eqref{Pohozaev 3}, we have
 \begin{equation}
     \begin{aligned}
         \int_{\Omega}\int_{\Omega}&\frac{u_{p}^{p+1}(y)u_{p}^{p}(x)}{|x-y|^{\alpha}}dydx\\
         &\leq \left(\int_{\partial\Omega}\frac{1}{\langle x,\nu\rangle}d\sigma_{x}\right)^{\frac{1}{2}}\left(\int_{\partial\Omega}\langle x,\nu\rangle\left(\frac{\partial u_{p}}{\partial \nu}\right)^{2}d\sigma_{x}\right)^{\frac{1}{2}}\\
         &=\left(\int_{\partial\Omega}\frac{1}{\langle x,\nu\rangle}d\sigma_{x}\right)^{\frac{1}{2}}\left(\frac{(4-\alpha)}{p+1}\int_{\Omega}\int_{\Omega}\frac{u_{p}^{p+1}(y)u_{p}^{p+1}(x)}
{|x-y|^{\alpha}}dydx\right)^{\frac{1}{2}}.
     \end{aligned}
 \end{equation}
 Then by assumption $(H_{2})$ and Corollary \ref{cor energy}, we have
 \begin{equation}
     \begin{aligned}
         L_{0}&=\limsup_{p\to+\infty}p\int_{\Omega}\int_{\Omega}\frac{u_{p}^{p+1}(y)u_{p}^{p}(x)}{|x-y|^{\alpha}}dydx\\
         &\leq \limsup_{p\to+\infty}\left(\int_{\partial\Omega}\frac{1}{\langle x,\nu\rangle}d\sigma_{x}\right)^{\frac{1}{2}}\left(\frac{(4-\alpha)p^{2}}{p+1}\int_{\Omega}\int_{\Omega}\frac{u_{p}^{p+1}(y)u_{p}^{p+1}(x)}
{|x-y|^{\alpha}}dydx\right)^{\frac{1}{2}}\\
&<(2\pi e)^{\frac{1}{2}}(2(4-\alpha)^{2}\pi e)^{\frac{1}{2}}=2(4-\alpha)\pi e.
     \end{aligned}
 \end{equation}
 This completes the proof.
\end{proof}

Thanks to the result in Lemma \ref{lemma for estimate L-0}, we can assume that
\begin{equation}
    f_{p}\weakto \mu\text{~weakly in~}\mathcal{M}(\Omega) \text{~as~}p\to+\infty,
\end{equation}
where $\mathcal{M}(\Omega)$ is the space of Radom measure. In particular, we have $\mu(\Omega)\leq L_{0}$.

For any $\delta>0$, we say a point $x_{*}\in\Omega$ to be a $\delta$-regular point with respect to $\mu$, if there exists $\varphi\in C_{c}(\Omega)$ satisfying $0\leq\varphi\leq 1$, $\varphi\equiv1$ near $x_{*}$ such that
\begin{equation}
    \int_{\Omega}\varphi d\mu<\frac{(4-\alpha)\pi}{\frac{1}{e}+2\delta}.
\end{equation}
Moreover, we denote
\begin{equation}
    \Sigma_{\mu}(\delta):=\left\{x\in\Omega:x\text{~is not a $\delta$-regular point with respect to~$\mu$}\right\}.
\end{equation}

\begin{Lem}\label{lemma-propertity-of-sigma}
Let $\alpha\in(0,1)$ and $\Omega\subset\R^{2}$ be a smooth bounded domain satisfying $(H_{1})$. For any $\delta>0$ and $x_{*}\in\Omega$, we have that $x_{*}\in\Sigma_{\mu}(\delta)$ if and only if for any $R>0$ such that $B_{R}(x_{*})\subset\Omega$, it holds $\|\bar{u}_{p}\|_{L^{\infty}(B_{R}(x_{*}))}\to+\infty$ as $p\to+\infty$.
\end{Lem}
\begin{proof}
    First, given any $x_{*}\in\Sigma_{\mu}(\delta)$, it holds that $\|\bar{u}_{p}\|_{L^{\infty}(B_{R}(x_{*}))}\to+\infty$ as $p\to+\infty$, for any $R>0$ such that $B_{R}(x_{*})\subset\Omega$. Otherwise, there exists $R_{1}>0$ such that $B_{R_{1}}(x_{*})\subset\Omega$ and $\|\bar{u}_{p}\|_{L^{\infty}(B_{R_{1}}(x_{*}))}\leq C$ as $p\to+\infty$. Then by the HLS inequality and \eqref{estimate of u-p-p+1} , we have
    \begin{equation}
    \begin{aligned}
        \int_{B_{R_{1}}(x_{*})}f_{p}(x)dx&=p\int_{B_{R_{1}}(x_{*})}\int_{\Omega}\frac{u_{p}^{p+1}(y)u_{p}^{p}(x)}{|x-y|^{\alpha}}dydx\\
        &\leq \frac{C}{p}\left(p\int_{B_{R_{1}}(x_{*})}\int_{\Omega}\frac{u_{p}^{p+1}(y)u_{p}^{p-1}(x)}{|x-y|^{\alpha}}dydx\right)\\
        &\leq \frac{C}{p} \left(\int_{\Omega}p^{\frac{2}{4-\alpha}}u_{p}^{\frac{4p}{4-\alpha}}(x)dx\right)^{\frac{4-\alpha}{2}}\\
        &\leq \frac{C}{p}\to0\text{~as~}p\to+\infty
    \end{aligned}        
    \end{equation}
Hence $x_{*}$ is a $\delta$-regular point with respect to $\mu$, which contradicts the assumption that $x_{*} \in \Sigma_{\mu}(\delta)$.

Next, we need to prove that if $x_{*}\not\in\Sigma_{\mu}(\delta)$, then there exists a $R_{0}>0$ such that $B_{R_{0}}(x_{*})\subset\Omega$ and $\|\bar{u}_{p}\|_{L^{\infty}(B_{R_{0}}(x_{*}))}\leq C$ as $p\to+\infty$. Since $\|f_{p}\|_{L^{1}(\Omega)}\leq C$, by applying the elliptic $L^{p}$ estimate with the duality argument, we have that $\bar{u}_{p}$ are uniformly bounded in $W^{1,q}(\Omega)$ for any $1\leq q<2$, hence $\bar{u}_{p}$ are uniformly bounded in $L^{1}(\Omega)$. Now, we claim that there exist $R_{0}>0$ and $\delta_{0}>0$ small enough such that
\begin{equation}\label{claim}
    \|f_{p}\|_{L^{1+\delta_{0}}(B_{2R_{0}}(x_{*}))}\leq C,\text{~~as~~}p\to+\infty.
\end{equation}
Then, we can apply the Harnack inequality (\cite[Theorem 8.17]{gilbarg1977elliptic}) to obtain 
\begin{equation}
    \|\bar{u}_{p}\|_{L^{\infty}(B_{R_{0}}(x_{*}))}\leq C\left(\|\bar{u}_{p}\|_{L^{1}(B_{2R_{0}}(x_{*}))}+ \|f_{p}\|_{L^{1+\delta_{0}}(B_{2R_{0}}(x_{*}))}\right)\leq C.
\end{equation}
Now, we show the claim \eqref{claim} holds. Since $x_{*}\not\in\Sigma_{\mu}(\delta)$, then there exists $R_{1}>0$ such that $B_{4R_{1}}(x_{*})\subset\Omega$ and
\begin{equation}\label{eq-f-p-1}
    \int_{B_{4R_{1}}(x_{*})}f_{p}(x)dx<\frac{(4-\alpha)\pi}{\frac{1}{e}+\delta}\text{~~for any~}p\text{~large enough~.}
\end{equation}
Split $\bar{u}_{p}$ into two parts, $\bar{u}_{p}=\bar{u}_{p,1}+\bar{u}_{p,2}$ such that
\begin{align}
    \begin{cases}
        -\Delta\bar{u}_{p,1} = f_{p}, & \text{in } B_{4R_{1}}(x_{*}), \\
        \quad\ \ \bar{u}_{p,1} = 0, & \text{on } \partial B_{4R_{1}}(x_{*}),
    \end{cases}
    \qquad \text{and} \qquad
    \begin{cases}
        -\Delta\bar{u}_{p,2} = 0, & \text{in } B_{4R_{1}}(x_{*}), \\
        \quad\ \ \bar{u}_{p,2} = \bar{u}_{p}, & \text{on } \partial B_{4R_{1}}(x_{*}).
    \end{cases}
\end{align}
By applying the maximum principle, we have $0<\bar{u}_{p,1}<\bar{u}_{p}$ and  $0<\bar{u}_{p,2}<\bar{u}_{p}$ in $B_{4R_{1}(x_{*})}$. For $\bar{u}_{p,2}$, by the mean value theorem for harmonic functions
\begin{equation}\label{eq-u-p-2}
    \begin{aligned}
        \|\bar{u}_{p,2}\|_{L^{\infty}(B_{2R_{1}}(x_{*}))}\leq C\|\bar{u}_{p,2}\|_{L^{1}(B_{4R_{1}}(x_{*}))}\leq C\|\bar{u}_{p}\|_{L^{1}(\Omega)}\leq C.
    \end{aligned}
\end{equation}
For $\bar{u}_{p,1}$, by Lemma \ref{lem ren4.3}, we have
\begin{equation}\label{eq-estimate-for-u-p-1}
     \int_{B_{4R_{1}}(x_{*})}e^{\frac{\gamma\bar{u}_{p,1}(x)}{\|f_{p}\|_{L^{1}(B_{4R_{1}}(x_{*}))}}}\leq C_{\gamma},\text{~~for any~}\gamma\in(0,4\pi).
\end{equation}
Notice that by the H\"older inequality, the HLS inequality, Proposition \ref{prop L-infty estimate} and \eqref{estimate of u-p-p+1}, we get
\begin{equation}
    \begin{aligned}
        &\left(\int_{B_{R_{1}}(x_{*})}f_{p}^{1+\delta_{0}}(x)dx\right)^{\frac{1}{1+\delta_{0}}}\\
        &=p\left(\int_{B_{R_{1}}(x_{*})}\left(\int_{\Omega}\frac{u_{p}^{p+1}(y)u_{p}^{p}(x)}{|x-y|^{\alpha}}dy\right)^{1+\delta_{0}}dx\right)^{\frac{1}{1+\delta_{0}}}\\
        &=p\left(\int_{B_{R_{1}}(x_{*})}\left(\int_{B_{2R_{1}}(x_{*})}\frac{u_{p}^{p+1}(y)u_{p}^{p}(x)}{|x-y|^{\alpha}}dy+\int_{\Omega\setminus B_{2R_{1}}(x_{*})}\frac{u_{p}^{p+1}(y)u_{p}^{p}(x)}{|x-y|^{\alpha}}dy\right)^{1+\delta_{0}}dx\right)^{\frac{1}{1+\delta_{0}}}\\
         &\leq p\left(\int_{B_{R_{1}}(x_{*})}\left(\int_{B_{2R_{1}}(x_{*})}\frac{u_{p}^{p+1}(y)u_{p}^{p}(x)}{|x-y|^{\alpha}}dy+\frac{C}{R_{1}^{\alpha}p^{\frac{1}{2}}}u_{p}^{p}(x)\right)^{1+\delta_{0}}dx\right)^{\frac{1}{1+\delta_{0}}}\\
         &\leq  p\left(\int_{B_{R_{1}}(x_{*})}\left(\int_{B_{2R_{1}}(x_{*})}\frac{u_{p}^{p+1}(y)u_{p}^{p}(x)}{|x-y|^{\alpha}}dy\right)^{1+\delta_{0}}dx\right)^{\frac{1}{1+\delta_{0}}}\\
         &\quad+\frac{C}{R_{1}^{\alpha}p^{\frac{1}{2}}}\left(\int_{B_{R_{1}}(x_{*})}(p^{\frac{1}{p}}u_{p})^{p(1+\delta_{0})}(x)dx\right)^{\frac{1}{1+\delta_{0}}}
    \end{aligned}
\end{equation}
and
\begin{equation}
    \begin{aligned}
    &p\left(\int_{B_{R_{1}}(x_{*})}\left(\int_{B_{2R_{1}}(x_{*})}\frac{u_{p}^{p+1}(y)u_{p}^{p}(x)}{|x-y|^{\alpha}}dy\right)^{1+\delta_{0}}dx\right)^{\frac{1}{1+\delta_{0}}}\\
        &\leq Cp\left(\int_{B_{R_{1}}(x_{*})}\left(\int_{B_{2R_{1}}(x_{*})}\frac{u_{p}^{p}(y)u_{p}^{p}(x)}{|x-y|^{\alpha}}dy\right)^{1+\delta_{0}}dx\right)^{\frac{1}{1+\delta_{0}}}\\
        &\leq C \left(\int_{B_{2R_{1}}(x_{*})}(p^{\frac{1}{2p}}u_{p})^{pq}(x)dx\right)^{\frac{2}{q}},
    \end{aligned}
\end{equation}
where $q=\frac{4(1+\delta_{0})}{2+(2-\alpha)(1+\delta_{0})}$.
Since $\log x\leq \frac{x}{e}$ for any $x>0$, then
\begin{equation}\label{eq=p-1-p-u-p}
    (p^{\frac{1}{2p}}u_{p}(x))^{p}\leq  (p^{\frac{1}{p}}u_{p}(x))^{p}\leq e^{p\log(p^{\frac{1}{p}}u_{p}(x))}\leq e^{(\frac{1}{e}+\frac{\delta}{2})\bar{u}_{p}(x)},
\end{equation}
for any $\delta>0$, $x\in\Omega$ and $p$ large enough. 
Take $\delta_{0}>0$ small enough such that 
\begin{equation}\label{eq=p-1-p-u-p-1}
  0<\frac{(4-\alpha)(1+\delta_{0})}{4}\frac{1+\frac{\delta}{2}e}{1+\delta e}<\frac{(4-\alpha)(1+\delta_{0})}{2+(2-\alpha)(1+\delta_{0})}\frac{1+\frac{\delta}{2}e}{1+\delta e}<1.
\end{equation}
Then using  \eqref{eq-f-p-1}, \eqref{eq-u-p-2}, \eqref{eq-estimate-for-u-p-1} and \eqref{eq=p-1-p-u-p}, we have
\begin{equation}
    \begin{aligned}
         \int_{B_{2R_{1}}(x_{*})}(p^{\frac{1}{2p}}u_{p})^{pq}(x)dx&\leq \int_{B_{2R_{1}}(x_{*})}e^{(\frac{1}{e}+\frac{\delta}{2})q
         \bar{u}_{p}(x)}dx\\
         &\leq C\int_{B_{2R_{1}}(x_{*})}e^{(\frac{1}{e}+\frac{\delta}{2})q
         \bar{u}_{p,1}(x)}dx\\
         &\leq C\int_{B_{4R_{1}}(x_{*})}e^{q\frac{4-\alpha}{4}\frac{1+\frac{\delta}{2}e}{1+\delta e}
         \frac{4\pi \bar{u}_{p,1}(x)}{\|f_{p}\|_{L^{1}(B_{4R_{1}}(x_{*}))}}}dx\\
         &\leq C
    \end{aligned}
\end{equation}
and
\begin{equation}
    \begin{aligned}
        \int_{B_{R_{1}}(x_{*})}(p^{\frac{1}{p}}u_{p})^{p(1+\delta_{0})}(x)dx&\leq \int_{B_{R_{1}}(x_{*})}e^{(\frac{1}{e}+\frac{\delta}{2})(1+\delta_{0})\bar{u}_{p}(x)}dx\\
        &\leq C\int_{B_{2R_{1}}(x_{*})}e^{(\frac{1}{e}+\frac{\delta}{2})(1+\delta_{0})\bar{u}_{p,1}(x)}dx\\
        &\leq C\int_{B_{4R_{1}}(x_{*})}e^{\frac{(4-\alpha)(1+\delta_{0})}{4}\frac{1+\frac{\delta}{2}e}{1+\delta e}\frac{4\pi\bar{u}_{p}(x)}{\|f_{p}\|_{L^{1}(B_{4R_{1}}(x_{*}))}}}dx\\
        &\leq C.
    \end{aligned}
\end{equation}
Combining the above estimates together
\begin{equation}
    \left(\int_{B_{R_{1}}(x_{*})}f_{p}^{1+\delta_{0}}(x)dx\right)^{\frac{1}{1+\delta_{0}}}\leq C,
\end{equation}
and the claim \eqref{claim} holds. This completes the proof.
\end{proof}



\begin{Cor}\label{cor-only-one-blow-up-point}
Let $\alpha\in(0,1)$ and $\Omega\subset\R^{2}$ be a smooth bounded domain satisfying $(H_{1})$ and $(H_{2})$.  Then, it holds that $\Sigma_{\mu}(\delta)=\mathcal{S}=\{x_{0}\}$.
\end{Cor}
\begin{proof}
By Lemma \ref{lemma-propertity-of-sigma}, we have $\Sigma_{\mu}(\delta) = \mathcal{S}$ for all sufficiently small $\delta > 0$. Moreover, combining Lemma \ref{lemma for estimate L-0} with the definition of $\delta$-regular points yields
\begin{equation}
    L_{0} \geq \mu(\Omega) \geq \#\mathcal{S} \cdot \frac{(4-\alpha)\pi}{\frac{1}{e} + \delta},
\end{equation}
where $\#\mathcal{S}$ denotes the cardinality of $\mathcal{S}$. Consequently, 
\begin{equation}
    1 \leq \#\mathcal{S} \leq \frac{L_{0}(1 + \delta e)}{(4-\alpha)\pi e} < 2
\end{equation}
for any $\delta > 0$ sufficiently small, which implies $\#\mathcal{S} = 1$.
\end{proof}

\subsection{Decay estimate}
Throughout this subsection, we always assume that $\alpha\in(0,1)$ and $\Omega\subset\R^{2}$ is a smooth bounded domain satisfying $(H_{1})$ and $(H_{2})$.
\begin{Lem}\label{lemma-outside-the-non-regular-point}
    For any compact set $K\subset{\Omega}\setminus\{x_{0}\}$, there exists $C_{K}>0$  such that
    \begin{equation}
        \|u_{p}\|_{L^{\infty}(K)}\leq \frac{C_{K}}{p}\text{~and~} \int_{K}f_{p}(x)dx\leq\frac{C_{K}}{p},
    \end{equation}
for any $p$ large enough.
\end{Lem}
\begin{proof}
    By Lemma \ref{lemma-propertity-of-sigma}, we get $ \|\bar{u}_{p}\|_{L^{\infty}(K)}\leq C_{K}$ for any $p$ large and so $ \|{u}_{p}\|_{L^{\infty}(K)}\leq \frac{C_{K}}{p}$ for any $p$ large.
Moreover, by the HLS inequality and \eqref{estimate of u-p-p+1}, we obtain that
\begin{equation}
\begin{aligned}
    \int_{K}f_{p}(x)dx&=p\int_{K}\int_{\Omega}\frac{u_{p}^{p+1}(y)u_{p}^{p}(x)}{|x-y|^{\alpha}}dydx\\
    &\leq \frac{C_{K}}{p}p\int_{K}\int_{\Omega}\frac{u_{p}^{p+1}(y)u_{p}^{p-1}(x)}{|x-y|^{\alpha}}dydx\\
    &\leq \frac{C_{K}}{p}\left(\int_{\Omega}p^{\frac{2}{4-\alpha}}u_{p}^{\frac{4p}{4-\alpha}}(x)dx\right)^{\frac{4-\alpha}{2}}\leq\frac{C_{K}}{p},
\end{aligned}    
\end{equation}
for any $p$ large. This completes the proof.
\end{proof}

\begin{Lem}\label{lemma-convergence-out-of-the-blow-up-point}
    There exists $\gamma>0$ such that
    \begin{equation}\label{eq-converge-to-the-green-function}
        \lim_{p\to+\infty}pu_{p}=\gamma G(\cdot,x_{0})\text{~~in~~}C^{2}_{loc}(\bar{\Omega}\setminus\{x_{0}\}),
    \end{equation}
    where
    \begin{equation}
        \gamma=\lim_{\delta\to0}\lim_{p\to+\infty}\int_{B_{\delta}(x_{0})}f_{p}(x)dx.
    \end{equation}
\end{Lem}
\begin{proof}
 It follows from Corollary \ref{cor-only-one-blow-up-point} and Lemma \ref{lemma-out-the-blow-up-point}.
\end{proof}

\begin{Lem}\label{lemma-convergence-to-right-hand-integral} 
Let $0 < r < \frac{1}{2}\text{dist}(x_0, \partial\Omega)$, and define
    \begin{equation}
    \begin{aligned}
         \beta_{p}:=\int_{B_{\frac{r}{\varepsilon_{p}}}(0)}\left(\int_{\Omega_{p}}\frac{(1+\frac{v_{p}(y)}{p})^{p+1}}{|x-y|^{\alpha}}dz\right)(1+\frac{v_{p}(x)}{p})^{p}dy.
    \end{aligned}
\end{equation}
where $\varepsilon_p$ and $v_p$ are defined in \eqref{defin-vare-p} and \eqref{definition of v} respectively.
Then it holds that
\begin{equation}
    \lim_{p\to+\infty}\beta_{p}=2(4-\alpha)\pi.
\end{equation}
\end{Lem}
\begin{proof}
First, by the Fatou's Lemma and Proposition \ref{prop convergence of z-p}, we have
\begin{equation}\label{eq-proof-beta-p-1}
    \begin{aligned}
         \liminf_{p\to+\infty}&\int_{B_{\frac{r}{\varepsilon_{p}}}(0)}\left(\int_{\Omega_{p}}\frac{(1+\frac{v_{p}(y)}{p})^{p+1}}{|x-y|^{\alpha}}dz\right)(1+\frac{v_{p}(x)}{p})^{p}dy\\
         &\geq \int_{\R^{2}}\int_{\R^{2}}\frac{e^{v(y)}e^{v(x)}}{|x-y|^{\alpha}}dydx=2(4-\alpha)\pi.
    \end{aligned}
\end{equation}
Notice that
    \begin{equation}\label{eq-proof-beta-p-2}
    \begin{aligned}
      &\lim_{p\to+\infty}\int_{B_{\frac{r}{\varepsilon_{p}}}(0)}\left(\int_{\Omega_{p}}\frac{(1+\frac{v_{p}(z)}{p})^{p+1}}{|y-z|^{\alpha}}dz\right)(1+\frac{v_{p}(y)}{p})^{p}dy\\
      &=\lim_{p\to+\infty}\frac{p}{u_{p}(x_{p})}\int_{B_{r}(x_{p})}\left(\int_{\Omega}\frac{u_{p}^{p+1}(z)}{|y-z|^{\alpha}}dz\right)u_{p}^{p}(y)dy\\   
    \end{aligned}
\end{equation}
and by Proposition \ref{prop L-infty estimate} and Lemma \ref{lemma-outside-the-non-regular-point}, we have 
\begin{equation}\label{eq-proof-beta-p-3}
    \begin{aligned}
        &\frac{p}{u_{p}(x_{p})}\int_{B_{r}(x_{p})}\left(\int_{\Omega}\frac{u_{p}^{p+1}(z)}{|y-z|^{\alpha}}dz\right)u_{p}^{p}(y)dy\\
        &=\frac{p}{u_{p}(x_{p})}\int_{B_{\delta}(x_{0})}\left(\int_{\Omega}\frac{u_{p}^{p+1}(z)}{|y-z|^{\alpha}}dz\right)u_{p}^{p}(y)dy\\
      &\quad+\frac{p}{u_{p}(x_{p})}\int_{B_{r}(x_{p})\setminus B_{\delta}(x_{0})}\left(\int_{\Omega}\frac{u_{p}^{p+1}(z)}{|y-z|^{\alpha}}dz\right)u_{p}^{p}(y)dy\\
      &=\frac{p}{u_{p}(x_{p})}\int_{B_{\delta}(x_{0})}\left(\int_{\Omega}\frac{u_{p}^{p+1}(z)}{|y-z|^{\alpha}}dz\right)u_{p}^{p}(y)dy+o_{p}(1)
    \end{aligned}
\end{equation}
for any $\delta\in(0,r)$ small and $p$ large. Thus it  suffices to show that 
\begin{equation}\label{eq-proof-beta-p-4}
    \begin{aligned}
        \lim_{\delta\to0}\lim_{p\to+\infty}\frac{p}{u_{p}(x_{p})}\int_{B_{\delta}(x_{0})}\left(\int_{\Omega}\frac{u_{p}^{p+1}(z)}{|y-z|^{\alpha}}dz\right)u_{p}^{p}(y)dy\leq 2(4-\alpha)\pi.
    \end{aligned}
\end{equation}

Now, applying the local Poho\v zaev identity \eqref{Pohozaev 1} to the domain $\Omega^{\prime}=B_{\delta} (x_{0})$ with $z=x_{0}$, and then multiplying both sides by $p^{2}$, we obtain
\begin{equation}\label{eq-lemma-integral-convergence-proof-4}
    \begin{split}
&\frac{(4-\alpha)p^{2}}{2(p+1)}\int_{B_{\delta}(x_{0})}\left(\int_{\Omega}\frac{u_{p}^{p+1}(y)}
{|x-y|^{\alpha}}dy\right)u_{p}^{p+1}(x)dx\\
&=-\frac{\delta}{2}\int_{\partial B_{\delta}(x_{0})}|p\nabla u_{p}|^2d\sigma_{x}+\delta\int_{\partial B_{\delta}(x_{0})}\big\langle p\nabla u_{p},\nu(x)\big\rangle^{2} d\sigma_{x}\\
        &\quad+\frac{p^{2}}{p+1}\delta\int_{\partial B_{\delta}(x_{0})}\int_{ \Omega}\frac{u_{p}^{p+1}(y)u_{p}^{p+1}(x)}
{|x-y|^{\alpha}} dy d\sigma_{x}\\
&\quad-\frac{\alpha p^{2}}{2(p+1)}\int_{B_{\delta}(x_{0})}\int_{\Omega\setminus B_{\delta}(x_{0})}\frac{u_{p}^{p+1}(y)u_{p}^{p+1}(x)}
{|x-y|^{\alpha}}dydx\\
&\quad+\frac{\alpha p^{2}}{p+1}\int_{B_{\delta}(x_{0})}\int_{\Omega\setminus B_{\delta}(x_{0})}\langle x-x_{0},x-y\rangle\frac{u_{p}^{p+1}(y)u_{p}^{p+1}(x)}
{|x-y|^{\alpha+2}}dy dx.\\
    \end{split}
\end{equation}
Next, we estimate the terms on the right-hand side of \eqref{eq-lemma-integral-convergence-proof-4}. Notice that, by \eqref{eq-green-function} and \eqref{eq-singular-part-of-the-green-function} 
\begin{equation}
        G(x,x_{0})=\frac{1}{2\pi}\log\frac{1}{|x-x_{0}|}-H(x,x_{0}).
\end{equation}
Then by the regularity of $H$, if $\delta\in(0,r)$ small enough and $x\in\overline{B_{\delta}(x_{0})}\setminus\{x_{0}\}$, we have
\begin{equation}
    \begin{aligned}
        G(x,x_{0})=\frac{1}{2\pi}\log\frac{1}{|x-x_{0}|}+O(1)              
    \end{aligned}
\end{equation}
and
\begin{equation}
    \nabla G(x,x_{0})=-\frac{1}{2\pi}\frac{x-x_{0}}{|x-x_{0}|^{2}}+O(1).
\end{equation}
Lemma \ref{lemma-convergence-out-of-the-blow-up-point} now implies that
\begin{equation}\label{eq-lemma-integral-convergence-proof-5}
    \begin{aligned}
       &-\frac{\delta}{2}\int_{\partial B_{\delta}(x_{0})}|p\nabla u_{p}|^2d\sigma_{x}+\delta\int_{\partial B_{\delta}(x_{0})}\big\langle p\nabla u_{p},\nu(x)\big\rangle^{2} d\sigma_{x}\\
       &=-\frac{\delta}{2}\int_{\partial B_{\delta}(x_{0})}\left(-\frac{\gamma}{2\pi}\frac{x-x_{0}}{|x-x_{0}|^{2}}+O(1)\right)^{2}d\sigma_{x}\\
       &\quad+\delta\int_{\partial B_{\delta}(x_{0})}\left(-\frac{\gamma}{2\pi}\frac{\langle x-x_{0},\nu(x)\rangle}{|x-x_{0}|^{2}}+O(1)\right)^{2}d\sigma_{x}\\
       &=\frac{\gamma^{2}}{4\pi}+O(\delta).
    \end{aligned}
\end{equation}
On the other hand, by \eqref{eq-thm-energy}, Lemma \ref{lemma for estimate L-0}, \eqref{prop-multi-peak-eq-1}, the HLS inequality, \eqref{estimate of u-p-p+1} and property ($\mathbf{\mathcal{P}_{3}^{1}}$), we have
\begin{equation}
    \begin{aligned}
       \frac{p^{2}}{p+1}\delta\int_{\partial B_{\delta}(x_{0})}\int_{ \Omega}\frac{u_{p}^{p+1}(y)u_{p}^{p+1}(x)}
{|x-y|^{\alpha}} dy d\sigma_{x}=O(\delta),
    \end{aligned}
\end{equation}
\begin{equation}
    \begin{aligned}
        \frac{\alpha p^{2}}{2(p+1)}&\int_{B_{\delta}(x_{0})}\int_{\Omega\setminus B_{\delta}(x_{0})}\frac{u_{p}^{p+1}(y)u_{p}^{p+1}(x)}
{|x-y|^{\alpha}}dydx\\
&\leq C\|u_{p}\|_{L^{\infty}(\Omega\setminus B_{\delta}(x_{0}))}\int_{\Omega}f_{p}(x)dx=o_{p}(1),
    \end{aligned}
\end{equation}
and
\begin{equation}\label{eq-lemma-integral-convergence-proof-6}
    \begin{aligned}
        &\frac{\alpha p^{2}}{p+1}\left|\int_{B_{\delta}(x_{0})}\int_{\Omega\setminus B_{\delta}(x_{0})}\langle x-x_{0},x-y\rangle\frac{u_{p}^{p+1}(y)u_{p}^{p+1}(x)}
{|x-y|^{\alpha+2}}dy dx\right|\\
&\leq\frac{\alpha p^{2}}{p+1}\delta\int_{B_{\delta/2}(x_{0})}\int_{\Omega\setminus B_{\delta}(x_{0})}\frac{u_{p}^{p+1}(y)u_{p}^{p+1}(x)}
{|x-y|^{\alpha+1}}dy dx\\
&\quad+\frac{\alpha p^{2}}{p+1}\delta\int_{B_{\delta}(x_{0})\setminus B_{\delta/2}(x_{0})}\int_{\Omega\setminus B_{\delta}(x_{0})}\frac{u_{p}^{p+1}(y)u_{p}^{p+1}(x)}
{|x-y|^{\alpha+1}}dy dx\\
&\leq \frac{2\alpha p^{2}}{p+1}\|u_{p}\|_{L^{\infty}(\Omega\setminus B_{\delta}(x_{0}))}\int_{B_{\delta/2}(x_{0})}\int_{\Omega}\frac{u_{p}^{p+1}(y)u_{p}^{p}(x)}
{|x-y|^{\alpha}}dy dx\\
&\quad+ \frac{Cp^{1/2}}{\delta^{\frac{2-\alpha}{2}}}\|u_{p}\|_{L^{\infty}(\Omega\setminus B_{\delta}(x_{0}))}\int_{B_{\delta}(x_{0})\setminus B_{\delta/2}(x_{0})}\int_{\Omega\setminus B_{\delta}(x_{0})}\frac{u_{p}^{p}(y)}
{|x-y|^{\alpha+1}}dy dx\\
&\leq C\|u_{p}\|_{L^{\infty}(\Omega\setminus B_{\delta}(x_{0}))}+\|u_{p}\|_{L^{\infty}(\Omega\setminus B_{\delta}(x_{0}))}\frac{Cp^{1/2}}{\delta^{\frac{2-\alpha}{2}}} |B_{\delta}(x_{0})|^{\frac{2-\alpha}{4}}\left(\int_{\Omega}u_{p}^{\frac{4p}{4-\alpha}}(y)dy\right)^{\frac{4-\alpha}{4}}\\
&\leq C\|u_{p}\|_{L^{\infty}(\Omega\setminus B_{\delta}(x_{0}))}=o_{p}(1).
    \end{aligned}
\end{equation}
Combining \eqref{eq-lemma-integral-convergence-proof-4}-\eqref{eq-lemma-integral-convergence-proof-6} together
\begin{equation}\label{eq-lemma-integral-convergence-proof-7}
    \begin{aligned}
        \frac{(4-\alpha)p^{2}}{2(p+1)}\int_{B_{\delta}(x_{0})}\left(\int_{\Omega}\frac{u_{p}^{p+1}(y)}
{|x-y|^{\alpha}}dy\right)u_{p}^{p+1}(x)dx=\frac{\gamma^{2}}{4\pi}+O(\delta)+o_{p}(1).
    \end{aligned}
\end{equation}
Let 
\begin{equation}\label{eq-proof-beta-p-11}
    \beta_{p}(\delta):=\frac{p}{u_{p}(x_{p})}\int_{B_{\delta}(x_{0})}\left(\int_{\Omega}\frac{u_{p}^{p+1}(z)}{|y-z|^{\alpha}}dz\right)u_{p}^{p}(y)dy,
\end{equation}
then using \eqref{eq-lemma-integral-convergence-proof-7}, we obtain
\begin{equation}
    \begin{aligned}
        \beta_{p}(\delta)u_{p}(x_{p})^{2}&=pu_{p}(x_{p})\int_{B_{\delta}(x_{0})}\left(\int_{\Omega}\frac{u_{p}^{p+1}(z)}{|y-z|^{\alpha}}dz\right)u_{p}^{p}(y)dy\\
        &\geq p\int_{B_{\delta}(x_{0})}\left(\int_{\Omega}\frac{u_{p}^{p+1}(z)}{|y-z|^{\alpha}}dz\right)u_{p}^{p+1}(y)dy\\
        &=\frac{\gamma^{2}}{2(4-\alpha)\pi}+O(\delta)+o_{p}(1).
    \end{aligned}
\end{equation}
Notice that
\begin{equation}\label{eq-proof-beta-p-12}
    \begin{aligned}
        \gamma=\lim_{\delta\to0}\lim_{p\to+\infty}p\int_{B_{\delta}(x_{0})}\int_{ \Omega}\frac{u_{p}^{p+1}(y)u_{p}^{p}(x)}
{|x-y|^{\alpha}} dydx=\lim_{\delta\to0}\lim_{p\to+\infty}\beta_{p}(\delta)u_{p}(x_{p}).
    \end{aligned}
\end{equation}
Thus
\begin{equation}
    \begin{aligned}
          \lim_{\delta\to0}\lim_{p\to+\infty}\beta_{p}(\delta)=\lim_{\delta\to0}\lim_{p\to+\infty}\frac{p}{u_{p}(x_{p})}\int_{B_{\delta}(x_{0})}\left(\int_{\Omega}\frac{u_{p}^{p+1}(z)}{|y-z|^{\alpha}}dz\right)u_{p}^{p}(y)dy\leq 2(4-\alpha)\pi.
    \end{aligned}
\end{equation}
This completes the proof.
\end{proof}

For the rescaling function $v_{p}$, we derive the following decay estimates, which will be used to apply the Dominated Convergence Theorem.
\begin{Prop}\label{lemma-decay-estimate}
For any $\varepsilon>0$, there exist $R_{\varepsilon}>1$, $p_{\varepsilon}>1$ and $C_{\varepsilon}>0$ such that
    \begin{equation}
        v_{p}(x)\leq \left(\frac{\beta_{p}}{2\pi}-\varepsilon\right)\log\frac{1}{|x|}+C_{\varepsilon},
    \end{equation}
for any $2R_{\varepsilon}\leq |x|\leq\frac{r}{\varepsilon_{p}}$ and $p\geq p_{\varepsilon}$. Moreover, $\beta_{p}$ is defined by
\begin{equation}\label{definition of beta-p}
    \beta_{p}:=\int_{B_{\frac{r}{\varepsilon_{p}}}(0)}\left(\int_{\Omega_{p}}\frac{(1+\frac{v_{p}(y)}{p})^{p+1}}{|x-y|^{\alpha}}dz\right)(1+\frac{v_{p}(x)}{p})^{p}dy.
\end{equation}
\end{Prop}
\begin{proof}
For any $x\in\Omega_{p}:=\frac{\Omega-x_{p}}{\varepsilon_{p}}$, by \eqref{slightly subcritical choquard equation}, Green's representation formula and the definition of $v_{p}$, we have
    \begin{equation}
        \begin{aligned}
            u_{p}(\varepsilon_{p}x+x_{p})&=\int_{\Omega}G(\varepsilon_{p}x+x_{p},y)\left(\int_{\Omega}\frac{u_{p}^{p+1}(z)}{|y-z|^{\alpha}}dz\right)u_{p}^{p}(y)dy\\
            &=\varepsilon_{p}^{4-\alpha}\int_{\Omega_{p}}\int_{\Omega_{p}}\frac{u_{p}^{p+1}(\varepsilon_{p}z+x_{p})u_{p}^{p}(\varepsilon_{p}y+x_{p})G(\varepsilon_{p}x+x_{p},\varepsilon_{p}y+x_{p})}{|y-z|^{\alpha}}dzdy\\
            &=\frac{u_{p}(x_{p})}{p}\int_{\Omega_{p}}\int_{\Omega_{p}}\frac{(1+\frac{v_{p}(z)}{p})^{p+1}(1+\frac{v_{p}(y)}{p})^{p}G(\varepsilon_{p}x+x_{p},\varepsilon_{p}y+x_{p})}{|y-z|^{\alpha}}dzdy.
        \end{aligned}
    \end{equation}
    and
    \begin{equation}
        v_{p}(x)=-p+\int_{\Omega_{p}}\int_{\Omega_{p}}\frac{(1+\frac{v_{p}(z)}{p})^{p+1}(1+\frac{v_{p}(y)}{p})^{p}G(\varepsilon_{p}x+x_{p},\varepsilon_{p}y+x_{p})}{|y-z|^{\alpha}}dzdy.
    \end{equation}
Moreover, by the decomposition of the Green function $G(x,y)$
    \begin{equation}
        \begin{aligned}
            v_{p}(x)&=v_{p}(x)-v_{p}(0)\\
            &=\int_{\Omega_{p}}\int_{\Omega_{p}}\frac{(1+\frac{v_{p}(z)}{p})^{p+1}(1+\frac{v_{p}(y)}{p})^{p}G(\varepsilon_{p}x+x_{p},\varepsilon_{p}y+x_{p})}{|y-z|^{\alpha}}dzdy\\
            &\quad-\int_{\Omega_{p}}\int_{\Omega_{p}}\frac{(1+\frac{v_{p}(z)}{p})^{p+1}(1+\frac{v_{p}(y)}{p})^{p}G(x_{p},\varepsilon_{p}y+x_{p})}{|y-z|^{\alpha}}dzdy\\
            &=\frac{1}{2\pi}\int_{\Omega_{p}}\int_{\Omega_{p}}\log\frac{|y|}{|x-y|}\cdot\frac{(1+\frac{v_{p}(z)}{p})^{p+1}(1+\frac{v_{p}(y)}{p})^{p}}{|y-z|^{\alpha}}dzdy\\
            &\quad+\int_{\Omega_{p}}\int_{\Omega_{p}}\frac{(1+\frac{v_{p}(z)}{p})^{p+1}(1+\frac{v_{p}(y)}{p})^{p}H(\varepsilon_{p}x+x_{p},\varepsilon_{p}y+x_{p})}{|y-z|^{\alpha}}dzdy\\
            &\quad-\int_{\Omega_{p}}\int_{\Omega_{p}}\frac{(1+\frac{v_{p}(z)}{p})^{p+1}(1+\frac{v_{p}(y)}{p})^{p}H(x_{p},\varepsilon_{p}y+x_{p})}{|y-z|^{\alpha}}dzdy.
        \end{aligned}
    \end{equation}
    Since $H$ is Lipchitz continuous, thus
    \begin{equation}
       | H(\varepsilon_{p}x+x_{p},\varepsilon_{p}y+x_{p})-H(x_{p},\varepsilon_{p}y+x_{p})|\leq C.
    \end{equation}
    On the other hand, by Proposition \ref{prop L-infty estimate} and Lemma \ref{lemma for estimate L-0}
    \begin{equation}
    \begin{aligned}
        \int_{\Omega_{p}}\int_{\Omega_{p}}\frac{(1+\frac{v_{p}(z)}{p})^{p+1}(1+\frac{v_{p}(y)}{p})^{p}}{|y-z|^{\alpha}}dzdy&=\frac{p}{u_{p}(x_{p})}\int_{\Omega}\int_{\Omega}\frac{u_{p}^{p+1}(z)u_{p}^{p}(y)}{|y-z|^{\alpha}}dzdy\\
        &\leq C\int_{\Omega}f_{p}(y)dy\leq C.
    \end{aligned}        
    \end{equation}
    Thus
    \begin{equation}\label{lemma-decay-estimate-proof-6}
        \begin{aligned}
             v_{p}(x)
             &=\frac{1}{2\pi}\int_{\Omega_{p}}\log\frac{|y|}{|x-y|}\left(\int_{\Omega_{p}}\frac{(1+\frac{v_{p}(z)}{p})^{p+1}}{|y-z|^{\alpha}}dz\right)(1+\frac{v_{p}(y)}{p})^{p}dy+O(1).
        \end{aligned}
    \end{equation}
Notice that
 \begin{equation}
    \int_{\R^{2}}\int_{\R^{2}}\frac{e^{v(y)}e^{v(x)}}{|x-y|^{\alpha}}dydx=2(4-\alpha)\pi,
 \end{equation}
thus by Proposition \ref{prop convergence of z-p} and Fatou's lemma, for any given $\varepsilon>0$, there exists $R_{\varepsilon}>0$ such that
 \begin{equation}
 \begin{aligned}
      \liminf_{p\to+\infty}\int_{B_{R_{\varepsilon}(0)}}\int_{\Omega_{p}}\frac{(1+\frac{v_{p}(y)}{p})^{p+1}(1+\frac{v_{p}(x)}{p})^{p}}{|x-y|^{\alpha}}dydx
      &\geq \int_{B_{R_{\varepsilon}(0)}}\int_{\R^{2}}\frac{e^{v(y)}e^{v(x)}}{|x-y|^{\alpha}}dydx\\
     &>2(4-\alpha)\pi-\varepsilon/2.
 \end{aligned}    
 \end{equation}
Hence there exists $p_{\varepsilon}>1$ large enough such that
 \begin{equation}\label{eq-decay-estimate for z-p-9}
     \int_{B_{R_{\varepsilon}(0)}}\int_{\Omega_{p}}\frac{(1+\frac{v_{p}(z)}{p})^{p+1}(1+\frac{v_{p}(y)}{p})^{p}}{|y-z|^{\alpha}}dzdy>2(4-\alpha)\pi-\varepsilon.
 \end{equation}
Next, we consider the case $2R_{\varepsilon}\leq|x|\leq\frac{r}{\varepsilon_{p}}$. When $|y|\geq\frac{2r}{\varepsilon_{p}}$, it's easy to see that 
$\frac{2}{3}\leq\frac{|y|}{|x-y|}\leq 2$, which implies 
\begin{equation}
    \begin{aligned}
        &\int_{\{|y|\geq\frac{2r}{\varepsilon_{p}}\}\cap\Omega_{p}}\log\frac{|y|}{|x-y|}\left(\int_{\Omega_{p}}\frac{(1+\frac{v_{p}(z)}{p})^{p+1}}{|y-z|^{\alpha}}dz\right)(1+\frac{v_{p}(y)}{p})^{p}dy\\
        &\leq \log2\int_{\Omega_{p}}\left(\int_{\Omega_{p}}\frac{(1+\frac{v_{p}(z)}{p})^{p+1}}{|y-z|^{\alpha}}dz\right)(1+\frac{v_{p}(y)}{p})^{p}dy\\
        &=\log2\frac{p}{u_{p}(x_{p})}\int_{\Omega}\left(\int_{\Omega}\frac{u_{p}^{p+1}(z)}{|y-z|^{\alpha}}dz\right)u_{p}^{p}(y)dy
        =O(1).
    \end{aligned}
\end{equation}
Then by \eqref{lemma-decay-estimate-proof-6}, we have 
    \begin{equation}\label{eq-estimate for z-p-101}
        \begin{aligned}
             v_{p}(x)
             &=\frac{1}{2\pi}\int_{\{|y|\leq\frac{2r}{\varepsilon_{p}}\}}\log\frac{|y|}{|x-y|}\left(\int_{\Omega_{p}}\frac{(1+\frac{v_{p}(z)}{p})^{p+1}}{|y-z|^{\alpha}}dz\right)(1+\frac{v_{p}(y)}{p})^{p}dy+O(1)\\
             &=\frac{1}{2\pi}\int_{\{|y|\leq R_{\varepsilon}\}}\log\frac{|y|}{|x-y|}\left(\int_{\Omega_{p}}\frac{(1+\frac{v_{p}(z)}{p})^{p+1}}{|y-z|^{\alpha}}dz\right)(1+\frac{v_{p}(y)}{p})^{p}dy\\
             &\quad+\frac{1}{2\pi}\int_{\{R_{\varepsilon}\leq|y|\leq\frac{2r}{\varepsilon_{p}}\}}\log\frac{|y|}{|x-y|}\left(\int_{\Omega_{p}}\frac{(1+\frac{v_{p}(z)}{p})^{p+1}}{|y-z|^{\alpha}}dz\right)(1+\frac{v_{p}(y)}{p})^{p}dy+O(1).\\
        \end{aligned}
    \end{equation}

Now, we need to estimate the terms on the right-hand side of \eqref{eq-estimate for z-p-101}.
Notice that, if $|y|\leq R_{\varepsilon}$, then $|x|\geq 2|y|$ and
$\frac{|y|}{|x-y|}\leq\frac{|y|}{|x|-|y|}\leq\frac{|y|}{|x|-\frac{|x|}{2}}\leq\frac{2R_{\varepsilon}}{|x|}$. Thus the first term on the right-hand side of \eqref{eq-estimate for z-p-101} can be estimated as
\begin{equation}\label{eq-estimate for z-p-102}
    \begin{aligned}
        &\frac{1}{2\pi}\int_{\{|y|\leq R_{\varepsilon}\}}\log\frac{|y|}{|x-y|}\left(\int_{\Omega_{p}}\frac{(1+\frac{v_{p}(z)}{p})^{p+1}}{|y-z|^{\alpha}}dz\right)(1+\frac{v_{p}(y)}{p})^{p}dy\\
        &\leq\frac{1}{2\pi}\log\frac{2R_{\varepsilon}}{|x|}\int_{\{|y|\leq R_{\varepsilon}\}}\left(\int_{\Omega_{p}}\frac{(1+\frac{v_{p}(z)}{p})^{p+1}}{|y-z|^{\alpha}}dz\right)(1+\frac{v_{p}(y)}{p})^{p}dy\\
        &\leq \frac{1}{2\pi}\log\frac{2R_{\varepsilon}}{|x|}(2(4-\alpha)\pi-\varepsilon),
    \end{aligned}
\end{equation}
where we have used that $\frac{2R_{\varepsilon}}{|x|}\leq 1$ and \eqref{eq-decay-estimate for z-p-9}. Next, we divide the second term on the right-hand side of \eqref{eq-estimate for z-p-101} into three parts
\begin{equation}\label{eq-estimate for z-p-103}
    \begin{aligned}
        &\frac{1}{2\pi}\int_{\{R_{\varepsilon}\leq|y|\leq\frac{2r}{\varepsilon_{p}}\}}\log\frac{|y|}{|x-y|}\left(\int_{\Omega_{p}}\frac{(1+\frac{v_{p}(z)}{p})^{p+1}}{|y-z|^{\alpha}}dz\right)(1+\frac{v_{p}(y)}{p})^{p}dy\\
        &=\frac{1}{2\pi}\int_{\{R_{\varepsilon}\leq|y|\leq\frac{2r}{\varepsilon_{p}}\}\cap\{|y|\leq2|x-y|\}}\log\frac{|y|}{|x-y|}\left(\int_{\Omega_{p}}\frac{(1+\frac{v_{p}(z)}{p})^{p+1}}{|y-z|^{\alpha}}dz\right)(1+\frac{v_{p}(y)}{p})^{p}dy\\
        &\quad+\frac{1}{2\pi}\int_{\{R_{\varepsilon}\leq|y|\leq\frac{2r}{\varepsilon_{p}}\}\cap\{|y|\geq2|x-y|\}}\log|y|\left(\int_{\Omega_{p}}\frac{(1+\frac{v_{p}(z)}{p})^{p+1}}{|y-z|^{\alpha}}dz\right)(1+\frac{v_{p}(y)}{p})^{p}dy\\
        &\quad+\frac{1}{2\pi}\int_{\{R_{\varepsilon}\leq|y|\leq\frac{2r}{\varepsilon_{p}}\}\cap\{|y|\geq2|x-y|\}}\log\frac{1}{|x-y|}\left(\int_{\Omega_{p}}\frac{(1+\frac{v_{p}(z)}{p})^{p+1}}{|y-z|^{\alpha}}dz\right)(1+\frac{v_{p}(y)}{p})^{p}dy\\
        &:=\rm{I}+II+III.
    \end{aligned}
\end{equation}
The first term $I$ can be estimated as
\begin{equation}
    \begin{aligned}
        &\rm{I}=\frac{1}{2\pi}\int_{\{R_{\varepsilon}\leq|y|\leq\frac{2r}{\varepsilon_{p}}\}\cap\{|y|\leq2|x-y|\}}\log\frac{|y|}{|x-y|}\left(\int_{\Omega_{p}}\frac{(1+\frac{v_{p}(z)}{p})^{p+1}}{|y-z|^{\alpha}}dz\right)(1+\frac{v_{p}(y)}{p})^{p}dy\\
        &\leq\frac{1}{2\pi}\log2\int_{\{R_{\varepsilon}\leq|y|\leq\frac{2r}{\varepsilon_{p}}\}\cap\{|y|\leq2|x-y|\}}\left(\int_{\Omega_{p}}\frac{(1+\frac{v_{p}(z)}{p})^{p+1}}{|y-z|^{\alpha}}dz\right)(1+\frac{v_{p}(y)}{p})^{p}dy\\
        &=\frac{1}{2\pi}\log2\int_{B_{\frac{r}{\varepsilon_{p}}}(0)}\left(\int_{\Omega_{p}}\frac{(1+\frac{v_{p}(z)}{p})^{p+1}}{|y-z|^{\alpha}}dz\right)(1+\frac{v_{p}(y)}{p})^{p}dy\\
        &\quad-\frac{1}{2\pi}\log2\int_{B_{R_{\varepsilon}}(0)}\left(\int_{\Omega_{p}}\frac{(1+\frac{v_{p}(z)}{p})^{p+1}}{|y-z|^{\alpha}}dz\right)(1+\frac{v_{p}(y)}{p})^{p}dy\\
        &\quad+ \frac{1}{2\pi}\log2\int_{\frac{r}{\varepsilon_{p}}\leq |y|\leq\frac{2r}{\varepsilon_{p}}}\left(\int_{\Omega_{p}}\frac{(1+\frac{v_{p}(z)}{p})^{p+1}}{|y-z|^{\alpha}}dz\right)(1+\frac{v_{p}(y)}{p})^{p}dy.
    \end{aligned}
\end{equation}
Moreover, by Proposition \ref{prop L-infty estimate} and Lemma \ref{lemma-outside-the-non-regular-point}
\begin{equation}
    \begin{aligned}
        &\frac{1}{2\pi}\log2\int_{\frac{r}{\varepsilon_{p}}\leq |y|\leq\frac{2r}{\varepsilon_{p}}}\left(\int_{\Omega_{p}}\frac{(1+\frac{v_{p}(z)}{p})^{p+1}}{|y-z|^{\alpha}}dz\right)(1+\frac{v_{p}(y)}{p})^{p}dy\\
        &=\frac{1}{2\pi}\log2\frac{p}{u_{p}(x_{p})}\int_{\{r\leq|y-x_{p}|\leq 2r\}}\left(\int_{\Omega}\frac{u_{p}^{p+1}(z)}{|y-z|^{\alpha}}dz\right)u_{p}^{p}(y)dy\\
        &\leq \varepsilon,\text{~for any~}p \text{~large}.
    \end{aligned}
\end{equation}
Thus
\begin{equation}\label{eq-estimate for z-p-104}
    \begin{aligned}
         \rm{I}&= \frac{1}{2\pi}\log2\int_{B_{\frac{r}{\varepsilon_{p}}}(0)}\left(\int_{\Omega_{p}}\frac{(1+\frac{v_{p}(z)}{p})^{p+1}}{|y-z|^{\alpha}}dz\right)(1+\frac{v_{p}(y)}{p})^{p}dy\\
        &\quad-\frac{1}{2\pi}\log2\left(\int_{B_{R_{\varepsilon}}(0)}\left(\int_{\Omega_{p}}\frac{(1+\frac{v_{p}(z)}{p})^{p+1}}{|y-z|^{\alpha}}dz\right)(1+\frac{v_{p}(y)}{p})^{p}dy+\varepsilon\right)\\
        &\leq\frac{1}{2\pi}\log 2(\beta_{p}-(2(4-\alpha)\pi-\varepsilon)+\varepsilon),
    \end{aligned}
\end{equation}
where we have used \eqref{definition of beta-p} and \eqref{eq-decay-estimate for z-p-9}.
Next, when $|y|\geq 2|x-y|$ , we have $|y|\leq 2|x|$, thus
\begin{equation}\label{eq-estimate for z-p-105}
    \begin{aligned}
        \rm{II}&=\frac{1}{2\pi}\int_{\{R_{\varepsilon}\leq|y|\leq\frac{2r}{\varepsilon_{p}}\}\cap\{|y|\geq2|x-y|\}}\log|y|\left(\int_{\Omega_{p}}\frac{(1+\frac{v_{p}(z)}{p})^{p+1}}{|y-z|^{\alpha}}dz\right)(1+\frac{v_{p}(y)}{p})^{p}dy\\
        &\leq \frac{1}{2\pi}\log(2|x|)\int_{B_{\frac{r}{\varepsilon_{p}}}(0)}\left(\int_{\Omega_{p}}\frac{(1+\frac{v_{p}(z)}{p})^{p+1}}{|y-z|^{\alpha}}dz\right)(1+\frac{v_{p}(y)}{p})^{p}dy\\
        &\quad-\frac{1}{2\pi}\log(2|x|)\left(\int_{B_{R_{\varepsilon}}(0)}\left(\int_{\Omega_{p}}\frac{(1+\frac{v_{p}(z)}{p})^{p+1}}{|y-z|^{\alpha}}dz\right)(1+\frac{v_{p}(y)}{p})^{p}dy\right)\\
        &\quad+\frac{1}{2\pi}\log(2|x|)\int_{\frac{r}{\varepsilon_{p}}\leq |y|\leq\frac{2r}{\varepsilon_{p}}}\left(\int_{\Omega_{p}}\frac{(1+\frac{v_{p}(z)}{p})^{p+1}}{|y-z|^{\alpha}}dz\right)(1+\frac{v_{p}(y)}{p})^{p}dy\\
        &\leq\frac{1}{2\pi}\log (2|x|)(\beta_{p}-(2(4-\alpha)\pi-\varepsilon)+\varepsilon).
    \end{aligned}
\end{equation}
Finally, using Lemma \ref{lemma uniform estimate} and the fact $\log t\leq 0$ if $t\leq 1$
\begin{equation}\label{eq-estimate for z-p-106}
    \begin{aligned}
        \rm{III}&=\frac{1}{2\pi}\int_{\{R_{\varepsilon}\leq|y|\leq\frac{2r}{\varepsilon_{p}}\}\cap\{|y|\geq2|x-y|\}}\log\frac{1}{|x-y|}\left(\int_{\Omega_{p}}\frac{(1+\frac{v_{p}(z)}{p})^{p+1}}{|y-z|^{\alpha}}dz\right)(1+\frac{v_{p}(y)}{p})^{p}dy\\
        &=\frac{1}{2\pi}\int_{\{R_{\varepsilon}\leq|y|\leq\frac{2r}{\varepsilon_{p}}\}\cap\{|x-y|\leq 1\}}\log\frac{1}{|x-y|}\left(\int_{\Omega_{p}}\frac{(1+\frac{v_{p}(z)}{p})^{p+1}}{|y-z|^{\alpha}}dz\right)(1+\frac{v_{p}(y)}{p})^{p}dy\\
        &\quad+\frac{1}{2\pi}\int_{\{R_{\varepsilon}\leq|y|\leq\frac{2r}{\varepsilon_{p}}\}\cap\{2\leq2|x-y|\leq |y|\}}\log\frac{1}{|x-y|}\left(\int_{\Omega_{p}}\frac{(1+\frac{v_{p}(z)}{p})^{p+1}}{|y-z|^{\alpha}}dz\right)(1+\frac{v_{p}(y)}{p})^{p}dy\\
        &\leq C \frac{1}{2\pi}\int_{\{|x-y|\leq 1\}}\log\frac{1}{|x-y|}dy=O(1).
    \end{aligned}
\end{equation}
Combining estimates \eqref{eq-estimate for z-p-101}-\eqref{eq-estimate for z-p-106} with Lemma \ref{lemma-convergence-to-right-hand-integral}, we obtain
\begin{equation}
    \begin{aligned}
        v_{p}(x) &\leq  \frac{1}{2\pi}\log\frac{2R_{\varepsilon}}{|x|}(2(4-\alpha)\pi-\varepsilon)+\frac{1}{2\pi}\log 2(\beta_{p}-(2(4-\alpha)\pi-\varepsilon)+\varepsilon)\\
        &\quad+\frac{1}{2\pi}\log (2|x|)(\beta_{p}-(2(4-\alpha)\pi-\varepsilon)+\varepsilon)+O(1)\\
        &=\frac{1}{2\pi}\log\frac{1}{|x|}(\beta_{p}-2(\beta_{p}-2(4-\alpha)\pi)-3\varepsilon)+C_{\varepsilon}\\
        &\leq \left(\frac{\beta_{p}}{2\pi}-\varepsilon\right)\log\frac{1}{|x|}+C_{\varepsilon},
    \end{aligned}
\end{equation}
for any $2R_{\varepsilon}\leq |x|\leq\frac{r}{\varepsilon_{p}}$ and $p\geq p_{\varepsilon}$. This completes the proof.
\end{proof}

\begin{Prop}\label{prop-4.15} It holds that $\lim_{p\to+\infty}u_{p}(x_{p})\to\sqrt{e}$.
\end{Prop}
\begin{proof}
Notice that by Lemma \ref{lemma-outside-the-non-regular-point} and the HLS inequality, we have
\begin{equation}
    \begin{aligned}
        &\int_{\Omega_{p}}\int_{\Omega_{p}}\frac{\left(1+\frac{v_{p}(y)}{p}\right)^{p+1}\left(1+\frac{v_{p}(x)}{p}\right)^{p+1}}{|x-y|^{\alpha}}dydx\\
        &=\frac{p}{u_{p}(x_{p})^{2}}\int_{\Omega}\int_{\Omega}\frac{u_{p}^{p+1}(y)u_{p}^{p+1}(x)}{|x-y|^{\alpha}}dydx\\
        &=\frac{p}{u_{p}(x_{p})^{2}}\int_{B_{2r}(x_{0})}\int_{B_{2r}(x_{0})}\frac{u_{p}^{p+1}(y)u_{p}^{p+1}(x)}{|x-y|^{\alpha}}dydx\\
        &\quad+\frac{p}{u_{p}(x_{p})^{2}}\int_{B_{2r}(x_{0})}\int_{\Omega\setminus B_{2r}(x_{0})}\frac{u_{p}^{p+1}(y)u_{p}^{p+1}(x)}{|x-y|^{\alpha}}dydx\\
        &\quad+\frac{p}{u_{p}(x_{p})^{2}}\int_{\Omega\setminus B_{2r}(x_{0})}\int_{\Omega}\frac{u_{p}^{p+1}(y)u_{p}^{p+1}(x)}{|x-y|^{\alpha}}dydx\\      
        &=\frac{p}{u_{p}(x_{p})^{2}}\int_{B_{2r}(x_{0})}\int_{B_{2r}(x_{0})}\frac{u_{p}^{p+1}(y)u_{p}^{p+1}(x)}{|x-y|^{\alpha}}dydx+o_{p}(1).
    \end{aligned}
\end{equation}
Moreover, since $B_{\frac{r}{2}}(x_{0})\subset B_{r}(x_{p})\subset B_{2r}(x_{0})$ for any $p$ large, then 
\begin{equation}
    \begin{aligned}
    &\frac{p}{u_{p}(x_{p})^{2}}\int_{B_{2r}(x_{0})}\int_{B_{2r}(x_{0})}\frac{u_{p}^{p+1}(y)u_{p}^{p+1}(x)}{|x-y|^{\alpha}}dydx\\
        &=\frac{p}{u_{p}(x_{p})^{2}}\int_{B_{r}(x_{p})}\int_{B_{r}(x_{p})}\frac{u_{p}^{p+1}(y)u_{p}^{p+1}(x)}{|x-y|^{\alpha}}dydx\\
        &\quad+\frac{p}{u_{p}(x_{p})^{2}}\int_{B_{2r}(x_{0})\setminus B_{r}(x_{p})}\int_{B_{r}(x_{p})}\frac{u_{p}^{p+1}(y)u_{p}^{p+1}(x)}{|x-y|^{\alpha}}dydx\\
        &\quad+\frac{p}{u_{p}(x_{p})^{2}}\int_{B_{2r}(x_{0})}\int_{B_{2r}(x_{0})\setminus B_{r}(x_{p})}\frac{u_{p}^{p+1}(y)u_{p}^{p+1}(x)}{|x-y|^{\alpha}}dydx\\
        &=\frac{p}{u_{p}(x_{p})^{2}}\int_{B_{r}(x_{p})}\int_{B_{r}(x_{p})}\frac{u_{p}^{p+1}(y)u_{p}^{p+1}(x)}{|x-y|^{\alpha}}dydx+o_{p}(1)\\
        &=\int_{B_{\frac{r}{\varepsilon_{p}}(0)}}\int_{B_{\frac{r}{\varepsilon_{p}}(0)}}\frac{\left(1+\frac{v_{p}(y)}{p}\right)^{p+1}\left(1+\frac{v_{p}(x)}{p}\right)^{p+1}}{|x-y|^{\alpha}}dydx+o_{p}(1).
    \end{aligned}
\end{equation}
Thus
\begin{equation}
    \begin{aligned}
        &\int_{B_{\frac{r}{\varepsilon_{p}}(0)}}\int_{B_{\frac{r}{\varepsilon_{p}}(0)}}\frac{\left(1+\frac{v_{p}(y)}{p}\right)^{p+1}\left(1+\frac{v_{p}(x)}{p}\right)^{p+1}}{|x-y|^{\alpha}}dydx\\
        &=\int_{\Omega_{p}}\int_{\Omega_{p}}\frac{\left(1+\frac{v_{p}(y)}{p}\right)^{p+1}\left(1+\frac{v_{p}(x)}{p}\right)^{p+1}}{|x-y|^{\alpha}}dydx+o_{p}(1).
    \end{aligned}
\end{equation}
On the other hand, by the definition of $v_{p}$ and Proposition \ref{lemma-decay-estimate}, we have
\begin{equation}
    \begin{aligned}
        0\leq \left(1+\frac{v_{p}(x)}{p}\right)^{p+1}\leq 1\text{~~for any~}|x|\leq2 R_{\varepsilon}
    \end{aligned}
\end{equation}
and
\begin{equation}
    0\leq \left(1+\frac{v_{p}(x)}{p}\right)^{p+1}\leq e^{(p+1)\log(1+\frac{v_{p}(x)}{p})}\leq e^{v_{p}(x)}\leq \frac{C}{|x|^{\frac{\beta_{p}-\varepsilon}{2\pi}}}\text{~~for any~}2R_{\varepsilon}\leq |x|\leq\frac{r}{\varepsilon_{p}}.
\end{equation}
Hence, by Lemma \ref{lemma-convergence-to-right-hand-integral}, there exists $C>0$ such that for any $p$ large enough
\begin{equation}
    \begin{aligned}
         0\leq \left(1+\frac{v_{p}(x)}{p}\right)^{p+1}\leq\frac{C}{1+|x|^{3-\alpha}}\text{~~for any~}0\leq |x|\leq\frac{r}{\varepsilon_{p}}.
    \end{aligned}
\end{equation}
Finally, using Proposition \ref{estimate of S-p}, Proposition \ref{prop convergence of z-p} and the Dominated Convergence Theorem, we have
    \begin{equation}
    \begin{aligned}
        2(4-\alpha)\pi&=\int_{\R^{2}}\int_{\R^{2}}\frac{e^{v(y)}e^{v(x)}}{|x-y|^{\alpha}}dydx\\
        &= \lim_{p\to+\infty}\int_{B_{\frac{r}{\varepsilon_{p}}(0)}}\int_{B_{\frac{r}{\varepsilon_{p}}(0)}}\frac{\left(1+\frac{v_{p}(y)}{p}\right)^{p+1}\left(1+\frac{v_{p}(x)}{p}\right)^{p+1}}{|x-y|^{\alpha}}dydx\\
        &= \lim_{p\to+\infty} \int_{\Omega_{p}}\int_{\Omega_{p}}\frac{\left(1+\frac{v_{p}(y)}{p}\right)^{p+1}(y)\left(1+\frac{v_{p}(x)}{p}\right)^{p+1}(x)}{|x-y|^{\alpha}}dydx\\
        &=\lim_{p\to+\infty} \frac{pS_{p}^{\frac{2(p+1)}{p}}}{u_{p}^{2}(x_{p})}=\frac{2(4-\alpha)\pi e}{\lim_{p\to+\infty}u_{p}^{2}(x_{p})}.
    \end{aligned}
\end{equation}
It follows that $\lim_{p\to+\infty}u_{p}(x_{p})=\sqrt{e}$.
\end{proof}

\begin{Prop}\label{prop-4.16} It holds that
\begin{enumerate}[label=\upshape(\arabic*)]

    \item $\lim_{p\to+\infty}pu_{p}(x)=2(4-\alpha)\pi\sqrt{e}~G(x,x_{0})\text{~in~}C^{2}_{loc}(\bar{\Omega}\setminus\{x_{0}\})$.
    \item  $f_{p}=p\left(\int_{\Omega}\frac{u_{p}^{p+1}(y)}{|x-y|^{\alpha}}dy\right)u_{p}^{p}(x)\to (2(4-\alpha)\pi\sqrt{e})\delta_{x_{0}},$
    in the sense of distribution, where $\delta_{x_{0}}$ is the Dirac delta function at point $x_{0}$.
    \item $x_{0}$ is a critical point of the Robin function $R(x)$. In particular, if $\Omega$ is a convex domain, then $x_0$ is the global minimum point of the Robin function.
\end{enumerate}
   
\end{Prop}
\begin{proof}
By Lemma \ref{lemma-convergence-to-right-hand-integral} and Proposition \ref{prop-4.15}, we obtain
\begin{equation}\label{eq-gamma-limit}
    \gamma = \lim_{\delta\to 0}\lim_{p\to +\infty} \int_{B_{\delta}(x_{0})} f_{p}(x) \, dx 
    = \lim_{\delta\to 0}\lim_{p\to +\infty} p \int_{B_{\delta}(x_{0})} \int_{\Omega} \frac{u_{p}^{p+1}(y) u_{p}^{p}(x)}{|x-y|^{\alpha}} \, dy dx 
    = 2(4-\alpha)\pi\sqrt{e},
\end{equation}
which proves (1) in combination with Lemma \ref{lemma-convergence-out-of-the-blow-up-point}. Next, Lemma \ref{lemma-outside-the-non-regular-point} implies that $f_{p}\to0$ uniformly on  any compact set $K\subset \bar{\Omega}\setminus\{x_{0}\}$. Taking $\varphi\in C_{c}(\Omega)$, we have
\begin{equation}
    \begin{aligned}
        \int_{\Omega}f_{p}(x)\varphi(x)dx&=\int_{B_{\delta}(x_{0})}f_{p}(x)\varphi(x)dx+\int_{\Omega\setminus B_{\delta}(x_{0})}f_{p}(x)\varphi(x)dx\\
        &=\varphi(x_{0})\int_{B_{\delta}(x_{0})}f_{p}(x)dx+\int_{B_{\delta}(x_{0})}f_{p}(x)(\varphi(x)-\varphi(x_{0}))dx+o_{p}(1)\\
        &=\varphi(x_{0})\int_{B_{\delta}(x_{0})}f_{p}(x)dx+o_{\delta}(1)+o_{p}(1),
    \end{aligned}
\end{equation}
and thus
\begin{equation}
\begin{aligned}
     \lim_{p\to+\infty} \int_{\Omega}f_{p}(x)\varphi(x)dx=\varphi(x_{0})\lim_{\delta\to 0}\lim_{p\to +\infty} \int_{B_{\delta}(x_{0})} f_{p}(x) \, dx= (2(4-\alpha)\pi\sqrt{e})\varphi(x_{0}).
\end{aligned}
\end{equation}
This proves part (2). Finally, applying the Pohozaev identity \eqref{Pohozaev 3} yields
\begin{equation}\label{eq-pohozaev-boundary}
    \int_{\partial \Omega} \nu(x) \left(\frac{\partial u_{p}}{\partial \nu}\right)^2 d\sigma_{x} = 0.
\end{equation}
Combining (1) with \cite[Lemma 5.1]{Ren1994TAMS}, we conclude
\begin{equation}
    \int_{\partial \Omega} \nu(x) \left(\frac{\partial G(x,x_{0})}{\partial \nu}\right)^2 d\sigma_{x} 
    = -\nabla R(x_{0}) = 0.
\end{equation}
If $\Omega$ is convex, then the Robin function $R(x)$ is strictly convex with a strictly positive definite Hessian, as shown by Caffarelli and Friedman \cite[Theorem 3.1]{Caffarelli1985Convexity}.  Consequently, $R(x)$ has a unique critical point, which must be its global minimum. This completes the proof.
\end{proof}
\begin{proof}[Proof of Theorem \ref{thm-2}]
Theorem \ref{thm-2} follows immediately from Proposition \ref{prop-4.15} and Proposition \ref{prop-4.16}.
\end{proof}

\bibliographystyle{abbrv}

\end{document}